\documentclass[10pt]{amsart}

\usepackage{hyperref}
\usepackage{enumerate}
\usepackage{comment}

\makeatletter

\@namedef{subjclassname@2010}{

  \textup{2010} Mathematics Subject Classification}

\usepackage{amssymb, amsmath,amsthm}
\usepackage{mathrsfs}
\newtheorem{thm}{Theorem}[section]

\newtheorem{prop}[thm]{Proposition}
\newtheorem{conj}[thm]{Conjecture}

\newtheorem{cor}[thm]{Corollary}
\newtheorem{lem}[thm]{Lemma}
\theoremstyle{definition}
\newtheorem{rem}[thm]{Remark}
\newtheorem{def1}[thm]{Definition}
\newtheorem{ex}[thm]{Example}
\newtheorem{prob}[thm]{Problem}

\newcommand{\ra}{\rightarrow}
\newcommand{\bk}{\backslash}
\newcommand{\mc}{\mathcal}
\newcommand{\mf}{\mathfrak}
\newcommand{\mb}{\mathbb}

\renewcommand{\ss}{\substack}

\newcommand{\e}{\varepsilon}

\renewcommand{\bar}{\overline}

\begin{document}
\title{Gap problems for integer-valued multiplicative functions}
\author{Alexander P. Mangerel}
\address{Department of Mathematical Sciences, Durham University, Stockton Road, Durham, DH1 3LE, UK}
\email{smangerel@gmail.com}
\maketitle

\begin{abstract}
Motivated by questions about the typical sizes of gaps $|f(n+1)-f(n)|$ in the sequence $(f(n))_n$, where $f$ is an integer-valued multiplicative function, we investigate the set of solutions
$$
\{n \in \mb{N} : f(n+a) = f(n) + b\}, \quad ab \neq 0.
$$
We formulate a conjecture classifying those multiplicative functions for which this set has logarithmic density zero, and prove that the conjectured classification is tight. Moreover, using techniques from additive combinatorics building on previous work of the author, we show how to reduce the classification problem to the study of \emph{local power maps} modulo prime $\ell$, i.e., maps $g: \mb{N} \ra \mb{Z}$ for which there is $0 \leq k_{\ell} < \ell-1$ such that
$$
g(n) \equiv n^{k_{\ell}} \pmod{\ell} \text{ for all } n \in \mb{N}.
$$
We prove a partial result towards our classification conjecture by employing a strategy of N. Jones that uses Kummer theory to study local power maps modulo many primes $\ell$.
\end{abstract}

\section{Introduction and Main Results}
\subsection{Motivation} 
Let $f: \mb{N} \ra \mb{N}$ be a multiplicative function that takes positive integer values. It is natural to speculate about the growth and distribution of the sequence of gaps $(|f(n+1)-f(n)|)_n$. Heuristics about the independence of multiplicative and additive structure in integers suggests that \emph{unless $f$ has very rigid properties} the following should hold true: 
\begin{equation} \label{eq:growingGap}
\text{if $f(n) \ra \infty$ as $n \ra \infty$ a.e.}\footnote{Here, a statement that occurs as $n \ra \infty$ a.e., i.e., almost everywhere, means that $n \ra \infty$ along a set of natural density $1$.} \text{ then
$|f(m+1)-f(m)| \ra \infty$ as $m \ra \infty$ a.e.}
\end{equation}
Conversely, if $|f(n+1)-f(n)|$ is bounded on a set of positive upper density then we would expect that either: 
\begin{enumerate}[(i)]
\item $f$ is bounded on a set of positive upper density, or 
\item $f$ exhibits \emph{rigid} behaviour (to be made more precise shortly). 
\end{enumerate}
Problems relating to the sequence of gaps $|f(n+1)-f(n)|$ of a multiplicative function $f: \mb{N} \ra \mb{C}$ has inspired many interesting previous works in the literature. For example, K\'{a}tai \cite{Kat2} conjectured that if $|f(n+1)-f(n)| \ra 0$ as $n \ra \infty$ then $f(n) = n^{it}$ for some $t \in \mb{R}$, a result that was proven by Wirsing (see the joint paper \cite{Wir}). On the basis of a conjecture of Erd\H{o}s \cite{Erd} about the average gaps in the sequence of an additive function, K\'{a}tai \cite{Kat2} also conjectured that if, as $x \ra \infty$,
$$
\frac{1}{x}\sum_{n \leq x} |f(n)| \gg 1 \text{ and } \frac{1}{x} \sum_{n \leq x} |f(n+1)-f(n)| = o(1)
$$ 
then $f(n) = n^{s}$ for some $s \in \mb{C}$ with $\text{Re}(s) < 1$, a result that was more recently proved by Klurman \cite{Klu}. In the other direction, Klurman and the author \cite{Rig}, addressing a folklore conjecture, proved that if $f: \mb{N} \ra \mb{C}$ is completely multiplicative and takes values on the unit circle then
$$
\liminf_{n \ra \infty} |f(n+1)-f(n)| = 0.
$$
In contrast to K\'{a}tai's problems mentioned above, which address the ``99\%'' problem of when the gaps are \emph{small for almost all} $n$, our interest in this paper is to understand the more challenging ``1\%'' variant of this problem, namely when the gaps are \emph{bounded on a positive upper density set}. \\
It is certainly true that there are multiplicative functions for which \eqref{eq:growingGap} does not hold. The simplest such example is to take the monomial $f(n) = n$, so that $|f(n+1)-f(n)| = 1$ for all $n$. This example can be generalised readily in various ways. 
\begin{ex} \label{ex:bddEx}
Let $S$ be a set of odd primes that is \emph{sparse} in the sense that
$$
\sum_{p \in S} \frac{1}{p} < \infty.
$$
Let $p_1,p_2 \notin S$ be fixed primes and let $b \in \mb{N}$. Define a completely multiplicative function $f$ at primes $p$ by
$$
f(p) := \begin{cases} p &\text{ if $p \notin S \cup \{p_1,p_2\}$} \\ bp &\text{ if $p \in \{p_1,p_2\}$} \\ 1 &\text{ if $p \in S$.} \end{cases}
$$
Then whenever $n$ is chosen so that $p_1||n$, $p_2||(n+1)$ and so that $n(n+1)$ has no prime factors from $S$, then $f(n+1) = f(n) + b$. Indeed,
$$
f(n) = bp_1 \cdot \frac{n}{p_1} = bn, \quad f(n+1) = bp_2 \cdot \frac{n+1}{p_2} = bn+b,
$$
Using a zero-dimensional sieve it can easily be shown that this occurs for a set of $n \in \mb{N}$ with natural density
$$
\frac{1}{p_1p_2}\left(1-\frac{1}{p_1}\right)\left(1-\frac{1}{p_2}\right)\prod_{p \in S} \left(1-\frac{2}{p}\right) > 0.
$$
\end{ex}
\noindent Such examples are all morally similar in that $p|f(p)$ outside of a sparse set of primes, and that for generic primes $p$ we more precisely have $f(p) = p$, as in the . \\
As an application of the more general results in this paper (discussed in the next subsection), we have the following consequence. In the statement, given a set $A \subseteq \mb{N}$ we write
$$
dA := \lim_{x \ra \infty} \frac{|A \cap [1,x]|}{x}, \quad \delta A := \lim_{x \ra \infty} \frac{1}{\log x}\sum_{n \in A \cap [1,x]} \frac{1}{n},
$$
whenever the limits exist.
\begin{cor}\label{cor:growingPos}
Assume the\footnote{By the ERH we mean that for each number field $K/\mb{Q}$ the Dedekind zeta function $\zeta_K$ of $K$ has no zeros in the half-plane $\text{Re}(s) > 1/2$. In fact, we do not need the full strength of this hypothesis; see Remark \ref{rem:ERHReq} for a relevant discussion of this.} Extended Riemann Hypothesis. Let $f: \mb{N} \ra \mb{N}$ be a multiplicative function, such that
$$
d\{ n \in \mb{N} : |f(n)| \leq C\} = 0 \text{ for all } C > 0.
$$
Then either
$$
\delta\{n \in \mb{N} : |f(n+1)-f(n)| \leq C\} = 0 \text{ for all } C > 0,
$$
or else precisely one of
\begin{equation}\label{eq:primeSumConds}
 \sum_{p|f(p)} \frac{1}{p} \text{ and }  \sum_{f(p) \neq p} \frac{1}{p}
\end{equation}
converges.
\end{cor}
\noindent We will deduce Corollary \ref{cor:growingPos} from Theorem \ref{thm:Inhom} below, in Section \ref{sec:GapCors}.
\begin{rem} \label{rem:condsNec}
We expect that the first alternative still holds even when $\sum_{p|f(p)} 1/p < \infty$ in \eqref{eq:primeSumConds}, but are unable to prove it at present. We leave this as an open problem to the interested reader. 
\end{rem}
\subsection{Inhomogeneous problems relating shifted values of $f$}
In proving Corollary \ref{cor:growingPos} we are naturally led to the following general problem. 
\begin{prob}\label{prob:classify}
Let $a,b,A, B \in \mb{Z}$ with $aAB \neq 0$. Classify those multiplicative functions $f: \mb{N} \ra \mb{C}$ for which the set
\begin{equation}\label{eq:NfabAB}
\mc{N}_{f,a,b;A,B} := \{n \in \mb{N} : A f(n+a) = Bf(n) + b\}
\end{equation}
satisfies\footnote{This statement should be understood as including the possibility that $\delta \mc{N}_{f,a,b;A,B}$ does not exist.} $\delta \mc{N}_{f,a,b;A,B} \neq 0$.
\end{prob}
In this paper we are especially concerned with the case where $A = B = 1$ and $ b\neq 0$. Our work extends the classification theorems for complex-valued multiplicative functions that was undertaken in \cite{ManConsec}, where we treated the case $b = 0$ and arbitrary non-zero $A,B$. The main result of that paper, \cite[Thm. 1.1]{ManConsec}, showed (in slightly more generality) that for any multiplicative function $f: \mb{N} \ra \mb{C}$ \emph{at least one} of the following events must occur:
\begin{enumerate}[(i)]
\item the set of primes $\{p : |f(p)| \neq 1\}$ is \emph{sparse} in the sense that
$$
\sum_{p : |f(p)| \neq 1} \frac{1}{p} < \infty;
$$
\item for any non-zero integers $a, A,B$ the set
$$
\mc{N}_{f,a;A,B}' := \{n \in \mb{N} : Af(n+a) = Bf(n) \neq 0\}
$$
satisfies $\delta \mc{N}_{f,a;A,B}' = 0$.
\end{enumerate}
It is not hard to show (see Lemma \ref{lem:fnGrows} below) that (i) occurs whenever $|f(n)|$ is bounded on a set of positive natural density, and thus $|f(n)| \to \infty$ \emph{does not} hold a.e. In view of the discussion in the previous section such functions will be of little interest to us, and in the sequel we will on occasion refer to these examples as \emph{essentially bounded}. 
Thus, \cite[Thm. 1.1]{ManConsec} shows that if $f: \mb{N} \ra \mb{C}$ is multiplicative and \emph{not} essentially bounded, $\delta\mc{N}_{f,a;A,B}' = 0$ whenever $aAB \neq 0$. We focus in this paper on characterising the set $\mc{N}_{f,a,b;A,B}$ when $b \neq 0$ and $f$ takes \emph{integer values}. In contrast to the case $b = 0$ the non-vanishing condition $f(n)f(n+a) \neq 0$ may be removed. \\
In the remainder of the paper we will focus on the case $A = B = 1$, though with only minor modifications the same techniques may be applied to study the general case $AB \neq 0$. In the sequel, therefore, given $a,b \in \mb{Z} \bk \{0\}$ and an arithmetic function $f: \mb{N} \ra \mb{Z}$ we abbreviate for convenience
$$
\mc{N}_{f,a,b} := \mc{N}_{f,a,b;1,1} =  \{n \in \mb{N} : f(n+a) = f(n) + b\}.
$$
\subsection{Classifying examples for the inhomogeneous problem}
Extending Example \ref{ex:bddEx}, it is not difficult to find $f: \mb{N} \ra \mb{Z}$ such that $f(n+a) = f(n) + b$, whenever $a|b$. Indeed, let $a, b \in \mb{Z} \bk \{0\}$ with $a|b$. Let $p_1,p_2$ be distinct primes with $(p_1p_2,a) = 1$, and let $S$ be a sparse set of odd primes not including $p_1,p_2$. Then the completely multiplicative function $f:\mb{N} \ra \mb{Z}$ defined at primes via
$$
f(p) = \begin{cases} p &\text{ if } p \notin S \cup \{ p_1,p_2\} \\ p b/a &\text{ if } p \in \{p_1,p_2\} \\ 1 &\text{ if } p \in S, \end{cases}
$$
then $f(n+a) = f(n) + b$ whenever\footnote{In the sequel, given a set of primes $T$ we write $(m,T) = 1$ to mean that $p|m \Rightarrow p \notin T$.}
$$
n \in \{am \in \mb{N} : p_1 || m, \, p_2||(m+1), (n(n+a),S) = 1\}.
$$ 
Once again, this set may be shown to have positive density 
$$
\frac{1}{ap_1p_2}(1-1/p_1)(1-1/p_2) \prod_{p \in S} \left(1-\frac{2}{p}\right).
$$ 
The classification problem for the inhomogeneous case therefore admits many possible examples. The following conjecture gives a precise description of all of these.
\begin{conj} \label{conj:Inhom}
Let $f: \mb{N} \ra \mb{Z}$ be a multiplicative function that satisfies
\begin{equation} \label{eq:fpnonSparse}
\sum_{p: |f(p)| \neq 1} \frac{1}{p} = \infty.
\end{equation}
Suppose $\delta(\mc{N}_{f,a,b}) \neq 0$ for some non-zero $a,b \in \mb{Z}$.
Then the set $
\mc{E}_{f} := \{p : |f(p)| \neq p\}$ satisfies
$$
\sum_{p \in \mc{E}_f} \frac{1}{p} < \infty.
$$
Moreover, if $f$ is completely multiplicative then there is a divisor $d$ of $a$ such that 
$$
f(d) \mid b \text{ and }(a/d) \mid (b/f(d)).
$$
\end{conj}

\subsection{Partial result towards Conjecture \ref{conj:Inhom}}
We give evidence towards Conjecture \ref{conj:Inhom} by proving that its conclusion holds for a non-trivial collection of multiplicative functions, which we motivate presently. \\
According to Conjecture \ref{conj:Inhom}, any integer-valued multiplicative function $f$ which is not essentially bounded and for which $\delta(\mc{N}_{f,a,b}) \neq 0$ for some $a,b \in \mb{Z}$ non-zero must necessarily satisfy $p|f(p)$ outside of a sparse set of primes $p$. It turns out that the condition that $p|f(p)$ for ``sufficiently many'' $p$ is also sufficient to showing that $f$ satisfies Conjecture \ref{conj:Inhom}. \\
Let $\mc{P}$ denote the set of all primes. To state our result, recall that the \emph{Dirichlet density} of a set of primes $\mc{S} \subseteq \mc{P}$, if it exists, is defined as
$$
d_{\text{Dir}}(\mc{S}) := \lim_{X \ra \infty} \frac{1}{\log\log X}\sum_{\ss{p \leq X \\ p \in \mc{S}}} \frac{1}{p} > 0.
$$
\begin{thm}\label{thm:Inhom}
Let $f: \mb{N} \ra \mb{Z}$ be a multiplicative function.
Assume one of the following conditions\footnote{It is easily seen that if either (a) or (b) holds then so does \eqref{eq:fpnonSparse}.} holds: 
\begin{enumerate}[(a)]
\item $d_{\text{Dir}}(\{p \in \mc{P} : p | f(p)\}) > 0$, or
\item the Extended Riemann Hypothesis holds, and $\sum_{p|f(p)} p^{-1} = \infty$.
\end{enumerate}
Suppose that $\delta(\mc{N}_{f,a,b}) \neq 0$ for some pair of $a,b \in \mb{Z}$ non-zero.
Then the set $
\mc{E}_{f} := \{p : |f(p)| \neq p\}$ satisfies
$$
\sum_{p \in \mc{E}_f} \frac{1}{p} < \infty.
$$
Moreover, if $f$ is completely multiplicative then there is a divisor $d$ of $a$ such that 
$$
f(d) \mid b \text{ and }(a/d) \mid (b/f(d)).
$$
\end{thm}
\begin{rem}\label{rem:ERHReq}
In fact, our arguments yield something slightly stronger than this. Define
$$
S_f := \{p \in \mc{P} : p|f(p)\}.
$$
Then the unconditional conclusion of Theorem \ref{thm:Inhom} holds as long as 
$$
\sum_{\ss{p \leq X \\ p \in S_f}} \frac{1}{p} \gg_{\e} (\log\log X)^{1/2+\e}, \quad X \ra \infty.
$$
Moreover, we do not need the full strength of the Extended Riemann Hypothesis in our conditional result either. Indeed, it suffices for a version of RH to hold for the Dedekind zeta functions of an infinite family of Kummer extensions. Moreover, a narrow zero-free region of Vinogradov-Korobov type would also suffice. See Section \ref{subsec:Prop63} below for relevant details, and for remarks regarding the quality of sufficient zero-free regions for this problem. 
See also Remark \ref{rem:SfSizeSpec} below for a discussion related to whether considerations about $S_f$ are necessary in this problem.
\end{rem}
\noindent 
\subsection{A converse result}
At first glance the conclusion of Conjecture \ref{conj:Inhom} appears rather flexible, and one might ask whether it could possibly be strengthened. We prove, however, that the classification proposed in Conjecture \ref{conj:Inhom} is tight. In particular:
\begin{enumerate}[(i)]
\item the exceptional set $\mc{E}_f$ can essentially be arbitrary, 
\item the signs $f(p)/p$ for $p \notin \mc{E}_f$ are not significantly restricted, and
\item the divisibility conditions given on $a$ and $b$ are necessary.
\end{enumerate}
\begin{prop} \label{prop:converse}
Let $a \in \mb{Z}$ be non-zero, let $d|a$ and let $b \neq a/d$ be an integer multiple of $a/d$. Let $\mc{S}$ be a set of odd primes that satisfies
$$
\{p \in \mc{P} : p|d\} \subseteq \mc{S}, \quad |\mc{S} \bk \{p \in \mc{P} :p|d\}| \geq 2 \text{ and } \sum_{p \in \mc{S}} \frac{1}{p} < \infty.
$$
Then there are uncountably many sets $\mc{T} \subseteq \mc{P} \bk \mc{S}$ with $\sum_{p \in \mc{T}}p^{-1} = \infty$  
for which a completely multiplicative function
$f = f_{\mc{T}} : \mb{N} \ra \mb{Z}$ may be constructed with the following properties:
\begin{enumerate}[(i)]
\item for each $p \notin \mc{S}$,
$$
f(p) = \begin{cases} -p &\text{ if } p \in \mc{T} \\ p &\text{ if } p \notin \mc{T}, \end{cases}
$$
\item for each $p \in \mc{S}$,  $|f(p)| \neq p$, 
\item $f(d)|b$ and $(a/d)|(b/f(d))$, and
\item $\delta(\mc{N}_{f,a,b}) \neq 0$.
\end{enumerate}
When $\mc{T} = \emptyset$ there also exists a corresponding function $f = f_{\emptyset}$ that satisfies the properties (i)-(iv).
\end{prop}
\begin{rem}
We highlight in passing the following relationship that our results bear with Diophantine equations. Observe first of all that for any $d \geq 1$ the function $f^d$ is integer-valued and multiplicative whenever $f$ is. We may thus view $\mc{N}_{f^d,a;A,B}'$ as the locus of non-zero solutions in $(X,Y) = (f(n),f(n+a))$ to the diagonal binary homogeneous equation
$$
AX^d - BY^d = 0, \quad XY \neq 0.
$$
More generally, if $b \neq 0$ we may consider the corresponding\footnote{Naturally, one might also allow the powers of $X$ and $Y$ here to differ. Similar analysis applies in that case (treating two different multiplicative functions $f^k,f^\ell$), so we confine ourselves here to the case that they have the same power.} \emph{inhomogeneous} binary problems
$$
AX^d - BY^d = b
$$
(here the condition $XY \neq 0$ may be removed). 
In this direction, Problem \ref{prob:classify} addresses the classification of multiplicative functions $f: \mb{N} \ra \mb{Z}$ for which the corresponding locus of solutions
$\mc{N}_{f^d,a,b;A,B}$ satisfies $\delta \mc{N}_{f^d,a,b;A,B} \neq 0$.
Thus, among integer solutions $(X,Y)$ to an inhomogeneous binary problem we seek to understand how frequently such solutions arise in the form $(f(n),f(n+a))$, for some integer $n$. (We focus here on $d = 1$, since the problem for larger $d$ may be studied in the same way, replacing $f$ by $f^d$ as above.)
\end{rem}
\section*{Acknowledgments} 
\noindent We would like to thank Oleksiy Klurman for a stimulating conversation about the contents of this paper, and for his encouragement.
\section*{Notation and Conventions}
\noindent Throughout the paper, the letters $a,b,d,m,n,k$ denote integers, and $p,\ell$ are reserved for primes. \\
We write $\mc{P}$ to denote the set of all primes.\\
Given $d \geq 1$ we denote by $\mu_d$ the set of $d$th order roots of unity. We also write
$$
\mb{U} := \{z \in \mb{C} : |z| \leq 1\}
$$
to denote the closed unit disc in the complex plane. \\
Given a set $\mc{A}$ of primes and $n \in \mb{N}$ we write $(n,\mc{A}) = 1$ whenever $p \nmid n$ for all $p \in \mc{A}$. \\
Given a prime $\ell$ and an integer $m$ coprime to $\ell$, we write $\text{ord}_{\ell}(m)$ to denote the order of $m \pmod{\ell}$ as an element of the multiplicative group $(\mb{Z}/\ell \mb{Z})^{\times}$. \\
Given an integer $q \geq 1$ we write $\Xi_q$ to denote the group of Dirichlet characters $\pmod{q}$. We denote by $\chi_0^{(q)}$ (or simply $\chi_0$ when the modulus is understood from context) the principal character $\pmod{q}$. If $q$ is a prime power then we also denote by $\chi_q$ a generator for $\Xi_q$.\\
Given a set of prime powers $\mc{S}$ and $n \in \mb{N}$ we write
$$
\omega_{\mc{S}}(n) := |\{p^\nu \in \mc{S} : p^\nu || n\}|.
$$
Given a prime $p$ and an integer $n$ we write 
$$
\nu_p(n) := \max \{\nu \geq 1 : p^\nu|n\}.
$$
Given a finite abelian group $(G,+)$, written additively, and given a set $A \subseteq G$ and an integer $m \geq 1$ we write
$$
mA := \{a_1 + \cdots + a_m : a_i \in A\}
$$
to denote the \emph{$m$-fold sumset} of $A$.\\
Given $n \in \mb{N}$ we write $\phi(n)$ to denote the Euler phi function. 
We also write $P^+(1) := 1$, and for $n > 1$ let $P^+(n)$ denote the largest prime factor of $n \in \mb{N}$. Finally, we write $d(n)$ to denote the number of divisors of $n$.\\
Given a probability space $(\Omega,\mc{B},\mb{P})$ and a random variable $X$ on $\Omega$ we write $\mb{E}(X)$ to denote the expectation of $X$.\\
Finally, the following definitions will be used repeatedly in the sequel.
\begin{def1} \label{def:sparse}
Let $C > 0$. We say that a subset $\mc{S}$ (finite or infinite) of prime powers is \emph{$C$-sparse} if
$$
\sum_{p^\nu \in \mc{S}} \frac{1}{p^\nu} \leq C.
$$
More generally, we say that $\mc{S}$ is \emph{sparse} if there is a $C > 0$ for which it is $C$-sparse.
\end{def1}

\begin{def1}\label{def:logDens}
Given a set $\mc{A} \subset \mb{N}$ the \emph{upper (natural) density} and \emph{upper logarithmic density} of $\mc{A}$ are, respectively,
$$
\bar{d}(\mc{A}) := \limsup_{x \ra \infty} \frac{1}{x} |\mc{A}\cap [1,x]|, \quad \bar{\delta}(\mc{A}) := \limsup_{x \ra \infty} \frac{1}{\log x} \sum_{n \leq x} \frac{1_{\mc{A}}(n)}{n}.
$$
\end{def1}

\section{Proof Ideas}
\subsection{Proof strategy for Theorem \ref{thm:Inhom}}
To prove Theorem \ref{thm:Inhom} we apply two version of the ``projection method'' introduced in \cite{ManConsec}, namely the ``archimedean'' variant found there, as well as a ``non-archimedean'' variant; see Section \ref{subsec:nonArch} for the details. The basic objective of this method is to obtain information about the distribution of the (possibly unbounded) sequence $(f(p))_p$ by ``projecting'' $f$ onto the closed unit disc
$$
\mb{U} := \{z \in \mb{C} : |z| \leq 1\}.
$$ 
in a way that preserves its multiplicative structure. Specifically, we study a collection of \emph{bounded} multiplicative functions of the shape
$$
|f|_t(n) := |f(n)|^{it}, \quad f^{\chi}(n) := \chi(f(n)), \quad f(n) \neq 0,
$$
where $t \in \mb{R}$ and $\chi$ is a Dirichlet character, suitably chosen in each case. The analysis of the functions $|f|_t$ will follow swiftly from the corresponding work done in \cite[Sec. 3]{ManConsec}. The key novelty of the present paper concerns our treatment of the compositions $f^{\chi}$. These ideas were inspired by previous work \cite{ManHighOrd} of the author in the context of studying partial sums of Dirichlet characters of large order. \\ \\
\underline{The Archimedean argument:} 
As in Section 3 of \cite{ManConsec}, we use the compositions $|f|_t(n) := |f(n)|^{it}$ (for $f(n) \neq 0$) to probe the magnitude of the prime values $(|f(p)|)_p$. Our main result in that direction, Proposition \ref{prop:ArchInhom}, is the exact analogue of \cite[Prop. 3.1]{ManConsec}. It states that if $\overline{\delta}(\mc{N}_{f,a,b}) > \delta$ then there is $0 < r \ll_{\delta} 1$ and a $O_{\delta}(1)$-sparse set $\mc{S}_{\delta}$ of primes such that
\begin{equation}\label{eq:fppr}
|f(p)|/p^r \in (1/2,2) \text{ for all } p \notin \mc{S}_{\delta}.
\end{equation}
The proof follows the same lines as \cite[Prop. 3.1]{ManConsec}. The argument there involved an analysis of those $f$ for which the set of $n$ that satisfies 
\begin{equation}\label{eq:archCond}
f(n) \neq 0 \text{ and } \frac{1}{2} < |f(n+a)/f(n)| < 2
\end{equation}
has positive upper logarithmic density. For $n \in \mc{N}_{f,a,b}$ the above chain of inequalities holds as long as $|f(n)| > 2|b|$. We use the Tur\'{a}n-Kubilius inequality (Lemma \ref{lem:TK}) and the fact that $|f(p)| \neq 1$ for \emph{many} primes $p$ to show that $|f(n)|> 100|b|$, say, for all but a small subset of $\mc{N}_{f,a,b}$ (see Lemma \ref{lem:bddCase}). In this way, we are able to apply the arguments
in \cite[Sec. 3]{ManConsec} to our problem.\\ \\
\underline{The non-Archimedean argument:}
We use the compositions $f^{\chi}(n) := \chi(f(n))$, where $\chi$ is a Dirichlet character of suitably chosen conductor $q$, to obtain information about the distribution of $(f(p) \pmod{q})_p$. We do this for many moduli $q$, which then sheds light on the global behaviour of $f(p)$, outsider of a sparse set of primes $p$. 
\subsubsection*{Shortcomings of the na\"{i}ve approach}
In analogy to the archimedean argument described above, we might study functions $f$ for which
$$
\bar{\delta}\{n \in \mb{N} : f(n+a) = f(n) + b\neq 0\} > \delta
$$
by considering the implied congruence 
$$
f(n+a) \equiv f(n) + b \pmod{q}
$$ 
for well-chosen \emph{large} primes $q$ that do not divide $f(n)f(n+a)$. \\
When $b = 0$ the orthogonality relation for Dirichlet characters allows us to relate the solutions to this congruence to averages
$$
\frac{1}{q-1} \sum_{\chi \pmod{q}} \frac{1}{\log x} \sum_{n \leq x} \frac{\chi(f(n+a))\bar{\chi}(f(n)+b)}{n}
= \frac{1}{q-1} \sum_{\chi \pmod{q}} \frac{1}{\log x} \sum_{n \leq x} \frac{f^\chi(n+a)\bar{f^\chi}(n)}{n}.
$$
Crucially, the summands are of the form $g(n+a)\bar{g}(n)$, where $g$ is a multiplicative function taking values in $\mb{U}$.
Given the hypothesis $\bar{\delta}(\mc{N}_{f,a,0}) > \delta$, 
taking $q$ large enough compared to $\delta$, we may ignore 
the contribution from the principal character $\chi_0$,
leading to a lower bound for the
\emph{non-trivial} correlations
\begin{equation}\label{eq:largebinCorrelHom}
\max_{\ss{\chi \pmod{q} \\ \chi \neq \chi_0}} \left|\frac{1}{\log x} \sum_{n \leq x} \frac{f^{\chi}(n+a) \bar{f^{\chi}}(n)}{n} \right| \gg \delta,
\end{equation}
for $\gg_{\delta} q$ primitive characters $\pmod{q}$. 
The theory of correlations of multiplicative functions, specifically Tao's theorem on binary correlations of bounded multiplicative functions (see Theorem \ref{thm:Tao}) can then be fruitfully applied to deduce non-trivial information\footnote{In fact, this was the approach taken in a previous version of \cite{ManConsec}, at least for integer-valued $f$. It was realised thereafter that the theorem could be proven (in the broader context of complex-valued functions) simply by bootstrapping the Archimedean argument involving the compositions $|f|_t$.}  about $f$.\\
If we try to argue in the same fashion with $b \neq 0$ then the averaged correlations
$$
\frac{1}{q-1} \sum_{\chi \pmod{q}} \frac{1}{\log x} \sum_{n \leq x} \frac{\chi(f(n+a))\bar{\chi}(f(n)+b)}{n}
$$
arise. Here, the map $n \mapsto \bar{\chi}(f(n)+b)$ is \emph{no longer multiplicative} if $q \nmid b$. Since $b$ is fixed and $q$ is meant to be large as a function of $\delta$, we must  have $q \nmid b$. Thus, a direct appeal to Tao's theorem is not possible here. \\
Since $\chi$ is periodic, one could na\"{i}vely circumvent this obstacle by splitting the set of $n$ according to the residue class $f(n) \pmod{q}$ and then applying orthogonality a second time. This yields
$$
\frac{1}{(q-1)^2} \sum_{\chi,\psi \pmod{q}} \left(\sum_{c \pmod{q}} \psi(c)\bar{\chi}(c+b)\right) \frac{1}{\log x} \sum_{n \leq x} \frac{f^{\chi}(n+a)\bar{f^{\psi}}(n)}{n}. 
$$
The bracketed sum (which is, up to a simple change of variables, a Jacobi sum) has absolute value $\sqrt{q}$ whenever $\chi,\psi$ are distinct non-principal characters. After removing the contributions from the cases that $\chi,\psi$ or $\chi\bar{\psi}$ are principal, we at best obtain the lower bound\footnote{This bound, obtained by the triangle inequality, cannot in general be improved upon given that the complex arguments of the Jacobi sums of characters modulo $q$ are known to distribute uniformly on $S^1$ by work of Katz and Zheng \cite{KaZh}, and thus are difficult to control.}
$$
\max_{\ss{\chi,\psi \pmod{q} \\ \chi \neq \psi \\ \chi,\psi \neq \chi_0}} \left|\frac{1}{\log x} \sum_{n \leq x} \frac{f^{\chi}(n+a)\bar{f^{\psi}}(n)}{n}\right| \gg \frac{\delta}{\sqrt{q}}.
$$
For $q = q(x)$ large relative to $\delta$, this lower bound is too small for existing results about correlations of multiplicative functions to be effective for general $f$.
\subsubsection*{An alternative strategy}
Seeking a different approach, therefore, we observe the following. Suppose $\ell$ is prime and $n \in \mc{N}_{f,a,b}$ satisfies $\ell||n$. Writing $n = m\ell$ with $\ell \nmid m$, we have
$$
f(\ell m+a) = f(\ell)f(m)+b.
$$
Assume further that $(f(\ell),b) = 1$, and let $q$ be a prime power that divides $f(\ell)$.
It follows that
\begin{equation}\label{eq:MFAP}
f(\ell m+a) \equiv b \pmod{q}.
\end{equation}
It is this sort of congruence that we capture using the functions $\{f^{\chi}\}_{\chi \pmod{q}}$, rather than $f(n+a) \equiv f(n) + b \pmod{q}$. Besides avoiding the non-multiplicativity issue 
discussed earlier, this approach allows us to 
bring to bear the quantitative theory of \emph{mean values} of multiplicative functions (over arithmetic progressions), which is currently much better understood than that of \emph{correlations}. \\
In order to apply this idea we use an estimate due to Elliott (see Lemma \ref{lem:Ell} below), which shows that for $X$ large, $\mc{N}_{f,a,b}$ contains approximately the ``expected'' number of $n$ with $\ell || n$, i.e., 
\begin{equation}\label{eq:EllAppIdea}
|\{n \in \mc{N}_{f,a,b} : \ell || n\} \cap [1,X]| \approx \frac{1}{\ell}\left(1-\frac{1}{\ell}\right) |\mc{N}_{f,a,b} \cap [1,X]|
\end{equation}
for all but a small set of primes $\ell$ (that may depend on $X$). As $\bar{\delta}(\mc{N}_{f,a,b}) > \delta$ we can find an infinite increasing sequence $(x_k)_k$ along which the right-hand side of \eqref{eq:EllAppIdea} is $\gg \delta x_k/\ell$ when $X = x_k$, outside of a small set of exceptional\footnote{The dependence of the exceptional set on $x_k$ introduces a complication, but is dealt with by suitably arranging the arguments in the sequel.} $\ell$. Moreover, most of those $\ell$ for which \eqref{eq:EllAppIdea} holds satisfy $(2b,f(\ell)) = 1$ and $\ell \nmid a$ (see Lemma \ref{lem:Gdeltakm}). Thus, for any odd prime power $q|f(\ell)$, orthogonality of Dirichlet characters modulo $q$ gives
$$
\frac{1}{\phi(q)} \sum_{\chi \pmod{q}} \bar{\chi}(b) \sum_{\ss{n \leq x_k \\ n \equiv a \pmod{\ell}}} f^{\chi}(n) \geq \sum_{\ss{m \leq x_k/\ell \\ \ell \nmid m}} 1_{f(m\ell + a) \equiv b \pmod{q}} \gg \frac{\delta x_k}{\ell}.
$$
In particular, we find that for $\gg_{\delta} \phi(q)$ characters $\chi \pmod{q}$ we have
$$
\left|\sum_{\ss{n \leq x_k \\ n \equiv a \pmod{\ell}}} f^{\chi}(n)\right| \gg \frac{\delta x_k}{\ell}.
$$
Since $f^{\chi}$ is a bounded multiplicative function, work of Granville, Harper and Soundararajan \cite{GHS} implies (see Lemma \ref{lem:charDist}) that as long as $\ell$ is not too large, for any such $\chi$ we have
$$
\max_{\psi \pmod{\ell}} \left|\sum_{n \leq x_k} f^{\chi}(n) \bar{\psi}(n)\right| \gg \delta x_k.
$$
This is the desired analogue of \eqref{eq:largebinCorrelHom} in our circumstances, in which the summands $n \mapsto f^{\chi} \bar{\psi}(n)$ are multiplicative functions. By a quantitative version of Hal\'{a}sz's theorem (see Theorem \ref{thm:HMT}) we deduce that there are $\gg_{\delta}\phi(q)$ characters $\chi \pmod{q}$ for which there exists a character $\psi_{\chi}$ modulo $\ell$ and a $t_{\chi} \in \mb{R}$ such that
\begin{align}\label{eq:approxChifPsi}
\chi(f(p)) \approx \psi_{\chi}(p)p^{it_{\chi}} \text{ for most } p \leq x_k.
\end{align}
A key point here is that the character groups $\Xi_q$ and $\Xi_\ell$ are both cyclic.  On this basis, we use inverse sumset theorems from elementary additive combinatorics to study the set of $\chi$ satisfying \eqref{eq:approxChifPsi}. 
We 
show more precisely that there is a \emph{large} subgroup $H = H_{k,q,\ell} \leq \Xi_q$ of $\chi$ that satisfy \eqref{eq:approxChifPsi} with
$$
t_{\chi} = 0 \text{ for all } \chi \in H.
$$
Moreover, the mapping $\gamma: \chi \mapsto \psi_{\chi}$ is a homomorphism whose image lies in $\Xi_\ell$.
Since $H$ and $\gamma(H)$ are both cyclic, we show (see Proposition \ref{prop:almostHom} below) that there is a positive integer $D = O_{\delta}(1)$, independent of $k$ and $\ell$, a homomorphism
$$
\varphi: (\mb{Z}/\ell\mb{Z})^\times \rightarrow (\mb{Z}/q\mb{Z})^\times
$$
and an $O_{\delta}(1)$-sparse set of primes $\mc{T}_{\ell} \subseteq \mc{P}\cap [2,x_k]$ such that the following holds:
if $p \leq x_k$, $p \notin \mc{T}_\ell$ then
$$
p \equiv a \pmod{\ell} \Rightarrow f(p)^D \equiv \varphi(a) \pmod{q}.
$$
(The exceptional set $\mc{T}_{\ell}$ depends on $\ell$ and $k$, which introduces additional complications). \\
Assuming Conjecture \ref{conj:Inhom} holds, we expect that if $q|f(\ell)$ then $q$ is a power of $\ell$. In that case, taking $q = \ell$, $f^D$ is behaves (at least along the primes) like an endomorphism $\varphi$ on $(\mb{Z}/\ell \mb{Z})^{\times}$. Since this group is cyclic, it is easy to show that any such endomorphism is a power map $a \mapsto a^{g_{\ell}}$, for some $0 \leq g_{\ell} < \ell-1$. Thus, we deduce that
\begin{equation}\label{eq:locPowMapDef}
f(p)^D \equiv p^{g_{\ell}} \pmod{\ell} \text{ for all but a sparse set } p \in \mc{T}_{\ell},
\end{equation}
whenever $\ell$ is a prime (excluded from a sparse set) such that $\ell|f(\ell)$.
\subsubsection*{Local-to-global principle for power maps} Given a prime $\ell$ let us say that a map $F: \mb{N} \ra \mb{Z}$ is a \emph{local power map modulo $\ell$} if there exists an integer $k_{\ell}$ such that
\begin{equation}\label{eq:LPMProp}
F(n) \equiv n^{k_{\ell}} \pmod{\ell} \text{ for all } n \in \mb{N}.
\end{equation}
We say that $F$ is a \emph{global power map} if there is a non-negative integer $k$ such that
$$
F(n) = n^k \text{ for all } n \in \mb{N}.
$$
Obviously, any global power map is a local power map modulo any prime $\ell$. Conversely, it is not difficult to show that if $F$ satisfies \eqref{eq:LPMProp} with $k_{\ell} = k$ for infinitely many $\ell$ then $F$ is a global power map. It is reasonable to expect more generally that if $F$ is a local power map (without restrictions on $k_{\ell}$) for \emph{sufficiently many} distinct primes $\ell$ then $F$ is a global power map. Such a local-to-global phenomenon for multiplicative functions $F$ is the subject of a 1988 conjecture due to Fabrykowski and Subbarao \cite{FabSub} (see Conjecture \ref{conj:FS} and Lemma \ref{lem:red1} for relevant statements). Their conjecture states that if $F$ is a multiplicative function that satisfies \eqref{eq:LPMProp} modulo infinitely many primes $\ell$ then $F$ is a global power map. As far as the author is aware the conjecture is still open in full generality. \\ 
Using an ingenious argument based on applying the Chebotarev density theorem to suitably chosen Kummer extensions of $\mb{Q}$, Jones \cite{Jones} proved that such a local-to-global phenomenon occurs provided the set of moduli $\ell$ for which $F$ is a local power map modulo $\ell$ has positive upper density in the primes, i.e.,
$$
\limsup_{X \ra \infty} \frac{1}{\pi(X)} |\{\ell \leq X : \eqref{eq:LPMProp} \text{ holds with } \ell\}| > 0.
$$
Equation \eqref{eq:locPowMapDef} shows that $F := f^D$ is a local power map modulo $\ell$ along the primes $p \notin \mc{T}_{\ell}$, for all primes $\ell$ in $S_f := \{\ell \in \mc{P} : \ell|f(\ell)\}$, outside of a sparse set. 
Incorporating some additional ideas into Jones' argument, we show a similar local-to-global phenomenon for $f^D$ in our more restricted situation, under the (weaker) assumption that $d_{\text{Dir}}(S_f) > 0$ (or indeed, conditionally on ERH, that $\sum_{p|f(p)} p^{-1} = \infty$). \\
More precisely, we first show that if \eqref{eq:locPowMapDef} holds for a set of primes with positive Dirichlet density then $f(\ell)^D$ is a power of $\ell$ for all but a sparse set of primes $\ell$ (see Proposition \ref{prop:Assumlfl}). Arguing by contradiction, we show that if $f(\ell)$ is \emph{not} a power of $\ell$ for many $\ell$ then we may cover the bulk of $S_f$ by \emph{thin} subsets of primes that are unramified in suitably constructed Kummer extensions 
$$
\mb{Q}(\zeta_m,p_1^{1/m},p_2^{1/m},F(p_1)^{1/m},F(p_2)^{1/m}), \quad m,p_1,p_2 \in \mc{P}, \, p_1 \neq p_2.
$$ 
Care is taken to ensure that the values $p_1,p_2,F(p_1),F(p_2)$ are multiplicatively independent, in order for the Galois groups of the extensions to be suitably large. To do this we must deal with the possibility that the exceptional sets $\mc{T}_{\ell}$ overlap. We succeed in showing in this direction that the vast majority of primes $p$ belong to only \emph{very few} of the sets $\mc{T}_\ell$. \\
Having done this, we show that if $|f(p)| = p^{\alpha_p}$ for all but a sparse set of $p$ then $\alpha_p = 1$ for all such $p$ (see Proposition \ref{prop:veryRigid}).  To do this, we first prove that $\alpha_p$ is constant (outside of a sparse set of $p$). This is accomplished by locating arbitrarily large primes $\ell$ and primes $p_\ell \notin \mc{T}_\ell$ that are primitive roots $\pmod{\ell}$ and that satisfy $|f(p_\ell)| \ll p_\ell^r$, where $r$ is the exponent from \eqref{eq:fppr}. To see why this is helpful, note that 
$$
p_{\ell}^{D\alpha_{p_\ell}} = f(p_\ell)^D \equiv p_{\ell}^{g_\ell} \pmod{\ell},
$$
with $D\alpha_{p_\ell} = O_{\delta}(1)$ and $g_\ell < \ell-1$. If $p_\ell$ is a primitive root modulo $\ell$ and $\ell \geq \ell_0(\delta)$ then, necessarily, $g_{\ell} = D\alpha_{p_\ell}.$ 
By pigeonholing over the possible (integer) values of $\alpha_{p_\ell} = O_{\delta}(1)$ we find that 
$$
g_{\ell} = D\alpha_0 = O_{\delta}(1) \text{ for many } \ell.
$$
We then show that for each fixed prime $p \geq p_0(\delta)$ we may find one of these $\ell$, large relative to $p$ and $\delta$, such that $f(p)^D \equiv p^{D\alpha_0} \pmod{\ell}$, 
leading to the conclusion $|f(p)| = p^{\alpha_0}$ for all but a sparse set of $p$. \\
To show that $\alpha_0 = 1$ we must exploit the equation $f(n+a) = f(n) + b$, which is linear in values of $(f(m))_m$, taking care to control the influence of certain exceptional primes in the factorisations of $n$ and $n+a$ using elementary sieve theory. \\
Putting all of these arguments together, we obtain the first claim of Theorem \ref{thm:Inhom}. The conditions on $a,b$ are a simple consequence when $f$ is completely multiplicative.
\subsection{Proof strategy for Proposition \ref{prop:converse}}
By hypothesis, $\mc{S}$ is a sparse set of primes such that $\mc{S}_d := \mc{S} \bk \{p : p|d\}$ satisfies $|\mc{S}_d| \geq 2$. Let $\mc{T}$ be a set of primes disjoint from $\mc{S}$.\\
To prove Proposition \ref{prop:converse} we construct the sets
$$
\mc{N} := \{dpm : p \in \mc{S}_d, \, (m,\mc{S}) = 1\}, \quad \mc{N}_{\mc{T}} := \{dpm \in \mc{N} : \lambda_{\mc{T}}(m) = +1\},
$$
where $\lambda_{\mc{T}}(m)$ is the completely multiplicative function defined at primes via 
$$
\lambda_{\mc{T}}(p) = \begin{cases} -1 \text{ if } p \in \mc{T} \\ +1 \text{ if } p \notin \mc{T}.\end{cases}
$$
We then construct a completely multiplicative function $f = f_{\mc{T}}$ satisfying the conditions (i)-(iii) of the proposition (in particular that $f(p) = \lambda_{\mc{T}}(p) p$ for all $p \notin \mc{S}$) such that
$$
\mc{N}_{\mc{T}}' := \{n \in \mb{N} : n,n+a \in \mc{N}_{\mc{T}}\} \subseteq \mc{N}_{f,a,b}.
$$
To verify (iv) it suffices to show that $\overline{\delta}(\mc{N}_{\mc{T}}') > 0$ for uncountably many $\mc{T}$. Note that when $\mc{T} = \emptyset$ we have $\mc{N} = \mc{N}_{\mc{T}}$, and the positive density of $\mc{N}_{\emptyset}'$ (which is stronger than what we are aiming for) follows from a zero-dimensional sieve estimate.  In general, however, the condition $\lambda_{\mc{T}}(m) = +1$ introduces a complication. \\
Writing $g(n) := 1_{(n,\mc{S}) = 1}$, we reduce matters to showing that for any pair of fixed, distinct primes $p_1,p_2 \in \mc{S}_d$, 
and fixed integers $u,v \in \mb{Z}$ with $(u,p_2) = (v,p_1) = 1$,
$$
\frac{1}{d\log x} \sum_{\ss{n \leq x/(dp_1p_2) \\ \lambda_{\mc{T}}(np_2+u) = \lambda_{\mc{T}}(np_1+v) = +1}} \frac{g(np_2+u)g(np_1+v)}{n} \gg_{d,p_1,p_2} 1.
$$
Writing $1_{\lambda_{\mc{T}}(m) = +1} = (1+\lambda_{\mc{T}}(m))/2$, the problem becomes one of estimating the correlations
$$
\frac{1}{d\log x}\sum_{n \leq x/dp_1p_2} \frac{g\lambda_{\mc{T}}^{\eta_1}(np_2+u) g\lambda_{\mc{T}}^{\eta_2}(np_1+v)}{n}, \text{ for each } (\eta_1,\eta_2) \in \{0,1\}^2.
$$
The sum with $(\eta_1,\eta_2) = (0,0)$ effectively reduces the problem to the case $\mc{T} = \emptyset$, which is easily handled. 
The heart of the matter is to handle the contributions from $(\eta_1,\eta_2) \neq (0,0)$, and it suffices to show that these correlations are $o(1)$. \\
Appealing to Tao's theorem once again, the correlations are $o(1)$ for all $(\eta_1,\eta_2) \neq (0,0)$ as long as $\lambda_{\mc{T}}$ is \emph{non-pretentious}, i.e., $\lambda_{\mc{T}}$ does not correlate with any fixed, real Dirichlet character $\chi$ in the sense that
\begin{equation}\label{eq:nonPretCondlambdaT}
\lim_{x \ra \infty} \sum_{p \leq x} \frac{1-\lambda_{\mc{T}}(p)\chi(p)}{p} = \infty \text{ for each real, primitive } \chi.
\end{equation}
We show that this is the case almost surely
for $\mc{T}$ \emph{randomly chosen} (see Lemma \ref{lem:nonPretrandom}). That is, we show that $\lambda_{\mc{T}}$ is non-pretentious in the above sense if we let $\mc{T}$ be the outcome of a random variable\footnote{It should be noted that we may also \emph{deterministically} construct sets $\mc{T}$ for which $\lambda_{\mc{T}}$ is non-pretentious in the above sense. For example, any set $\mc{T}$ with Dirichlet density zero would suit our purposes. In any case, our probabilistic construction is intended to indicate that this is generic behaviour.} $\mathbf{T}$, restricted to a set of probability 1, defined by
$$
\mathbf{T} := \{p \notin \mc{S} : X_p = 1\}, 
$$
where $(X_p)_{p \notin \mc{S}}$ is a collection of i.i.d. Bernoulli random variables with mean $1/2$. Equation \eqref{eq:nonPretCondlambdaT} is proved using elementary calculations in probability theory, see Section \ref{sec:converse}. 

\section{Background and auxiliary results} \label{sec:background}
\subsection{Lemmas about additive functions}
Recall that $g: \mb{N} \ra \mb{C}$ is \emph{additive} if whenever $n,m\in \mb{N}$ are coprime, 
$$
g(mn) = g(m) + g(n).
$$
For an additive function, and $X \geq 2$ set
$$
A_g(X) := \sum_{p^\nu \leq X} \frac{g(p^\nu)}{p^\nu}\left(1-\frac{1}{p}\right), \quad B_g(X) := \left(\sum_{p^\nu \leq X} \frac{|g(p^\nu)|^2}{p^{\nu}}\right)^{1/2}.
$$
In the sequel, the following example will reoccur several times. Given $\mc{S}$ a set of prime powers define the additive function
$$
\omega_{\mc{S}}(n) := |\{p^{\nu} || n : p^\nu \in \mc{S}\}|.
$$
In this case,
\begin{equation}\label{eq:AgBgComp}
A_{\omega_{\mc{S}}}(X) = \sum_{\ss{p^\nu \leq X \\ p^\nu \in \mc{S}}} \frac{1}{p^\nu}\left(1-\frac{1}{p}\right) = \sum_{\ss{p^\nu \leq X \\ p^\nu \in \mc{S}}} \frac{1}{p^\nu} + O(1) = B_{\omega_{\mc{S}}}(X)^2 + O(1).
\end{equation}
The function $A_g(X)$ is asymptotically the mean value of $g(n)$ for $n \in [1,X]$, as $X \ra \infty$. The following lemma allows us to estimate how frequently and by how much $g(n)$ can deviate from $A_g(X)$.
\begin{lem}[Tur\'{a}n-Kubilius inequality]\label{lem:TK}
Let $g: \mb{N} \ra \mb{C}$ be an additive function, and let $X \geq 3$. Then
$$
\frac{1}{X}\sum_{n \leq X} |g(n) - A_g(X)|^2 \ll B_g(X)^2.
$$
The implicit constant is absolute.
\end{lem}
\begin{proof}
This is \cite[Lem. 4.1]{Ell1}.
\end{proof}
\noindent The following ``dual'' form of Lemma \ref{lem:TK} was discovered by Elliott.
\begin{lem}[Elliott] \label{lem:Ell}
Let $(a_n)_n \subset \mb{C}$ be a sequence and let $X \geq 1$. Then
$$
\sum_{p^\nu \leq X} p^{\nu}\left|\sum_{\ss{n \leq X \\ p^\nu||n}} a_n - \frac{1}{p^\nu}\left(1-\frac{1}{p}\right) \sum_{n \leq X} a_n\right|^2 \ll X\sum_{n \leq X} |a_n|^2.
$$
\end{lem}
\begin{proof}
This is \cite[Lem. 4.2]{Ell1}.
\end{proof}
\subsection{Background on the pretentious distance}
As in \cite{ManConsec} we will make frequent use of the pretentious distance of Granville and Soundararajan. Given arithmetic functions $f,g : \mb{N} \ra \mb{U}$ and $x \geq 2$ we define
$$
\mb{D}(f,g;x) := \left(\sum_{p \leq x} \frac{1-\text{Re}(f(p)\bar{g}(p))}{p}\right)^{1/2}.
$$
\noindent We recall the following two forms of the \emph{pretentious triangle inequality} (see \cite[Lem. 3.1]{GSPret}). For any $J \geq 1$ and 1-bounded arithmetic functions $f,g,h, (f_j)_{1 \leq j \leq J},(g_j)_{1 \leq j \leq J}$ we have
\begin{align}
&\mb{D}(f,h;x) \leq \mb{D}(f,g;x) + \mb{D}(g,h;x), \label{eq:usualTri}\\
&\mb{D}(f_1\cdots f_J,g_1 \cdots g_J;x) \leq \sum_{1 \leq j \leq J} \mb{D}(f_j,g_j;x).
\label{eq:multTri}
\end{align}
Controlling the distance $\mb{D}(n^{it},1;x)$ uniformly in a range of $t$ is important in the study of mean values and correlations of 1-bounded multiplicative functions. In this direction, 
Lemma 2.1 in \cite{ManConsec} gave the standard estimates
\begin{align} \label{eq:nitTo1Small}
\mb{D}(1,n^{it};x)^2 \geq \begin{cases}  \log(1+|t|\log x) + O(1) &\text{ if } |t| \leq 10 \\
(1/3-o(1)) \log\log x &\text{ if } 10 \leq |t| \leq x^2.
\end{cases}
\end{align}
In addition, in the sequel we will need to consider the distance between $n^{it}$ and more general primitive Dirichlet characters $\chi$. In this direction we will use the following standard fact repeatedly.
\begin{lem}\label{lem:charnit}
There is an absolute constant $C_0 \geq 1$ such that for $C \geq C_0$ the following is true. Let $x$ be sufficiently large relative to $C$, and let $1 \leq q \leq x^{1/(5C)}$ be a prime power. Let $\chi$ be a Dirichlet character modulo $q$ and let $|t| \leq x^2$. If
$$
\mb{D}(\chi, n^{it};x)^2 \leq C
$$
then $\chi$ is principal and $|t| \ll_C \tfrac{1}{\log x}$.
\end{lem}
\begin{proof}
Assume first that $\chi$ is non-principal. By\footnote{The bound in \cite{MRAP} is given only in the range $|t| \leq 10 x$, but the same argument extends to the full range $|t| \leq x^2$.} \cite[Lem. 7.8]{MRAP} we have
$$
C \geq \mb{D}(\chi,n^{it};x)^2 \geq \frac{1}{4} \log\left(\frac{\log x}{\log q}\right) + O(1) \geq \frac{5}{4}C + O(1) > C
$$
if $C_0$ is sufficiently large and $C \geq C_0$. This contradiction implies that $\chi$ must in fact be principal. \\
Next, we observe that
$$
C \geq \mb{D}(\chi_0,n^{it};x)^2 \geq \mb{D}(1,n^{it};x)^2 + O(1).
$$
In light of \eqref{eq:nitTo1Small},
if $x$ is large enough we must have $|t| \leq 10$, and indeed $|t| \log x \ll_C 1$, as claimed.
\end{proof}
\noindent Finally, for real-valued multiplicative functions $g: \mb{N} \ra [-1,1]$ we will use the following.
\begin{lem}\label{lem:realNonPret}
Let $x \geq 3$. Let $g : \mb{N} \ra [-1,1]$ be multiplicative, and let $\chi$ be a fixed Dirichlet character. Put $\e_{\chi} := 1$ if $\chi$ is real-valued, and $\e_{\chi} = 0$ otherwise. Then
$$
\inf_{\e_{\chi} \leq |t| \leq x^2} \mb{D}(g,\chi(n)n^{it};x) \geq \frac{1}{4}\sqrt{\log\log x} + O_{\chi}(1).
$$
Moreover, if $\chi$ is real-valued then
$$
\inf_{|t| \leq 1} \mb{D}(g,\chi(n)n^{it};x) \geq \frac{1}{3}\mb{D}(g,\chi;x) + O(1).
$$
\end{lem}
\begin{proof}
This is \cite[Lem. C.1]{MRT}.
\end{proof}
\subsection{Auxiliary results about multiplicative functions}
As in \cite{ManConsec}, Tao's theorem on logarithmically-averaged binary correlations plays a crucial role in our analysis.
\begin{thm}[Tao, \cite{Tao}] \label{thm:Tao}
Let $\eta > 0$ and let $a,c \in \mb{N}$, $b,d \in \mb{Z}$ satisfy $ad-bc \neq 0$. Let $g_1,g_2: \mb{N} \ra \mb{U}$ be 1-bounded multiplicative functions with the property that for some $x \geq x_0(\eta)$,
$$
\frac{1}{\log x} \left|\sum_{n \leq x} \frac{g_1(an+b) g_2(cn+d)}{n} \right| \geq \eta.
$$
Then there are primitive Dirichlet characters $\psi_j = \psi_{j,x}$ with conductor $O_{\eta}(1)$ and real numbers $t_j = t_{j,x}$ with $|t_j| = O_{\eta}(x)$ such that for $j = 1,2$,
$$
\mb{D}(g_j,\psi_j(n)n^{it_j};x)^2 = \sum_{p \leq x} \frac{1-\normalfont{\text{Re}}(g_j(p)\bar{\psi}_j(p)p^{-it_j})}{p} = O_{\eta}(1).
$$
\end{thm}
In addition, we will also need the following quantitative version of Hal\'{a}sz's theorem on mean values of multiplicative functions, which provides similar but more precise data about functions with large partial sums than Tao's theorem.
\begin{thm}[Hal\'{a}sz-Montgomery-Tenenbaum Inequality] \label{thm:HMT}
Let $x \geq 3$ and $T \geq 1$. Let $g : \mb{N} \ra \mb{U}$ be a multiplicative function. Then
$$
\frac{1}{x} \sum_{n \leq x} g(n) \ll (1+M_g)e^{-M_g} + \frac{1}{T},
$$
where $M_g = M_g(x;T)$ is defined by
$$
M_g(x;T) := \min_{|t| \leq T} \mb{D}(g,n^{it};x)^2 = \min_{|t| \leq T} \sum_{p \leq x} \frac{1-\normalfont{\text{Re}}(g(p)p^{-it})}{p}.
$$
\end{thm}
\begin{proof}
This is \cite[Cor. III.4.12]{Ten}.
\end{proof}
The following variant of Shiu's theorem, which is a consequence of a result due to Pollack \cite{Pol}, will allow us to work conveniently with sums over shifted primes in Section \ref{sec:LocGlob}.
\begin{lem} \label{lem:Pol}
Let $X \geq 2$ and let $k \geq 0$ be an integer. For each prime $p \leq X$ let $\mc{E}_p \subseteq (\mb{Z}/p\mb{Z})^{\times}$ be a set of $\nu(p) \leq k$ non-zero residue classes, and let
$$
\mathscr{S} := \bigcap_{p \leq X} \{n \leq X : \, n \pmod{p} \notin \mc{E}_p\}.
$$
Then, uniformly over all multiplicative functions $f: \mb{N} \ra \mb{R}$ that satisfy  
$$
0 \leq f(n) \leq d(n) \text{ for all } n,
$$ 
we have
$$
\sum_{\ss{n \leq X \\ n \in \mathscr{S}}} f(n) \ll_k \frac{X}{\log X} \exp\left(\sum_{p \leq X} \frac{f(p)-\nu(p)}{p}\right).
$$
\end{lem}
\begin{proof} 
This is immediate from \cite[Thm. 1.1]{Pol}, taking $A_1 = 2$ in the definition of $\mathscr{M}$ there, and $y = X$.
\end{proof}

\section{Applying the projection method: the archimedean argument}\label{sec:arch}
Given $a \in \mb{N}$, $b \in \mb{Z} \bk \{0\}$ and a multiplicative function $f: \mb{N} \ra \mb{Z}$ recall the notation
$$
\mc{N}_{f,a,b} := \{n \in \mb{N} : f(n+a) = f(n) + b\}.
$$
In this section we use the compositions $|f|_t := |f|^{it}$ to study the sequence of magnitudes $(|f(p)|)_p$ when $\overline{\delta}(\mc{N}_{f,a,b}) > 0$. In most aspects this is identical to the procedure in \cite[Sec. 3]{ManConsec}, where $b = 0$.
\begin{prop} \label{prop:ArchInhom}
Let $f: \mb{N} \ra \mb{Z}$ be a multiplicative function that satisfies
\begin{equation}\label{eq:cond01}
\sum_{p : |f(p)| \neq 1} \frac{1}{p} = \infty,
 \quad \sum_{p : f(p) = 0} \frac{1}{p} < \infty.
\end{equation}
Let $a \in \mb{N}$, $b\in \mb{Z} \bk \{0\}$, and assume that $\bar{\delta}(\mc{N}_{f,a,b}) > \delta$ for some $\delta \in (0,1/2)$. 
Then there is a $0 < r \ll_{\delta} 1$ and a $O_{\delta}(1)$-sparse sequence of primes $\mc{S}_{\delta}$ 
such that 
$$
\frac{1}{2} < |f(p)|/p^r < 2 \text{ for all } p \notin \mc{S}_{\delta}.
$$
\end{prop}

\begin{rem} \label{rem:sparseVan}
We may assume that for $f$ in Conjecture \ref{conj:Inhom} the set $\{p \in \mc{P} : f(p) = 0\}$ is sparse, and thus Conjecture \ref{conj:Inhom} classifies all non-trivial functions for which $\overline{\delta}(\mc{N}_{f,a,b}) > 0$ for some $ab \neq 0$. \\
Indeed, suppose otherwise. Then the multiplicative function $1_{f(n) \neq 0}$ has vanishing logarithmic mean: 
$$
\frac{1}{\log x}\sum_{n \leq x} \frac{1_{f(n) \neq 0}}{n} \ll \exp\left(-\sum_{\ss{p \leq x \\ f(p) = 0}} \frac{1}{p}\right) = o(1), \quad x \ra \infty.
$$
It follows that $f(n+a) = f(n) = 0$ for a logarithmic density $1$ set of $n$. In particular, $\delta(\mc{N}_{f,a,b}) = 0$ whenever $ab \neq 0$, contradicting our hypothesis. There is therefore no loss in assuming both of the conditions \eqref{eq:cond01} here. 
\end{rem}

\noindent Let $\e \in (0,1)$. Note that if $n \in \mc{N}_{f,a,b}$ satisfies $|f(n)| > |b|\e^{-1}$  then
$$
f(n+a)/f(n) \in (1-\e,1+\e).
$$
Our arguments treat precisely those $n$ for which such a condition holds. The following lemma implies that the converse inequality $|f(n)| \leq |b|\e^{-1}$ holds for very few $n$.
\begin{lem} \label{lem:bddCase}
Assume that $f$ satisfies both conditions in \eqref{eq:cond01}. Then for any fixed $M \geq 1$ we have
$$
\frac{1}{\log x}\sum_{\ss{n \leq x \\ 0 < |f(n)| \leq M}} \frac{1}{n} = o(1).
$$
\end{lem}
\begin{rem}\label{rem:bddCase}
It is not difficult to see that the proof furthermore implies the stronger bound
$$
\frac{1}{x} |\{n \leq x : 0 < |f(n)| \leq M\}| = o(1) \text{ as } x \ra \infty.
$$
\end{rem}
\begin{proof}
Let $\mc{S} := \{p^\nu : |f(p^\nu)| \neq 0, 1\}$, and define $\omega_{\mc{S}}(n)$ to be the number of prime powers $p^\nu \in \mc{S}$ with $p^\nu||n$. 
Since $f$ takes integer values, $|f(p^\nu)| \geq 2$ for all $p^\nu \in \mc{S}$. Thus, 
$$
\text{if } 0 < |f(n)| \leq M \text{ then }
\omega_{\mc{S}}(n) \leq \frac{\log (2M)}{\log 2}.
$$
According to \eqref{eq:cond01}, as $X \ra \infty$ we have
$$
E_{\mc{S}}(X) := \sum_{\ss{p^\nu \leq X \\ p^\nu \in \mc{S}}} \frac{1}{p^\nu}\left(1-\frac{1}{p}\right) \geq \frac{1}{2}\sum_{\ss{p \leq X \\ |f(p)| \neq 0,1}} \frac{1}{p} \ra \infty.
$$ 
Furthermore, by Lemma \ref{lem:TK} (and \eqref{eq:AgBgComp}),
$$
\frac{1}{X} \sum_{X \leq n < 2X} |\omega_{\mc{S}}(n) - E_{\mc{S}}(X)|^2 \ll E_S(X), \quad X \geq 2.
$$
Therefore when $X \geq X_0(M)$ we have
\begin{align*}
|\{X \leq n < 2X : 0 < |f(n)| \leq M\}| &\leq |\{X \leq n < 2X : |\omega_{\mc{S}}(n) - E_{\mc{S}}(X)| > E_{\mc{S}}(X)/2 \}| \\
&\ll \frac{1}{E_{\mc{S}}(X)^2} \sum_{n \leq 2X} |\omega_{\mc{S}}(n) - E_{\mc{S}}(X)|^2 \ll \frac{X}{E_{\mc{S}}(X)}. 
\end{align*}
Set $z := x^{1/E_{\mc{S}}(x)} \geq x^{1/(2\log\log x)}$. Taking a logarithmic average and decomposing over dyadic scales,
\begin{align*}
\frac{1}{\log x}\sum_{\ss{n \leq x \\ 0 < |f(n)| \leq M}} \frac{1}{n} &\leq \frac{1}{\log x}\sum_{z < y = 2^j \leq x} \frac{1}{y} |\{y \leq n < 2y: 0 < |f(n)| \leq M\}| + O\left(\frac{\log z}{\log x}\right) \\
&\ll \frac{1}{\log x}\sum_{z < y = 2^j \leq x} \frac{1}{E_{\mc{S}}(y)} + \frac{1}{\sqrt{\log x}} \ll \frac{1}{E_{\mc{S}}(z)},
\end{align*}
and the claim follows as $x$ (and thus also $z$) tends to infinity.
\end{proof}

\begin{proof}[Proof of Proposition \ref{prop:ArchInhom}]
Let $(X_j)_j \subseteq \mb{N}$ be an infinite increasing subsequence on which
$$
\bar{\delta}(\mc{N}_{f,a,b}) = \lim_{j \ra \infty} \frac{1}{\log X_j} \sum_{\ss{n \leq X_j \\ f(n+a) = f(n) + b}} \frac{1}{n}.
$$
Given $n \in \mb{N}$ with $f(n) \neq 0$, define $g(n) := \log |f(n)|$. 
Observe that 
$$
\text{if } f(n+a) = f(n) + b \text{ with } f(n+a)f(n) \neq 0, \text{ and } |f(n)| \geq 10|b|
$$ 
then it follows that
$$
|g(n+a)-g(n)| = |\log|f(n+a)/f(n)|| \leq \max\{\log(1+|b/f(n)|), -\log(1-|b/f(n)|)\} 
< \frac{1}{2},
$$
and in particular that
$$
1_{f(n+a) = f(n) + b} \leq 2 \max\{1-|g(n+a)-g(n)|,0\}.
$$
In order to use this we first discard those $n$ for which $f(n)f(n+a) = 0$ or $0 < |f(n)| \leq 10|b|$. \\
Let $j \geq j_0(\delta,|b|)$ be large and set $X = X_j$. Observe that since $b \neq 0$, by Lemma \ref{lem:bddCase} we have 
\begin{align*}
\frac{1}{\log X} \sum_{\ss{n \leq X \\ f(n+a) = f(n)+b \\ f(n)f(n+a) = 0}} \frac{1}{n} &\leq \frac{1}{\log X} \sum_{\ss{n \leq X \\ f(n+a) = b}}\frac{1}{n} + \frac{1}{\log X} \sum_{\ss{n \leq X \\ f(n) = -b}} \frac{1}{n} \\
&\ll \frac{1}{\log X} \sum_{\ss{m \leq X \\ 0 < |f(m)| \leq |b|}}\frac{1}{m}  + \frac{a}{\log X}
< \delta^2.
\end{align*}
Moreover, by the same argument we also have
$$
\frac{1}{\log X} \sum_{\ss{n \leq X \\ 0 < |f(n)| \leq 10 |b|}} \frac{1}{n} < \delta^2.
$$
Thus, if $j \geq j_0(\delta)$ we find that
\begin{align*}
\delta \leq \frac{1}{\log X} \sum_{\ss{n \leq X \\ f(n+a) = f(n)+b}} \frac{1}{n} &= \frac{1}{\log X} \sum_{\ss{n \leq X \\ f(n)f(n+a) \neq 0 \\ |f(n)| \geq 10|b|}} \frac{1_{f(n+a) = f(n)+b}}{n} + O(\delta^2) \\
&\leq \frac{2}{\log X} \sum_{\ss{n \leq x \\ f(n)f(n+a)\neq 0}} \frac{\max\{1-|g(n+a)-g(n)|,0\}}{n} + O(\delta^2).
\end{align*}
Recall the notation $|f|_{\beta}(n) := |f(n)|^{i\beta}$ for $n \in \mb{N}$ satisfying $f(n) \neq 0$ and $\beta \in \mb{R}$. Taking $B = \delta^{-2}$, this is sufficient to establish the conclusion of \cite[Lem. 3.2]{ManConsec}, namely that
the set
$$
\left\{\alpha \in [-B,B] : \left|\frac{1}{\log X}\sum_{\ss{n \leq X \\ f(n+a)f(n) \neq 0}} \frac{|f|_{\alpha}(n+a)|f|_{-\alpha}(n)}{n}\right| \geq \delta/16\pi\right\}
$$
has Lebesgue measure $\geq \delta/16\pi$ for each $X$ along an infinite subsequence of $(X_j)_j$.
As discussed in \cite[Sec. 3]{ManConsec}, the conclusion of \cite[Prop. 3.1]{ManConsec} follows from this, and so there is a $0 \leq r \ll_{\delta} 1$ and a set of primes $\mc{S}_{\delta}$ outside of which
$$
e^{-\delta^2} \leq |f(p)|/p^r \leq e^{\delta^2}.
$$
Since $\delta \in (0,1/2)$ we deduce that $e^{\delta^2} < 2$, and so
$$
\sum_{p : |f(p)|/p^r \notin (1/2,2)} \frac{1}{p} \leq \sum_{p \in \mc{S}_{\delta}} \frac{1}{p} < \infty.
$$
Now if $r = 0$ then as $f(p) \in \mb{Z}$,
$$
\sum_{p: |f(p)| \neq 1} \frac{1}{p} = \sum_{p : |f(p)| \notin (1/2,2)} \frac{1}{p} < \infty,
$$
a contradiction.
Therefore, $r > 0$ as required.
\end{proof}
\begin{rem}\label{eq:rPosRem}
Given that $r > 0$, we see that for $p \geq p_0(\delta)$,
$$
p \notin \mc{S}_{\delta} \Rightarrow |f(p)| \geq p^r/2 > 1.
$$
In particular, using Proposition \ref{prop:ArchInhom} we conclude that
$$
\sum_{\ss{p^\nu \leq X \\ |f(p^\nu)| \neq 0,1}} \frac{1}{p^\nu} = \log\log X + O\left(1 + \sum_{\ss{p \leq X \\ p \in \mc{S}_{\delta}}} \frac{1}{p}\right) = \log\log X + O_{\delta}(1),
$$
which quantitatively refines the first hypothesis in \eqref{eq:cond01}.
\end{rem}
\section{The non-Archimedean Argument: Reduction to local power maps} \label{subsec:nonArch}
\subsection{Setting up Proposition \ref{prop:almostHom}}
The objective of this section is to prove Proposition \ref{prop:almostHom}. Since its precise statement is somewhat technical we defer providing it for the moment and roughly describe its content here.\\
Since $\overline{d}(\mc{N}_{f,a,b}) \geq \overline{\delta}(\mc{N}_{f,a,b}) > \delta$ we can find an infinite increasing sequence $(x_k)_k$ such that
\begin{equation}\label{eq:posUppDensxk}
|\mc{N}_{f,a,b} \cap [1,x_k]| \geq \delta x_k, \quad k \geq k_0.
\end{equation}
Proposition \ref{prop:almostHom} asserts that there are positive integers $A,D = O_{\delta}(1)$ such that at each such scale $x_k$, for the vast majority of primes $\ell \leq x_k^{1/A}$ the power $f^D$ of $f$ 
behaves like a homomorphism 
$$
\varphi: (\mb{Z}/\ell \mb{Z})^{\times} \ra (\mb{Z}/q_{\ell} \mb{Z})^{\times}
$$ 
for any prime power $q_{\ell}|f(\ell)$ of our choice. More precisely, outside of a sparse set $\mc{T}_{\ell}$ of primes $p \leq x_k$, 
$$
f(p)^D \equiv \varphi(a) \pmod{q_\ell} \text{ whenever } p \equiv a \pmod{\ell}.
$$ 
In the case that $\ell|f(\ell)$ we may select $q_{\ell} = \ell$, and deduce that $f^D$ is (essentially) a \emph{local power map} modulo $\ell$, i.e. that for some $0 \leq g_{\ell} < \ell-1$,
$$
f(p)^D \equiv p^{g_{\ell}} \pmod{\ell}, \quad p \in [2,x_k] \bk \mc{T}_{\ell}.
$$
In the next section we will see how to ``glue'' together these local conditions to establish global information about $f$. \\
To this end, we begin by showing that $\mc{N}_{f,a,b}$ contains many multiples of primes $\ell$ for which $(b,f(\ell)) = 1$. This will allow us to show that
$$
f(n+a) = f(m\ell+a) \equiv b \pmod{q_{\ell}}
$$
for many multiples $n = m\ell \in \mc{N}_{f,a,b}$, a condition that we can then study using the compositions $\{f^{\chi}\}_{\chi \pmod{q_\ell}}$.
\begin{lem} \label{lem:multSet}
Let $(x_k)_k$ be the sequence for which \eqref{eq:posUppDensxk} holds. Then there is a collection of sets $\mc{P}_{\delta,k}$ of prime powers that satisfy
$$
\sum_{\ss{p^\nu \leq x_k \\ p^\nu \in \mc{P}_{\delta,k}}} \frac{1}{p^\nu} \ll \delta^{-2}, \quad k \geq 1
$$
such that if $k \geq k_0(\delta)$ and $p^\nu \leq x_k$, $p^\nu \notin \mc{P}_{\delta,k}$ then
$$
|\{n \leq x_k : n \in \mc{N}_{f,a,b}, \, p^\nu || n\}| \geq \frac{\delta x_k}{4p^\nu}.
$$
\end{lem}
\begin{proof}
Let $k \geq k_0(\delta)$, and define $a_n := 1_{\mc{N}_{f,a,b}}(n)$ for each $n \in \mb{N}$.
For each $p^\nu \leq x_k$ define also
$$
\Delta_{p^\nu}(x_k) := \left|\frac{p^\nu}{x_k} \sum_{\ss{n \leq x_k \\ p^\nu || n}} a_n - \left(1-\frac{1}{p}\right) \frac{1}{x_k}\sum_{n \leq x_k} a_n\right|,
$$
and define
$$
\mc{P}_{\delta,k} := \{p^\nu \leq x_{k} : \Delta_{p^\nu}(x_{k}) > \delta/4\}.
$$
By Lemma \ref{lem:Ell},
$$
\sum_{p^\nu \leq x_k} \frac{\Delta_{p^\nu}(x_k)^2}{p^\nu} \ll \frac{1}{x_k} \sum_{n\leq x_k} |a_n|^2 \ll 1,
$$
so that
$$
\sum_{\ss{p^\nu \in \mc{P}_{\delta,k}}} \frac{1}{p^\nu} = \sum_{\ss{p^\nu \leq x_k \\ \Delta_{p^\nu}(x_k) > \delta/4}} \frac{1}{p^\nu} \leq \frac{16}{\delta^2} \sum_{p^\nu \leq x_k} \frac{\Delta_{p^\nu}(x_k)^2}{p^\nu} \ll \delta^{-2}.
$$
On the other hand, note that
$$
\sum_{\ss{n \leq x_k \\ p^\nu||n}} a_n = |\{n \leq x_k : n \in \mc{N}_{f,a,b}, \, p^\nu||n\}|,
$$
and using \eqref{eq:posUppDensxk}, 
$$
\sum_{n \leq x_k} a_n = |\mc{N}_{f,a,b}\cap [1,x_k]| \geq \delta x_k.
$$
Hence, for each $p^\nu \leq x_k$ with $p^\nu \notin \mc{P}_{\delta,k}$ we have
$$
|\{n \leq x_k : n \in \mc{N}_{f,a,b}, \, p^\nu ||n\}| \geq \frac{x_k}{p^\nu} \left(\delta \left(1-\frac{1}{p}\right) - \Delta_{p^\nu}(x_k)\right) \geq \frac{\delta x_k}{4p^\nu},
$$
as required.
\end{proof}
Set $\mc{P}_f := \{p^\nu : |f(p^\nu)| \neq 0,1\}$. Given $m \in \mb{N}$ and $k \geq k_0(\delta)$ define also 
$$
\mc{G}_{\delta,k} := \mc{P}_f \bk (\mc{P}_{\delta,k} \cup \mc{S}_{\delta}), \quad \mc{G}_{\delta,k;m} := \{p^\nu \in \mc{G}_{\delta,k} : (m,f(p^\nu)) = 1\}.
$$
Combining Remark \ref{eq:rPosRem} and Lemma \ref{lem:multSet},
\begin{equation}\label{eq:relGdPf}
\sum_{\ss{p^\nu \leq x_k \\ p \in \mc{G}_{\delta,k}}} \frac{1}{p^\nu} = \sum_{\ss{p^\nu \leq x_k \\ p^\nu \in \mc{P}_f}} \frac{1}{p^\nu} + O\left( \sum_{\ss{p \leq x_k \\ p \in \mc{S}_{\delta}}} \frac{1}{p} + \sum_{\ss{p^\nu \leq x_k \\ p^\nu \in \mc{P}_{\delta,k}}} \frac{1}{p^\nu}\right) = \log\log x_k + O_{\delta}(1).
\end{equation}
The next lemma shows that \eqref{eq:relGdPf} holds with $\mc{G}_{\delta,k}$ replaced by $\mc{G}_{\delta,k;m}$, for any $m$ fixed.
\begin{lem} \label{lem:Gdeltakm}
For any $m \in \mb{N}$ we have
$$
\sum_{\ss{p^\nu \leq x_k \\ p^\nu \in \mc{G}_{\delta,k;m}}} \frac{1}{p^\nu} = \log\log x_k + O_{\delta, \omega(m)}(1).
$$
\end{lem}
\begin{proof}
For each prime $\ell|m$ define the set
$$
\mc{B}_{\ell} := \{p^\nu \in \mc{P}_f : \ell | f(p^\nu)\}, \quad \mc{B}_m := \bigcup_{\ell|m} \mc{B}_{\ell}.
$$
Then $(m,f(p^\nu)) > 1$ if and only if $p^\nu \in \mc{B}_m$, i.e., $\mc{G}_{\delta,k;m} = \mc{G}_{\delta,k} \bk \mc{B}_m$. We will show that 
\begin{equation}\label{eq:Bblarge}
\sum_{\ss{p^\nu \leq x_k \\ p^\nu \in \mc{B}_m}} \frac{1}{p^\nu} \ll_{\delta, \omega(m)} 1, \quad k \ra \infty.
\end{equation}
Thus, assume for the sake of contradiction that \eqref{eq:Bblarge} fails. For convenience, given $\ell|m$ we introduce the notation
$$
B_{\ell}(y) := \sum_{\ss{p^\gamma \leq y \\ p^\gamma \in \mc{B}_\ell}} \frac{1}{p^\gamma}, \quad y \geq 2,
$$
noting that $B_m(y) \ra \infty$ as $y \ra \infty$ by assumption. \\
Let
$\nu := \max_{\ell|m} \nu_{\ell}(b)$, and select $k_0 = k_0(\delta,\omega(m))$ sufficiently large so that 
$$
B_m(x_k) \geq \omega(m)\max\{\delta^{-2}, 20(\nu+1)\} \text{ for all } k \geq k_0.
$$
By the union bound we have
\begin{equation} \label{eq:BlBd}
\max_{\ell|m} B_{\ell}(x_k) \geq \frac{1}{\omega(m)} \sum_{\ell|m} B_{\ell}(x_k) \geq \frac{1}{\omega(m)} B_m(x_k) \geq \max\{\delta^{-2}, 20(\nu+1)\}.
\end{equation}
Let $\ell_0$ denote a prime factor of $m$ that maximises the left-hand side of \eqref{eq:BlBd}, 
and write 
$$
\omega_{\ell_0}(n) := |\{p^\gamma || n : p^\gamma \in \mc{B}_{\ell_0}\}| = |\{p^\gamma || n : \ell_0 | f(p^\gamma)\}|, \quad n \in \mb{N}.
$$ 
By Lemma \ref{lem:TK},
\begin{align*}
|\{n \leq x_k : \omega_{\ell_0}(n) < B_{
\ell_0}(x_k)/2\}| &\leq |\{n \leq x_k : |\omega_{\ell_0}(n) - B_{\ell_0}(x_k)| > B_{\ell_0}(x_k)/2\}| \\
&\ll \frac{1}{B_{\ell_0}(x_k)^2} \sum_{n \leq x_k} |\omega_{\ell_0}(n) - B_{\ell_0}(x_k)|^2 \\
&\ll x_k/B_{\ell_0}(x_k). 
\end{align*}
Since $B_{\ell_0}(x_k) \geq \delta^{-2}$, it follows that for all but $O(\delta^2 x_k)$ integers $n \leq x_k$ we have
\begin{equation} \label{eq:ell0Cond}
\min\{\omega_{\ell_0}(n), \omega_{\ell_0}(n+a)\} \geq B_{\ell_0}(x_k)/2 \geq 10 (\nu+1).
\end{equation}
If $\delta$ is small enough then by \eqref{eq:posUppDensxk} there are $\geq \delta x_k/2$ integers $n \leq x_k$ for which
$$
f(n+a) = f(n) + b,
$$
and \eqref{eq:ell0Cond} holds. Now, since $f$ is multiplicative we have $f(p^\gamma)|f(n)$ whenever $p^\gamma || n$, whence 
$$
\ell_0^{\omega_{\ell_0}(n)} | f(n), \text{ i.e. } \nu_{\ell_0}(f(n)) \geq \omega_{\ell_0}(n).
$$
Therefore, by \eqref{eq:ell0Cond}, 
$$
\nu_{\ell_0}(f(n+a)) \geq \min\{\nu_{\ell_0}(f(n)),\nu_{\ell_0}(f(n+a))\} \geq 10(\nu+1) > \nu_{\ell_0}(b).
$$
On the other hand, since $\nu_{\ell_0}(f(n)) \neq \nu_{\ell_0}(b)$ the non-archimedean property of $\nu_{\ell_0}$ yields
$$
\nu_{\ell_0}(f(n+a)) = \nu_{\ell_0}(f(n)+b) = \min\{\nu_{\ell_0}(f(n)),\nu_{\ell_0}(b)\} \leq \nu_{\ell_0}(b),
$$
which is a contradiction. The claim follows.
\end{proof}

\noindent For $k \geq k_0(\delta)$ consider a prime $\ell \leq x_k$ with $\ell \in \mc{G}_{\delta,k; 2b}$. By choice, $f(\ell)$ is odd and 
$$
\frac{\ell^r}{2}  < |f(\ell)| < 2 \ell^r \text{ for some } 0 < r \ll_{\delta} 1.
$$
In the sequel we fix\footnote{Actually, any sufficiently large prime power divisor of $f(\ell)$ would be suitable, but we make this choice for concreteness.}
 
$$
q_{\ell} = \max\{p^\nu : p^\nu || f(\ell)\}.
$$ 
\begin{rem}\label{rem:qellProps}
Note that $q_{\ell}$ must be an odd prime power, so that the group of units $(\mb{Z}/q_\ell \mb{Z})^\times$ is cyclic. We will use this property several times in the sequel.  Moreover,
since $\omega(n) \leq 2\frac{\log n}{\log\log n}$ for all $n \geq n_0$, whenever $\ell$ is sufficiently large we have
$$
q_\ell \geq |f(\ell)|^{1/\omega(|f(\ell)|)} \gg_{\delta} (\log \ell)^{1/2}.
$$ 
\end{rem}
We are now in a position to state the main result of this section.
\begin{prop} \label{prop:almostHom}
Let $f: \mb{N} \ra \mb{Z}$ be a multiplicative function for which
$$
\sum_{p : |f(p)| \neq 1} \frac{1}{p} = \infty, \quad \sum_{p : f(p) = 0} \frac{1}{p} < \infty.
$$
Assume moreover that there are positive integers $a \geq 1$ and $b \neq 0$, and a $\delta \in (0,1/2)$ such that $\bar{\delta}(\mc{N}_{f,a,b}) > \delta$.
Then there are
integers $A,D,k_0,\ell_0 = O_{\delta}(1)$ with $2|D$ such that the following holds. \\
If $k \geq k_0$ and $\ell \in \mc{G}_{\delta, k; 2b} \cap [\ell_0, x_k^{1/A}]$ is prime then there is a 
set of primes $\mc{T}_\ell \subseteq \mc{P} \cap [2,x_k]$ satisfying
$$
\sum_{\ss{p \leq x_k \\ p \in \mc{T}_\ell}} \frac{1}{p} = O_{\delta}(1)
$$
and a homomorphism 
$$
\tilde{\rho}_\ell: (\mb{Z}/\ell\mb{Z})^\times \rightarrow (\mb{Z}/q_\ell\mb{Z})^\times
$$
such that for each $a \in (\mb{Z}/\ell \mb{Z})^\times$
\begin{equation}\label{eq:genHomCond}
p \equiv a \pmod{\ell}, \, p \in (\mc{P} \cap [2,x_k]) \bk \mc{T}_\ell \Rightarrow f(p)^D \equiv \tilde{\rho}_\ell(a) \pmod{q_\ell}.
\end{equation}
Moreover, if $q_{\ell}$ is a power of $\ell$ then there is a $0 \leq g_{\ell} < \ell-1$ such that
\begin{equation}\label{eq:ellDivfell}
p \in (\mc{P} \cap [2,x_k]) \bk \mc{T}_\ell 
\Rightarrow f(p)^D \equiv p^{g_{\ell}} \pmod{\ell}.
\end{equation}
\end{prop}
\subsection{Deductions from pretentiousness} 
The proof of Proposition \ref{prop:almostHom} relies on variants of the additive combinatorial arguments employed in \cite[Sec. 3]{ManConsec}, applied here to the groups of Dirichlet characters to specific moduli. To deploy these we first need a few lemmas, derived from the pretentious theory of multiplicative functions. \\
In the sequel, we assume that $\delta$ is smaller than any fixed absolute constant (the condition $\overline{\delta}(\mc{N}_{f,a,b}) > \delta$ can only be weakened under this assumption). We also let $k_0(\delta)$ be larger than any fixed constant in terms of $\delta$, and let $k \geq k_0(\delta)$.  
\begin{lem} \label{eq:largefchiSums}
For each prime $\ell \in \mc{G}_{\delta, k; 2b} \cap [\ell_0(\delta), \delta^2 x_k]$ there is a collection of non-principal characters $\mc{X}_\ell$ modulo $q_\ell$ of size $|\mc{X}_\ell| \geq \delta \phi(q_\ell)/8$ such that
\begin{equation} \label{eq:lowBdChar}
\min_{\chi \in \mc{X}_{\ell}} \left|\sum_{\ss{n \leq x_k \\ n \equiv a \pmod{\ell}}} f^{\chi}(n)\right| \geq \frac{\delta x_k}{10\ell}.
\end{equation}
\end{lem}
\begin{proof}
Let $\ell \in \mc{G}_{\delta,k;2b} \cap [\ell_0(\delta), \delta^2 x_k]$. Since $\ell \in \mc{G}_{\delta,k}$, Lemma \ref{lem:multSet} yields 
$$
|\{m \leq x_k/\ell : \ell m \in \mc{N}_{f,a,b}, \, \ell \nmid m\}| = |\{n \leq x_k : n \in \mc{N}_{f,a,b}, \, \ell || n\}| \geq \frac{\delta x_k}{4\ell}.
$$
Suppose $\ell m \in \mc{N}_{f,a,b}$ with $\ell \nmid m$. Since $q_{\ell}|f(\ell)$, 
$$
f(\ell m+a) =  f(\ell)f(m) + b \equiv b \pmod{q_\ell}.
$$
Thus, we find that
$$
\sum_{\ss{n \leq x_k \\ n \equiv a \pmod{\ell}}} 1_{f(n) \equiv b \pmod{q_\ell}} = \sum_{\ss{m \leq x_k/\ell \\ f(\ell m+a) \equiv b \pmod{q_\ell}}} 1 \geq \frac{\delta x_k}{4\ell}.
$$
Since $\ell \in \mc{G}_{\delta,k; 2b}$ we have $(b,q_{\ell}) = 1$. Thus, by the orthogonality of Dirichlet characters modulo $q_\ell$,
\begin{equation}\label{eq:qellOrthoApp}
\frac{1}{\phi(q_\ell)}\sum_{\chi \pmod{q_\ell}} \bar{\chi}(b) \sum_{\ss{n \leq x_k \\ n \equiv a \pmod{\ell}}} f^{\chi}(n) \geq \frac{\delta x_k}{4\ell}.
\end{equation}
By Remark \ref{rem:qellProps}, 
$$
\phi(q_\ell) \geq q_\ell^{0.99} \geq (\log \ell)^{0.49} \text{ for } \ell \geq \ell_0(\delta).
$$ 
Therefore, the contribution from the principal character modulo $q_\ell$ to \eqref{eq:qellOrthoApp} is
$$
\leq \frac{1}{\phi(q_\ell)} |\{n \leq x_k : n \equiv a \pmod{\ell}\}| \leq \frac{\delta^2 x_k}{\ell}.
$$
Now set
$$
\mc{X}_\ell := \{\chi \pmod{q_\ell} : \chi \neq \chi_0, \, \eqref{eq:lowBdChar} \text{ holds}\}.
$$
If $\delta$ is sufficiently small then by the triangle inequality and the bound $\ell \leq \delta^2 x_k$,
\begin{align*}
\frac{\delta x_k}{4\ell} &\leq \frac{\delta^2 x_k}{\ell} + \frac{1}{\phi(q_{\ell})} \sum_{\ss{\chi \pmod{q_{\ell}} \\ \chi \notin \mc{X}_\ell \cup \{\chi_0\}}} \left|\sum_{\ss{n \leq x_k \\ n \equiv a \pmod{\ell}}} f^{\chi}(n)\right| + \frac{1}{\phi(q_\ell)} \sum_{\chi \in \mc{X}_\ell} \left|\sum_{\ss{n \leq x_k \\ n \equiv a \pmod{\ell}}} f^{\chi}(n)\right| \\
&< \frac{x_k}{\ell}\left(\delta^2 + \frac{\delta}{10}\right) + \frac{|\mc{X}_\ell|}{\phi(q_\ell)}\left(\frac{x_k}{\ell} + 1\right) \leq \frac{x_k}{\ell}\left(\frac{\delta}{8} + \frac{|\mc{X}_\ell|}{\phi(q_\ell)}\right)
\end{align*}
This conclusion follows upon rearranging.
\end{proof}
\begin{lem}\label{lem:charDist}
Let $\ell \in \mc{G}_{\delta,k;2b}$ be a prime satisfying $\max\{|a|,\ell_0(\delta)\} < \ell \leq x_k^{1/A}$ for some 
$A \geq \delta^{-10}$. Then for each $\chi \in \mc{X}_\ell$ 
$$
\min_{\psi \pmod{\ell}} \min_{|t| \leq \log x_k} \mb{D}(f^{\chi},\psi(n)n^{it};x_k)^2 \leq 1.1\log(1/\delta).
$$
\end{lem}
\begin{proof}
Let $\chi \in \mc{X}_\ell$ and enumerate the characters $\{\psi \pmod{\ell}\} = \{\psi_{\chi,1},\psi_{\chi,2},\ldots,\psi_{\chi,\ell-1}\}$ so that
$$
\min_{|t| \leq \log x_k} \mb{D}(f^\chi, \psi_{\chi,j}(n)n^{it};x_k) \leq \min_{|t| \leq \log x_k} \mb{D}(f^\chi, \psi_{\chi,j+1}(n)n^{it};x_k), \quad 1 \leq j < \ell-1. 
$$
In particular,
$$
\min_{\psi \pmod{\ell}} \min_{|t| \leq \log x_k} \mb{D}(f^\chi, \psi(n)n^{it};x_k)^2 = \min_{|t| \leq \log x_k} \mb{D}(f^\chi, \psi_{\chi,1}(n)n^{it};x_k)^2 =: M_{\chi}. 
$$
By \cite[Thm. 1.8]{GHS}, we have
$$
\left|\sum_{\ss{n \leq x_k \\ n \equiv a \pmod{\ell}}} f^{\chi}(n) - \frac{\psi_{\chi,1}(a)}{\ell-1} \sum_{n \leq x_k} f^\chi(n)\bar{\psi}_{\chi,1}(n)\right| \ll \frac{x_k}{\ell}\frac{\log A}{A^{1-1/\sqrt{2}}} \ll \delta^2 \frac{x_k}{\ell}.
$$
Writing $\psi_{\chi} = \psi_{\chi,1}$ for convenience, we combine this with \eqref{eq:lowBdChar} and $\ell \nmid a$ to obtain
$$
\left|\sum_{n \leq x_k} f^{\chi}(n)\bar{\psi}_{\chi}(n)\right| \geq \frac{\delta x_k}{20},
$$
for small enough $\delta$. Applying Theorem \ref{thm:HMT} with $T = \log x_k$, the left-hand side of the previous estimate is
$$
\ll x_k(1+M_{\chi})e^{-M_{\chi}} + \frac{x_k \log\log x_k}{\log x_k}.
$$
We deduce upon rearranging that when $\delta$ is sufficiently small, $M_{\chi} \leq 1.1\log(1/\delta)$, as claimed.
\end{proof}
From now on, given $\chi \pmod{q_\ell}$ we select a character $\psi_{\chi}$ modulo $\ell$ that satisfies
$$
\min_{\psi \pmod{\ell}} \min_{|t| \leq \log x_k} \mb{D}(f^{\chi},\psi(n)n^{it};x_k) = \min_{|t| \leq \log x_k} \mb{D}(f^{\chi},\psi_{\chi}(n)n^{it};x_k),
$$
and refer to it as the \emph{minimising character} for $\chi$. 
(If more than one such character exists modulo $\ell$ for a given $\chi$ we choose any of them.)
To prove Proposition \ref{prop:almostHom} we will show the following.
\begin{lem} \label{lem:groupLem}
Let $\delta \in (0,1/2)$ be small.
Let $A \geq \delta^{-10}$, and let $\ell \in \mc{G}_{\delta,k;2b}$ be a prime with $\max\{|a|, \ell_0(\delta) \} < \ell \leq x_k^{1/A}$. Then there is a (cyclic) subgroup $H_\ell = \langle \tilde{\chi}\rangle$ of the character group $\pmod{q_\ell}$ satisfying the following properties:
\begin{enumerate}[(i)]
\item $|H_{\ell}| \gg_{\delta} \phi(q_{\ell})$
\item the minimising character $\psi = \psi_{\tilde{\chi}}$ for $\tilde{\chi}$ satisfies $\text{ord}(\psi) \mid |H_{\ell}|$
\item we have
\begin{equation}\label{eq:unifBd}
\max_{0 \leq j \leq |H_\ell|-1} \mb{D}(f^{\tilde{\chi}^j}, \psi^j;x_k) \ll_{\delta} 1.
\end{equation}
\end{enumerate}
\end{lem}
The proof of Lemma \ref{lem:groupLem} relies on the following inverse sumset lemma.
\begin{lem}\label{lem:findLowOrder}
Let $\eta > 0$ and $M \geq 1$. Let $G = \mb{Z}/M\mb{Z}$ be the cyclic group of order $M$, and let $B \subseteq G$ be a symmetric\footnote{We say that a subset $B$ of a cyclic group $\mb{Z}/N\mb{Z}$ is \emph{symmetric} when $b \in B$ if and only if $-b \in B$.} subset with $|B| \geq \eta M$. Then there is an integer $m = O_{\eta}(1)$ and a $d|M$ with $d \geq \eta M$ such that $mB = H$ is a subgroup of $G$ with $|H| = d$.
\end{lem}
\begin{proof}
Consider the sequence of sumsets
$(2^iB)_{i \geq 1}$, and let $j \geq 0$ be minimal such that 
$$
|2^{j+1}B| = |2^j B + 2^j B| < \frac{3}{2}|2^j B|.
$$
Note that as $B$ is symmetric, so is $2^j B$. Thus, by a lemma due to Freiman \cite{Fre}, $2^{j+1} B$ is in fact a subgroup of $G$. Since $|2^j B| \leq |G|$ we find
$$
M \geq |2^j B| \geq \left(\frac{3}{2}\right)^{j-1} |B| \geq \left(\frac{3}{2}\right)^{j-1} \eta M,
$$
so that $j = O(\log(2/\eta))$. Taking $m = 2^{j+1} = O_{\eta}(1)$, it follows that $mB = H$ for some subgroup $H \leq G$. \\
Now $G$ is cyclic, so $H$ must be as well. By Lagrange's theorem, $d = |H|$ divides $M$. Furthermore, as $|H| = |m B| \geq |B| \geq \eta M$ we find $d \geq \eta M$, and the proof is complete.
%
%
\end{proof}
As in the proof of \cite[Prop. 3.3]{ManHighOrd}, we will make crucial use of the following.
\begin{lem}[Ruzsa] \label{lem:RuzGenG}
Let $(G,\cdot)$ be an abelian group, and let $K \geq 0$. Let $\phi: G \ra \mb{R}$ be a map satisfying
$$
|\phi(x\cdot y)-\phi(x)-\phi(y)| \leq K \text{ for all } x,y \in G.
$$
Then there is a genuine homomorphism $\psi: G \ra \mb{R}$ such that
$$
\sup_{x \in G} |\phi(x)-\psi(x)| \leq K.
$$
\end{lem}
\begin{proof}
This is\cite[Statement (7.3)]{Ruz}.
\end{proof}

\begin{proof}[Proof of Lemma \ref{lem:groupLem}]
Let $\max\{|a|, \ell_0(\delta)\} < \ell \leq x_k^{1/A}$ be a prime with $\ell \in \mc{G}_{\delta,2b;k}$. Since $A \geq \delta^{-10}$, Lemma \ref{lem:charDist} shows that there is a set $\mc{X}_{\ell}$ of non-principal characters $\chi$ modulo $q_{\ell}$ for which there is a choice of minimising\footnote{If there are several possible choices of $t_{\chi}$ pick the one of minimal absolute value.} $t_{\chi} \in [-\log x_k,\log x_k]$ and $\psi_{\chi} \pmod{\ell}$ for which
$$
\mb{D}(f^\chi,\psi_{\chi}(n)n^{it_{\chi}};x_k)^2 \leq 1.1\log(1/\delta).
$$
Let $\chi_{q_{\ell}}$ be a generator for the (cyclic) group of characters modulo $q_{\ell}$, and for each $\chi \pmod{q_{\ell}}$ let $m_{\chi} \in \mb{Z}/\phi(q_\ell) \mb{Z}$ be such that $\chi = \chi_{q_{\ell}}^{m_{\chi}}$. Below, given a set $\Xi$ of characters modulo $q_\ell$ we define the \emph{index set of $\Xi$} to be
$$
\mc{I}(\Xi) := \{m_{\chi} : \chi \in \Xi\} \subset \mb{Z}/\phi(q_\ell) \mb{Z}.
$$
Let now 
$$
\tilde{\mc{X}_\ell} := \{\chi_1\bar{\chi}_2 : \chi_1,\chi_2 \in \mc{X}_{\ell}\}.
$$
We observe that if $\chi_1\bar{\chi}_2 \in \tilde{\mc{X}_\ell}$, with $\chi_j \in \mc{X}_{\ell}$ for $j = 1,2$, then as $f^{\chi_1\bar{\chi}_2} = f^{\chi_1} \bar{f^{\chi_2}}$ the pretentious triangle inequality yields
\begin{align} \label{eq:tildeXellprop}
\min_{\psi \pmod{\ell}} \min_{|t| \leq 2\log x_k} \mb{D}(f^{\chi_1\bar{\chi}_2}, \psi(n) n^{it};x_k) &\leq \mb{D}(f^{\chi_1\bar{\chi}_2}, \psi_{\chi_1}\bar{\psi}_{\chi_2}(n)n^{i(t_{\chi_1}-t_{\chi_2})};x_k) \nonumber\\
&\leq \sum_{j = 1,2} \mb{D}(f^{\chi_j}, \psi_{\chi_j}(n)n^{it_j};x_k) \nonumber \\
&\leq 2.2\sqrt{ \log(1/\delta)}.
\end{align}
We extend the definitions of $(\psi_{\chi},t_{\chi})$ to all $\chi \in \tilde{\mc{X}_{\ell}} \bk \mc{X}_{\ell}$ by selecting these to satisfy
$$
\min_{\psi \pmod{\ell}} \min_{t_{\chi} \in [-2\log x_k,2\log x_k]} \mb{D}(f^{\chi},\psi(n)n^{it};x_k) = \mb{D}(f^{\chi},\psi_{\chi}(n)n^{it_{\chi}};x_k).
$$
Now, consider the index set $\mc{I}(\tilde{\mc{X}_\ell})$. Clearly, $0 \in \mc{I}(\tilde{\mc{X}_\ell})$ and
$$
m \in \mc{I}(\tilde{X}_{\ell}) \text{ if and only if } -m \in \mc{I}(\tilde{X}_{\ell}).
$$ 
By Lemma \ref{lem:findLowOrder},
there is a positive integer $r_{\ell} = O_{\delta}(1)$ and a subgroup $H_{\ell}$ of the group of characters $\pmod{q_\ell}$ with $|H_{\ell}| \gg_{\delta} \phi(q_{\ell})$ such that
\begin{equation}\label{eq:sumsetCover}
\mc{I}(H_{\ell}) \subseteq r_{\ell} \mc{I}(\tilde{\mc{X}_\ell}) = \{a_1 + \cdots + a_{r_{\ell}} : a_j \in \mc{I}(\tilde{X}_{\ell}) \text{ for } 1 \leq j \leq r_\ell\}.
\end{equation}
Set $d_{\ell} := |H_{\ell}|$. There is a choice of generator $\tilde{\chi}$ for $H_{\ell}$ such that $m_{\ell} := m_{\tilde{\chi}}$ satisfies $m_{\ell} = \phi(q_{\ell})/d_{\ell} = O_{\delta}(1)$, and
$$
H_{\ell} = \{\tilde{\chi}^j : \, 0 \leq j \leq d_\ell-1\} = \{\chi_{q_{\ell}}^{jm_{\ell}} : \, 0 \leq j \leq d_\ell-1\}.
$$
Together with \eqref{eq:sumsetCover}, the above parametrisation of $H_{\ell}$ implies that for each $0 \leq j \leq d_{\ell}-1$ we have a representation
$$
j m_{\ell} = m_{\chi_1} + \cdots + m_{\chi_{r_{\ell}}}, \text{ with } m_{\chi_i} \in \mc{I}(\tilde{\mc{X}_\ell}) \text{ for each } 1 \leq i \leq r_{\ell}.
$$
This immediately implies that
\begin{equation}\label{eq:tildechijRep}
\tilde{\chi}^j = \chi_1 \cdots \chi_{r_{\ell}}, \quad \chi_i \in \tilde{\mc{X}}_\ell.
\end{equation}
Combining \eqref{eq:tildechijRep} with \eqref{eq:multTri} and \eqref{eq:tildeXellprop}, we find
\begin{align}
\min_{\psi \pmod{\ell}} \min_{|t| \leq 2r_{\ell} \log x_k} \mb{D}(f^{\tilde{\chi}^j}, \psi(n)n^{it};x_k) &\leq \mb{D}(f^{\chi_1} \cdots f^{\chi_{r_\ell}}, \psi_{\chi_1}\cdots \psi_{\chi_{r_\ell}}(n) n^{i(t_{\chi_1} + \cdots + t_{\chi_{r_\ell}})};x_k) \nonumber\\
&\leq \sum_{1 \leq s \leq r_{\ell}} \mb{D}(f^{\chi_s}, \psi_{\chi_s}(n) n^{it_{\chi_s}};x_k) \nonumber\\
&\leq r_{\ell} \max_{1 \leq s \leq r_{\ell}} \mb{D}(f^{\chi_s}, \psi_{\chi_s}(n) n^{it_{\chi_s}};x_k) \leq 2.2 r_{\ell} \sqrt{\log(1/\delta)}, \label{eq:onHl}
\end{align}
uniformly in $0 \leq j \leq d_{\ell}-1$. \\
Finally, we extend the definitions of $\psi_{\chi}$ and $t_{\chi}$ to the remaining characters in $\Xi_\ell \bk \tilde{\mc{X}_{\ell}}$ so that
$$
\mb{D}(f^\chi,\psi_{\chi}(n)n^{it_{\chi}};x_k) = \min_{\psi \pmod{\ell}}\min_{|t| \leq 4r_{\ell}\log x_k}  \mb{D}(f^\chi,\psi(n)n^{it};x_k).
$$
We will show that the maps $\gamma_{\ell} : H_{\ell} \ra \Xi_\ell$ and $\theta_{\ell} : H_{\ell} \ra \mb{R}$ given by 
$$
\gamma_{\ell}(\chi) := \psi_{\chi}, \quad \theta_{\ell}(\chi) := t_{\chi}\log x_k,
$$
have the following properties: 
\begin{enumerate}[(a)]
\item $\gamma_{\ell}$ is a homomorphism on $H_{\ell}$, and 
\item there is $K = O_{\delta}(1)$ such that $\theta_{\ell}$ is a \emph{$K$-approximate homomorphism}, i.e., such that
$$
\max_{\chi_1,\chi_2 \in H_{\ell}} |\theta_{\ell}(\chi_1\chi_2)-\theta_\ell(\chi_1)\theta_\ell(\chi_2)| \leq K.
$$
\end{enumerate}
Note that \eqref{eq:onHl} implies that 
\begin{equation}\label{eq:bdforHell}
\max_{\chi \in H_{\ell}} \mb{D}(f^\chi,\psi_{\chi}(n)n^{it_{\chi}};x_k) \leq 2.2 r_{\ell}\sqrt{\log(1/\delta)}.
\end{equation}
If $\tilde{\chi}^j$ has the representation \eqref{eq:tildechijRep} then, uniformly over all $0 \leq j \leq d_{\ell}-1$,
\begin{align*}
&\mb{D}(\psi_{\tilde{\chi}^j} \bar{\psi}_{\chi_1} \cdots \bar{\psi}_{\chi_{r_{\ell}}}, n^{i(t_{\tilde{\chi}^j} - t_{\chi_1} - \cdots - t_{\chi_{r_{\ell}}})};x_k) \\
&\leq  \mb{D}(f^{\tilde{\chi}^j}, \psi_{\tilde{\chi}^j}(n) n^{it_{\tilde{\chi}^j}};x_k) + \mb{D}(f^{\chi_1} \cdots f^{\chi_{r_{\ell}}},\psi_{\chi_1} \cdots \psi_{\chi_{r_{\ell}}} n^{i(t_{\chi_1} + \cdots + t_{\chi_{r_{\ell}}})};x_k) \\
&\leq 2 r_{\ell} \max_{1 \leq s \leq {r_{\ell}}} \mb{D}(f^{\chi_s},\psi_{\chi_s}(n) n^{it_{\chi_s}};x_k) \leq 4.4 r_{\ell}\sqrt{\log(1/\delta)}.
\end{align*}
Since $r_{\ell} = O_{\delta}(1)$, Lemma \ref{lem:charnit} implies that for $\eta$ sufficiently small in terms of $\delta$ and $\ell \leq x^{1/\eta}$ we have
$$
\psi_{\tilde{\chi}^j} = \psi_{\chi_1} \cdots \psi_{\chi_{r_{\ell}}},
$$
and moreover
$$
t_{\tilde{\chi}^j} = t_{\chi_1} + \cdots + t_{\chi_{r_{\ell}}} + O_{\delta}\left(\frac{1}{\log x_k}\right).
$$
Similarly, if $\xi_1,\xi_2 \in H_{\ell}$ then 
by \eqref{eq:onHl}, \eqref{eq:bdforHell}, the minimal property of $\psi_{\xi_1\xi_2}$ and $t_{\xi_1\xi_2}$, and \eqref{eq:usualTri} again,
\begin{align*}
\mb{D}(\psi_{\xi_1\xi_2}\bar{\psi}_{\xi_1}\bar{\psi}_{\xi_2}, n^{i(t_{\xi_1+\xi_2} - t_{\xi_1} - t_{\xi_2})}; x_k) &\leq \mb{D}(f^{\xi_1\xi_2}, \psi_{\xi_1\xi_2}(n)n^{it_{\xi_1\xi_2}};x_k) + \mb{D}(f^{\xi_1\xi_2}, \psi_{\xi_1}\psi_{\xi_2}(n)n^{i(t_{\xi_1} + t_{\xi_2})};x_k)\\
&\leq 2\mb{D}(f^{\xi_1\xi_2}, \psi_{\xi_1}\psi_{\xi_2}(n)n^{i(t_{\xi_1} + t_{\xi_2})};x_k)\\
&\leq 4 \max_{s = 1,2} \mb{D}(f^{\xi_s},\psi_{\xi_s}(n) n^{it_{\xi_s}};x_k) \leq 8.8 r_{\ell}\sqrt{\log(1/\delta)}.
\end{align*}
Thus, since $|t_{\xi_1\xi_2}-t_{\xi_1}-t_{\xi_2}| \leq 12r_{\ell}\log x_k$ and $r_\ell = O_{\delta}(1)$, again by Lemma \ref{lem:charnit} we obtain that whenever $\chi_1,\chi_2 \in H_{\ell}$, 
\begin{align*}
\gamma_{\ell}(\chi_1\chi_2) &= \psi_{\chi_1 \chi_2} = \psi_{\chi_1}\psi_{\chi_2} = \gamma_{\ell}(\chi_1)\gamma_{\ell}(\chi_2), \\
\theta_\ell(\chi_1\chi_2) &= t_{\chi_1\chi_2}\log x_k = (t_{\chi_1} + t_{\chi_2})\log x_k + O_{\delta}(1) \\
&= \theta_{\ell}(\chi_1) + \theta_{\ell}(\chi_2) + O_{\delta}(1).
\end{align*}
Thus, $\gamma_\ell$ is a homomorphism on $H_{\ell}$, while $\theta_{\ell}$ is a $K$-approximate homomorphism from $H_{\ell}$ to $\mb{R}$ for some $K = O_{\delta}(1)$, confirming (a) and (b) above. \\
Note that there are no non-zero homomorphisms from a finite abelian group to $\mb{R}$. By Lemma \ref{lem:RuzGenG},
(b) therefore implies 
$$
\max_{\chi \in H_\ell} |t_{\chi}| = \max_{\chi \in H_\ell} \left|\frac{\theta_{\ell}(\chi)}{\log x_k}\right| \leq \frac{K}{\log x_k}.
$$
Taylor expanding $p^{it_{\chi}}$ for $p \leq x_k^{1/(3K)}$, we find that for any $\chi \in H_{\ell}$,
$$
\mb{D}(1,n^{it_{\chi}};x_k^{1/(3K)})^2 \leq \sum_{p \leq x_k^{1/(3K)}} \frac{|1-p^{it_{\chi}}|}{p} \ll \frac{K}{\log x_k} \sum_{p \leq x_k^{1/(3K)}} \frac{\log p}{p} = O_{\delta}(1).
$$
Combining this with \eqref{eq:usualTri} and \eqref{eq:bdforHell},
\begin{equation}\label{eq:unifDfchipsichi}
\mb{D}(f^\chi, \psi_{\chi}(n); x_k)^2 \leq 2\left(\mb{D}(f^{\chi},\psi_{\chi}(n)n^{it_{\chi}};x_k^{1/(3K)})^2 + \mb{D}(1,n^{it_{\chi}};x_k^{1/(3K)})^2\right) + \sum_{x_k^{1/3K} \leq p \leq x_k} \frac{2}{p} \ll_{\delta} 1,
\end{equation}
uniformly over $\chi \in H_{\ell}$. \\
We can now complete the proof. By construction, $|H_{\ell}| \gg_{\delta} \phi(q_{\ell})$, so (i) holds. Next, as $H_{\ell}$ is cyclic, $\psi_{\tilde{\chi}} = \gamma_{\ell}(\tilde{\chi})$ generates the subgroup $\gamma_{\ell}(H_\ell)$ of $\Xi_\ell$. Since
$$
\psi_{\tilde{\chi}}^{d_\ell} = \gamma_{\ell}(\tilde{\chi})^{d_{\ell}} = \gamma_{\ell}(\tilde{\chi}^{d_\ell})
$$
is the trivial character modulo $\ell$, it follows that $\psi = \psi_{\tilde{\chi}}$ has order dividing $d_{\ell}$, verifying (ii). Finally, as each $\chi \in H_{\ell}$ is of the form $\tilde{\chi}^j$ for some $0 \leq j \leq d_\ell-1$, we have 
$$
\psi_{\chi} = \gamma(\tilde{\chi})^j = \psi^j,
$$
and so
$$
\max_{\chi \in H_{\ell}} \mb{D}(f^\chi,\psi_{\chi};x_k) = \max_{0 \leq j \leq d_{\ell}-1} \mb{D}(f^{\tilde{\chi}^j}, \psi^j;x_k).
$$
Thus, (iii) follows from \eqref{eq:unifDfchipsichi}.
\end{proof}

\begin{proof}[Proof of Proposition \ref{prop:almostHom}]
Let $A := \delta^{-10}$. Let $\ell \in \mc{G}_{\delta,k;2b}$ be a prime with $\max\{|a|, \ell_0(\delta) \} < \ell \leq x_k^{1/A}$. Let $\chi_{q_\ell}$ be a generator of the cyclic group $\Xi_{q_\ell}$. By Lemma \ref{lem:groupLem}, we find a subgroup $H_{\ell} = \langle \tilde{\chi}\rangle \leq \Xi_{q_\ell}$ with $|H_{\ell}| \gg_{\delta} \phi(q_{\ell})$ on which \eqref{eq:unifBd} holds. 
A choice of generator $\tilde{\chi} = \chi_{q_\ell}^{m_\ell}$ can be made 
so that $m_\ell = \phi(q_\ell)/|H_\ell| = O_{\delta}(1)$. 
By \eqref{eq:unifBd}, 
\begin{equation}\label{eq:avgUnifEst}
\sum_{p \leq x_k} \frac{1}{p} \text{Re}\left(1-\frac{1}{|H_\ell|} \sum_{0 \leq j \leq |H_\ell|-1} (\tilde{\chi}(f(p))\bar{\psi}(p))^j\right) = \frac{1}{|H_\ell|} \sum_{0 \leq j \leq |H_\ell|-1} \mb{D}(f^{\tilde{\chi}^j},\psi^j; x_k)^2 \ll_{\delta} 1.
\end{equation}
By Lemma \ref{lem:groupLem}(ii), $\psi^{|H_{\ell}|}$ is principal. Thus, for $p \leq x_k$ we have
$$
\frac{1}{|H_{\ell}|}\sum_{0 \leq j \leq |H_\ell|-1} (\tilde{\chi}(f(p))\bar{\psi}(p))^j = 1_{\tilde{\chi}(f(p))\bar{\psi}(p) = 1}.
$$
Inserting this into \eqref{eq:avgUnifEst}, we find
$$
\sum_{\ss{p \leq x_k \\ \tilde{\chi}(f(p)) \neq \psi(p)}} \frac{1}{p} = \sum_{p \leq x_k} \frac{1}{p}(1-1_{\tilde{\chi}(f(p)) = \psi(p)}) \ll_{\delta} 1.
$$
Note that $\psi$ takes values in the set $\mu_{\phi(q_{\ell})}$ of roots of unity of order $\phi(q_\ell)$. As $\chi_{q_\ell}$ is a bijection from $(\mb{Z}/q_\ell \mb{Z})^\times$ onto that set, there is an induced map 
$$
\rho_\ell: (\mb{Z}/\ell \mb{Z})^\times \to (\mb{Z}/q_\ell\mb{Z})^\times
$$
such that 
$$
\chi_{q_\ell}(\rho_\ell(a)) = \psi(a) \text{ for all } a \in (\mb{Z}/\ell \mb{Z})^\times.
$$
It is easily checked that $\rho_{\ell}$ is a homomorphism: since $\chi_{q_{\ell}}$ is a bijection it suffices to show that
$$
\chi_{q_\ell}(\rho_{\ell}(ab)) = \chi_{q_{\ell}}(\rho_{\ell}(a)\rho_{\ell}(b)) \text{ for any } a,b \in (\mb{Z}/\ell \mb{Z})^\times,
$$
and by definition, 
$$
\chi_{q_\ell}(\rho_\ell(ab)) = \psi(ab) = \psi(a)\psi(b) = \chi_{q_\ell}(\rho_\ell(a))\chi_{q_\ell}(\rho_\ell(b)) = \chi_{q_\ell}(\rho_\ell(a)\rho_\ell(b)).
$$
Finally, set 
$$
\mc{T}_\ell := \{p \leq x_k : \tilde{\chi}(f(p)) \neq \psi(p)\}
$$ 
and let $p \leq x_k$ with $p \notin \mc{T}_\ell$. As $\tilde{\chi} = \chi_{q_\ell}^{m_\ell}$, we deduce that 
$$
\text{if } p \equiv a \pmod{\ell} \text{ then } f(p)^{m_\ell} \equiv \rho_\ell(a) \pmod{q_\ell}.
$$ 
Since $m_\ell = O_{\delta}(1)$ uniformly over all $\ell \in \mc{G}_{\delta,k;2b} \cap [2,x_k^{1/A}]$ (sufficiently large relative to $\delta$ and $|a|$), there is an integer $M = O_{\delta}(1)$, independent of $k$, such that $m_\ell | M!$ for all such $\ell$.
If we set $D := 2M!$ and define the homomorphism $\tilde{\rho}_{\ell}(a) := \rho_{\ell}(a)^{D/m_{\ell}} \pmod{q_{\ell}}$, we obtain the conclusion \eqref{eq:genHomCond} uniformly over all $\ell \in \mc{G}_{\delta,k;2b} \cap [2,x_k^{1/A}]$ sufficiently large, as required.\\
Suppose now that $q_{\ell}$ is a power of $\ell$, for $\ell \in \mc{G}_{\delta,k;2b}$.
Composing the homomorphism $\tilde{\rho}_\ell$ with the projection $(\mb{Z}/q_{\ell}\mb{Z})^{\times} \twoheadrightarrow (\mb{Z}/\ell\mb{Z})^{\times}$, we obtain an endomorphism
$$
\rho^\ast_{\ell} : (\mb{Z}/\ell \mb{Z})^\times \to (\mb{Z}/\ell \mb{Z})^\times.
$$ 
Since $(\mb{Z}/\ell \mb{Z})^\times$ is cyclic,
select $u$ to be a primitive root modulo $\ell$ and define $0 \leq g_\ell < \ell-1$ implicitly from the formula $\rho^\ast_{\ell}(u) \equiv u^{g_{\ell}} \pmod{\ell}$. Then whenever $p \notin \mc{T}_{\ell}$, $p \neq \ell$, and $p \equiv u^{\theta_p} \pmod{\ell}$ we find
$$
f(p)^D \equiv \rho^\ast_{\ell}(u^{\theta_p}) \pmod{\ell} \equiv \rho^\ast_{\ell}(u)^{\theta_p} \pmod{\ell} \equiv u^{g_{\ell}\theta_p } \pmod{\ell} \equiv p^{g_{\ell}} \pmod{\ell}, 
$$
and \eqref{eq:ellDivfell} follows as well.
\end{proof}
\section{Local and global power maps} \label{sec:LocGlob}
\subsection{The Fabrykowski-Subbarao Conjecture and a result of N. Jones}
Let $\ell$ be a prime. Recall that a a map $F: \mb{N} \ra \mb{Z}$ is said to be a \emph{local power map} modulo $\ell$ if there is an integer $0 \leq g_{\ell} < \ell-1$ such that 
$$
F(n) \equiv n^{g_{\ell}} \pmod{\ell} \text{ for all } n \in \mb{N}. 
$$
If $F$ is a completely multiplicative function then $F$ is a local power map modulo $\ell$ if and only if
\begin{equation} \label{eq:LPMPrimes}
F(p) \equiv p^{g_{\ell}} \pmod{\ell} \text{ for all } p \in \mc{P}.
\end{equation}
For comparison, the conclusion of Proposition \ref{prop:almostHom} 
states that the map $F(n) = f(n)^D$ satisfies 
\begin{equation}\label{eq:LPMApprox}
F(p) \equiv p^{g_{\ell}} \pmod{\ell} \text{ for all } p \leq x_k, \, p \notin \mc{T}_{\ell},
\end{equation}
for each $\ell \in \mc{G}_{\delta,k;2b} \cap [\ell_0,x_k^{1/A}]$ with $\ell |f(\ell)$. That is, $f^D$ behaves like a local power map at all such primes $\ell$, with the complications arising from the exceptional set $\mc{T}_{\ell}$ and the potential dependence of $g_{\ell}$ on the scale $x_k$. \\
Fabrykowski and Subbarao conjectured that a multiplicative function that is a local power map modulo infinitely many primes must be a global power map, i.e., a monomial. 
\begin{conj}[Fabrykowski-Subbarao, \cite{FabSub}] \label{conj:FS}
Suppose $F: \mb{N} \ra \mb{Z}$ is a multiplicative function such that for infinitely many primes $\ell$,
$$
F(n+\ell) \equiv F(n) \pmod{\ell} \text{ for all $n \in \mb{N}$.}
$$
Then 
there is a non-negative integer $k$ such that $F(n) = n^k$ for all $n \in \mb{N}$.
\end{conj}
\begin{lem} \label{lem:red1}
Conjecture \ref{conj:FS} is equivalent to the following statement: \\ 
Let $F: \mb{N} \ra \mb{Z}$ be a completely multiplicative function for which there is an infinite sequence of primes $S_F$ with the property that
for each $\ell \in S_F$ there is an integer $0 \leq g_\ell < \ell-1$ such that 
\begin{equation}\label{eq:locPow}
F(n) \equiv n^{g_\ell} \pmod{\ell} \text{ for all } n \in \mb{N}.
\end{equation}
Then there is a non-negative integer $k$ such that $F(n) = n^k$ for all $n \in \mb{N}$.
\end{lem}
\begin{proof}
Conjecture \ref{conj:FS} trivially implies this conjecture, so it suffices to show that if $f$ satisfies the hypotheses of Conjecture \ref{conj:FS} then it must be completely multiplicative and satisfy \eqref{eq:locPow} for some $0 \leq g_\ell < \ell-1$ and each $\ell \in S_F$. \\
For $d \in \mb{N}$ let $S_F(d):= \{\ell \in S_F : \ell \nmid d\}$, and let $m,n \in \mb{N}$.  For each $\ell \in S_F(mn)$ select distinct primes $q_m,q_n$ such that $q_k \equiv k \pmod{\ell}$ for each $k \in \{m,n\}$; such primes exist by Dirichlet's theorem on primes in arithmetic progressions. It follows that $q_m,q_n$ are coprime, and moreover $mn \equiv q_mq_n \pmod{\ell}$. Thus, $F(mn) \equiv F(q_mq_n) \pmod{\ell}$, and moreover by multiplicativity we get
$$
F(q_mq_n) = F(q_m)F(q_n).
$$
This implies that
$$
F(mn) \equiv F(q_mq_n) \pmod{\ell} \equiv F(q_m)F(q_n) \pmod{\ell} \equiv F(m)F(n) \pmod{\ell}.
$$
We thus obtain that
$$
\ell|(F(mn) - F(m)F(n)) \text{ for all } \ell \in S_F(mn).
$$
Since $S_F$ is an infinite set, so is $S_F(mn)$, and we deduce that $F(mn) = F(m)F(n)$. As $m,n$ were arbitrary, $F$ must be completely multiplicative.\\
Next, fix $\ell \in S_F$ and let $\chi_\ell$ be a generator of $\Xi_{\ell}$.
By (a), $F^{\chi_\ell}$ is a completely multiplicative function, and moreover since $F$ is periodic modulo $\ell$ we have 
$$
F^{\chi_\ell}(n+\ell) = F^{\chi_\ell}(n) \text{ for all } n \in \mb{N}.
$$
Thus, $F^{\chi_\ell}$ is completely multiplicative, periodic modulo $\ell$ and satisfies $F^{\chi_\ell}(n) = 0$ whenever $\ell|n$. By definition, $F^{\chi_\ell}$ must be a Dirichlet character modulo $\ell$.
Since $\chi_\ell$ generates the group of such characters there must therefore exist $0 \leq g_\ell < \ell-1$ such that 
$$
\chi_\ell(F(n)) = F^{\chi_\ell}(n) = \chi_\ell^{g_\ell}(n) = \chi_\ell(n^{g_\ell}) \text{ for all } n.
$$
Since $\chi_\ell$ is injective on $\mb{Z}/\ell\mb{Z}$, we deduce that $F(n) \equiv n^{g_\ell} \pmod{\ell}$ for all $n$, as required.
\end{proof}
Conjecture \ref{conj:FS} remains open, but partial results exist under the assumption that $S_F(X)$ grows somewhat quickly. In particular, N. Jones \cite{Jones} proved that if $\limsup_{X \ra \infty} S_F(X)/\pi(X) > 0$ then $F$ must be a monomial.\\
In light of the similarities between \eqref{eq:LPMApprox} and \eqref{eq:LPMPrimes}, we apply Jones' method to derive structural information about $f$ (through $f^D$) provided sufficiently many $\ell|f(\ell)$. \\
For $X \geq 2$ define
$$
S_f := \{\ell \in \mc{P} : \, \ell | f(\ell)\} \subseteq \mc{P}_f, \quad S_f(X) := |S_f \cap [2,X]|.
$$
Using Jones' method, we prove the following.
\begin{prop}\label{prop:Assumlfl}
Let $f: \mb{N} \ra \mb{Z}$ be a multiplicative function for which
$$
\sum_{p : |f(p)| \neq 1} \frac{1}{p} = \infty, \quad \sum_{p : f(p) = 0} \frac{1}{p} < \infty,
$$
and also that
$$
\lim_{X \ra \infty} \frac{1}{\log\log X} \sum_{\ss{p \leq X \\ p \in S_f}} \frac{1}{p} > 0.
$$ 
Suppose moreover that there are integers $a \geq 1$ and $b \neq 0$, and a $\delta \in (0,1/2)$ such that $\bar{\delta}(\mc{N}_{f,a,b}) > \delta$.
Then the set
\begin{equation}\label{eq:nonPrimPow}
\{\ell \in \mc{P} : \ell' \nmid f(\ell) \text{ for all } \ell' \neq \ell \}
\end{equation}
is $O_{\delta}(1)$-sparse. \\
The same conclusion holds under the assumption of the Extended Riemann Hypothesis, provided that
$$
\sum_{p \in S_f} \frac{1}{p} = \infty.
$$
\end{prop}
That is, if $\ell | f(\ell)$ for a positive Dirichlet density set of primes $\ell$ then, outside of an $O_{\delta}(1)$-sparse set, 
$f(\ell) = \ell^{\alpha_{\ell}}$ for some $\alpha_{\ell} \in \mb{N}$. \\
As a complement to Proposition \ref{prop:Assumlfl}, we classify all multiplicative functions $f$ for which $f(\ell)$ is a power of $\ell$ for all but a sparse set of $\ell$.
\begin{prop}\label{prop:veryRigid}
Assume that $f$ satisfies the hypotheses of Proposition \ref{prop:Assumlfl}, and that the set
\begin{equation} \label{eq:primepow}
\{\ell \in \mc{P} : \ell' \nmid f(\ell) \text{ for all } \ell' \neq \ell \}
\end{equation}
is sparse. Then the following properties must hold:
\begin{enumerate}[(i)]
\item $|f(\ell)| = \ell$ for all but an $O_{\delta}(1)$-sparse set of primes $\ell$
\item if $f$ is completely multiplicative then there is a divisor $d|a$ such that $f(d)|b$, and $a/d$ divides $b/f(d)$. 
\end{enumerate}
\end{prop}
\subsection{Proof of Proposition \ref{prop:Assumlfl}} \label{subsec:Prop63}
We prove Proposition \ref{prop:Assumlfl} by appealing to elements of the work \cite{Jones}. As a first step towards applying the ideas therein, we address the nature of the exceptional sets $\mc{T}_{\ell}$, for $\ell \in \mc{G}_{\delta,k;2b}$. Below, given a prime $p \leq x_k$ we set
$$
\mc{T}(p) := \{\ell \in \mc{G}_{\delta,k;2b} \cap [2,x_k] : p \in \mc{T}_{\ell}\}.
$$
We would like to show that for many primes $p$, $\mc{T}(p)$ is a \emph{small} subset of $\mc{G}_{\delta,k;2b} \cap [2,x_k]$, as then we can deduce information on the behaviour of $f(p)^D \pmod{\ell}$ for many primes $\ell \in \mc{G}_{\delta,k;2b}$. This is implied by the following simple lemma.
\begin{lem} \label{lem:controlExcep}
With the notation of Proposition \ref{prop:almostHom}, let $y_k := x_k^{1/A}$, let $\mc{U} \subseteq \mc{G}_{\delta,k;2b} \cap [\ell_0,y_k]$ be a non-empty set of primes and put $\mc{U}(x) := \mc{U} \cap [2,x]$. Let $2 \leq P \leq x_k$ and let $\tau \in (0,1)$.
Then the set of primes $p \leq P$ for which
\begin{equation}\label{eq:TpBig}
\sum_{\ell \in \mc{T}(p) \cap \mc{U}(y_k)} \frac{1}{\ell} > \tau \sum_{\ell \in \mc{U}(y_k)} \frac{1}{\ell}
\end{equation}
holds satisfies 
$$
\sum_{\ss{p \leq P \\ \eqref{eq:TpBig} \text{ holds}}} \frac{1}{p} \ll_{\delta} \tau^{-1}.
$$
\end{lem}
\begin{proof}
By Proposition \ref{prop:almostHom} we have
$$
\max_{\ss{\ell \in \mc{G}_{\delta,k;2b} \\ \ell_0 \leq \ell \leq y_k}} \sum_{\ss{p \leq x_k \\ p \in \mc{T}_{\ell}}} \frac{1}{p} \ll_{\delta} 1.
$$
Thus, by Markov's inequality 
$$
\sum_{\ss{p \leq P \\ \eqref{eq:TpBig} \text{ holds}}} \frac{1}{p} \leq \tau^{-1} \left(\sum_{\ell \in \mc{U}(y_k)} \frac{1}{\ell}\right)^{-1}\sum_{p \leq P} \frac{1}{p} \sum_{\ss{\ell \in \mc{U}(y_k) \\ \ell \in \mc{T}(p)}} \frac{1}{\ell} = \tau^{-1} \left(\sum_{\ell \in \mc{U}(y_k)} \frac{1}{\ell}\right)^{-1} \sum_{\ss{\ell \in \mc{U}(y_k)}} \frac{1}{\ell} \sum_{\ss{p \leq P \\ p \in \mc{T}_\ell}} \frac{1}{p} \ll_{\delta} \tau^{-1},
$$
as claimed.
\end{proof}
We will also need the following well-known result.
\begin{lem}[Brun-Titchmarsh Inequality] \label{lem:BT}
Let $x \geq 100$ and let $1 \leq q < x/50$. Then
$$
\max_{\ss{a \pmod{q} \\ (a,q) = 1}} \pi(x;q,a) \ll \frac{x}{\phi(q)\log(x/q)}.
$$
\end{lem}
\begin{proof}
This follows from \cite[Thm. 6.6]{IK}.
\end{proof}

We will also use a ``Brun-Titchmarsh'' variant of the Chebotarev density theorem, due to Lagarias, Montgomery and Odlyzko \cite{LMO}. To state this we recall some algebraic preliminaries. \\
Given a Galois number field $L/\mb{Q}$ and a rational prime $p$ that does not ramify in $L$, let $\mf{p} \subseteq L$ be an ideal lying above $p$. Let $\mc{O}_L$ denote the ring of integers of $L$, and set $F_{\mf{p}} := \mc{O}_L/\mf{p}\mc{O}_L$, $F_p := \mb{Z}/p\mb{Z}$ as the respective residue fields of $\mf{p}$ in $L$ and $p$ in $\mb{Q}$. Writing $D_{\mf{p}}$ to denote the decomposition group of $\mf{p}$, we recall that the \emph{Frobenius element} $\text{Frob}_{p}$ associated to $p$ is the conjugacy class in the Galois group $G := \text{Gal}(L/\mb{Q})$ of the Frobenius automorphisms of $\mf{p}$ in $\text{Gal}(F_{\mf{p}}/F_p)$, viewed as an element of $D_{\mf{p}}$ (this class being independent of the choice of $\mf{p}$ above $p$). That is,
$$
\text{Frob}_p(\beta) \equiv \beta^p \pmod{\mf{p}} \text{ for all } \beta\in\mc{O}_L.
$$
Given a union of conjugacy classes $\mc{C}$ of elements of $G$, we let
$$
\pi_{\mc{C}}(x, L/\mb{Q}) := |\{p \leq x : \, p \text{ unramified in } L, \, \text{Frob}_p \in \mc{C}\}|.
$$
\begin{thm}[Lagarias-Odlyzko \cite{LO}; Lagarias-Montgomery-Odlyzko \cite{LMO}] \label{thm:LMO}
Let $L/\mb{Q}$ be a Galois number field with Galois group $G$ and absolute discriminant $D_L$.
Let $\mc{C} \subseteq G$ be a union of conjugacy classes of $G$. Then there are absolute constants $C_1,C_2 > 0$ such that if
\begin{equation} \label{eq:LMOrange}
\log x \geq C_1 (\log D_L)(\log\log D_L)(\log\log \log (C_2 D_L))
\end{equation}
then we have
$$
\pi_{\mc{C}}(x,L/\mb{Q}) \ll \frac{|\mc{C}|}{|G|} \pi(x),
$$
the implicit constant being absolute. \\
Under the assumption of ERH, we may replace \eqref{eq:LMOrange} by the range
$$
x \gg (\log D_L)^2 (\log\log D_L)^4.
$$
\end{thm}
\noindent In the sequel, by a \emph{Kummer extension} of $\mb{Q}$ we mean a number field $L/\mb{Q}$ of the form
$$
L = \mb{Q}(\zeta_m,\alpha_1^{1/m},\ldots,\alpha_d^{1/m}),
$$
where $m > 1$ is an integer, $\zeta_m$ is a primitive $m$th root of unity, and $\alpha_1,\ldots,\alpha_d \in \mb{Q}^+$ with $d \geq 1$. Note that $L$ is the splitting field of the polynomial $\prod_{1 \leq i \leq d} (X^m-\alpha_i)$, thus necessarily Galois. 
\begin{proof}[Proof of Proposition \ref{prop:Assumlfl}]
We will give full details of the proof of the unconditional part of Proposition \ref{prop:Assumlfl}. As the details are largely the same we simply give a sketch of the necessary modifications required for the ERH-conditional part in Remark \ref{rem:GRHImprovement}, below. \\
We follow the strategy of the proof of \cite[Thm. 1.7]{Jones} with $K = \mb{Q}$ and $A = \mb{N}$ in the notation there, but must take care to deal with those primes $p$ for which the local power condition $f(p)^D \equiv p^{g_{\ell}}\pmod{\ell}$ fails.  \\
Let $\mc{L}_f$ be the set of primes in \eqref{eq:nonPrimPow}. We assume for the sake of contradiction that $\mc{L}_f$ is not sparse, i.e., that
$$
L_f(y) := \sum_{\ss{\ell \leq y \\ \ell \in \mc{L}_f}} \frac{1}{\ell} \ra \infty, \quad y \ra \infty.
$$
Let $X \geq X_0(\delta)$ and suppose $k \geq k_0(\delta)$ is chosen large enough so that $y_k = x_k^{1/A} \geq X$, where $A$ is as given in Proposition \ref{prop:almostHom}. By assumption, there is $s_f > 0$ such that
$$
\sum_{\ss{\ell \leq X \\ \ell \in S_f}} \frac{1}{\ell} \sim s_f \log\log X, \quad X \ra \infty.
$$ 
To obtain a contradiction we will show on the other hand that as $X \ra \infty$,
$$
\sum_{\ss{\ell \leq X \\ \ell \in S_f}} \frac{1}{\ell} = o(\log\log X). 
$$ 
Applying Lemma \ref{lem:Gdeltakm}, note that
$$
\sum_{\ss{\ell \leq X \\ \ell \in S_f}} \frac{1}{\ell} \leq \sum_{\ss{\ell \leq X \\ \ell \in S_f \cap \mc{G}_{\delta,k; 2b}}} \frac{1}{\ell} +  \sum_{\ss{\ell \leq x_k \\ \ell \notin \mc{G}_{\delta,k;2b}}} \frac{1}{\ell} = \sum_{\ss{\ell \leq X \\ \ell \in S_f \cap \mc{G}_{\delta,k; 2b}}} \frac{1}{\ell} + O_{\delta}(1),
$$
so we focus on bounding the contribution from $S_f \cap \mc{G}_{\delta,k;2b}$. \\
Let $2 \leq Y \leq Z \leq (\log X)^{1/6}$, with $Y,Z$ being chosen later as slowly-growing functions of $X$ that satisfy 
$$
\frac{\log Z}{\log Y}, \, \frac{\log\log X}{\log Z} \ra \infty, \text{ as } X \ra \infty.
$$
We also set $P := \tfrac{1}{10}Y^{1/(r+1)}$, with $r$ given by Proposition \ref{prop:ArchInhom}.  In the sequel, for $m \geq 1$ we define the sets
\begin{equation*} 
\mf{P}_{m} := \{\ell \in \mc{S}_f \cap \mc{G}_{\delta,k;2b} : \ell \equiv 1 \pmod{m}\}, \quad \mf{P}_{m}(y) := \mf{P}_{m} \cap [2,y].
\end{equation*}
We have the trivial upper bound
\begin{align}\label{eq:initDecomp}
\sum_{\ss{\ell \leq X \\ \ell \in S_f \cap \mc{G}_{\delta,k;2b}}} \frac{1}{\ell} &\leq \sum_{\ss{\ell \leq X \\ \ell \not \equiv 1 \pmod{m} \\ \forall \ m \in [Y,Z]}} \frac{1}{\ell} + \sum_{\ss{Y \leq m \leq Z \\ m \text{ prime}}} \sum_{\ss{\ell \in \mf{P}_m(X)}} \frac{1}{\ell} 
\nonumber\\
&=: \mc{E}(X) + \sum_{\ss{Y \leq m \leq Z \\ m \text{ prime}}} \mc{H}_m(X).
\end{align}
\underline{Step 1: Estimation of $\mc{E}(X)$} \\
Observe that if $\ell \leq Y$ the condition $m \nmid (\ell-1)$ is trivial for all $Y \leq m \leq Z$. On the other hand, if $v > Y$ then
by Lemma \ref{lem:Pol},
\begin{align*}
\sum_{\ss{\ell \leq v \\ \ell \not\equiv 1 \pmod{m} \\ \forall \, m \in [Y,Z]}} 1 &\leq \sum_{\ss{n \leq v \\ n \not \equiv 1 \pmod{m} \\
\forall \, m \notin [Y,\min\{v,Z\}]}} 1_{p|n \Rightarrow p > v^{1/10}} \\
&\ll \frac{v}{\log v} \exp\left(-\sum_{Y < p \leq \min\{v,Z\}} \frac{1}{p} + \sum_{v^{1/10} < p \leq v} \frac{1}{p}\right) \\
&\ll \frac{\log Y}{\log \min\{v,Z\}} \frac{v}{\log v}.
\end{align*}
Thus, by partial summation we obtain
\begin{align}
\mc{E}(X) &\leq \sum_{\ell \leq Y} \frac{1}{\ell} + \int_Y^{X} \left(\sum_{\ss{\ell \leq v \\ \ell \not \equiv 1 \pmod{m} \\ \forall \, m \in [Y,Z]}} 1\right) \frac{dv}{v^2} + O\left(\frac{1}{\log Y}\right) \nonumber \\
&\ll \log\log Y +  (\log Y)\int_Y^Z \frac{dv}{v(\log v)^2} + \frac{\log Y}{\log Z} \int_Z^{X} \frac{dv}{v\log v} \nonumber\\
&\ll \log\log Y + \frac{\log Y}{\log Z} \log\log X. \label{eq:EBound}
\end{align}
In the remainder of the argument we will show that for each prime $m \in [Y,Z]$,
$$
\mc{H}_m(X) \ll \frac{1}{m^2} \log\log X +\frac{1}{m}\log Z.
$$
\underline{Step 2: Constructing a suitable Kummer extension} \\
In \cite[Sec. 5]{Jones}, in order to bound quantities analogous to $\mc{H}_m(X)$ Jones makes heavy use of Kummer theory, which is well-exposed in his article. We will merely quote as a black box the relevant facts about Kummer extensions, as needed.\\
For the time being, fix a prime $m \in [Y,Z]$. We claim that we can find primes $p_1,p_2 \leq P$ with the properties
\begin{enumerate}[(i)]
\item $p_1,p_2 \notin \mc{S}_{\delta}$,
\item $p_1 \nmid f(p_2)$ and $p_2 \nmid f(p_1)$, and
\item for both $p = p_1$ and $p = p_2$ we have
\begin{equation}\label{eq:iiiCond}
\sum_{\ell \in \mc{T}(p) \cap \mf{P}_m(X)} \frac{1}{\ell} \leq \frac{1}{100} \sum_{\ell \in \mf{P}_m(X)} \frac{1}{\ell}.
\end{equation}
\end{enumerate}
Recall that by assumption, $L_f(P) \ra \infty$ with $P$, so we assume that $P$ (and thus $X$) are as large as desired in an absolute sense. Applying Lemma \ref{lem:controlExcep} with $\mc{U} = \mf{P}_m \cap [2,X]$, we see that the set of primes $p \leq P$ with 
\begin{equation}\label{eq:largeTp}
\sum_{\ell \in \mc{T}(p) \cap \mf{P}_m(X)} \frac{1}{\ell} > \frac{1}{100} \sum_{\ell \in \mf{P}_m(X)} \frac{1}{\ell}
\end{equation} 
has a sum of reciprocals satisfying
\begin{equation}\label{eq:sparseLTp}
\sum_{\ss{p \leq P \\ \eqref{eq:largeTp} \text{ holds}}} \frac{1}{p} \ll_{\delta} 1.
\end{equation}
Let $P_1 \in [2,P]$ be chosen minimally so that $L_f(P_1) \geq L_f(P)^{1/2}$. By \eqref{eq:sparseLTp} and Proposition \ref{prop:ArchInhom},
$$
\sum_{\ss{p \leq P_1 \\ p \in \mc{L}_f \bk S_{\delta} \\ \eqref{eq:largeTp} \text{ fails}}} \frac{1}{p} \geq L_f(P_1) - \sum_{\ss{p \leq x_k \\ p \in S_{\delta}}} \frac{1}{p} + O_{\delta}\left(1\right) \geq \frac{1}{2}L_f(P)^{1/2}. 
$$
Thus, we may select $p_1 \in (\mc{L}_f \bk S_{\delta}) \cap [2,P_1]$ such that \eqref{eq:iiiCond} holds for $p = p_1$.
Furthermore, taking $r' := 10(r+1)$ and combining \eqref{eq:Bblarge} and \eqref{eq:sparseLTp}, we have
\begin{align*}
\sum_{\ss{P_1^{r'} \leq p \leq P \\ p \in \mc{L}_f \bk S_{\delta}
 \\ \eqref{eq:largeTp} \text{ fails} \\ p_1 \nmid f(p)}} \frac{1}{p} &\geq \sum_{\ss{P_1 \leq p \leq P \\ p \in \mc{L}_f}} \frac{1}{p} - \sum_{\ss{p \leq P \\ p_1 | f(p)}} \frac{1}{p} - \sum_{\ss{P_1 \leq p \leq P_1^{r'}}} \frac{1}{p} - \sum_{\ss{p \leq x_k \\ p \in S_{\delta}}} \frac{1}{p} - \sum_{\ss{p \leq P \\ \eqref{eq:largeTp} \text{ holds}}} \frac{1}{p}  \\ 
 &= L_f(P) - L_f(P_1) - B_{p_1}(P) - O_{\delta}\left(1\right) \geq L_f(P) - 2 L_f(P)^{1/2} \\
 &\geq \frac{1}{2}L_f(P),
 \end{align*}
 for $P$ sufficiently large. Thus, we may select $p_2 \in (\mc{L}_f \bk \mc{S}_{\delta}) \cap [P_1^{r'},P]$ with $p_1 \nmid f(p_2)$ and such that \eqref{eq:largeTp} fails and thus \eqref{eq:iiiCond} holds for $p= p_2$.
 Moreover, since $p_1 \notin \mc{S}_{\delta}$ we have $|f(p_1)| \leq 2p_1^r < 2 P_1^{r'/2} \leq p_2$, so $p_2 \nmid f(p_1)$. The primes $p_1,p_2$ thus satisfy the required conditions (i), (ii) and (iii) above. \\
Now, let $\zeta_m$ be a primitive root of unity of order $m$, let $F(n) := f(n)^D$ and denote by $L_m$ the Kummer extension
$$
L_m := \mb{Q}(\zeta_m, p_1^{1/m}, p_2^{1/m}, F(p_1)^{1/m}, F(p_2)^{1/m}).
$$ 
Let also $p_3 \notin \{p_1,p_2\}$ be a prime such that $p_3|f(p_1)$. Following the proof of Case 1 of Lemma 5.8 of \cite{Jones} (which is the only case of interest to us there given our assumption on $\mc{L}_f$), we find upon setting $c_f := \nu_{p_3}(F(p_1))$ that
$$
[L_m:\mb{Q}(\zeta_m)] \geq m^3 \text{ whenever $m > c_f$}.
$$
The condition $m > c_f$ is automatically satisfied when $m \geq Y$ and $X$ is sufficiently large, since by virtue of $p_1 \leq P \leq Y^{1/(r+1)}$ we find
$$
c_f \leq 2\log |f(p_1)| < (r+1) \log P \leq \log Y.
$$
Since also $[\mb{Q}(\zeta_m):\mb{Q}] = m-1$ for prime $m$,
 the tower property for degrees of number fields yields
\begin{equation} \label{eq:degLm}
[L_m:\mb{Q}] \geq m^3(m-1) \text{ for all } Y \leq m \leq Z.
\end{equation}
For later reference, we note that the set of $p$ that ramify in $L_m$ are precisely those dividing the absolute discriminant $D_{L_m}$, and by \cite[Lem. 5.2]{Jones},
$$
D_{L_m} \text{ divides } (p_1p_2f(p_1p_2))^{[L_m:\mb{Q}]} m^{5[L_m:\mb{Q}]}.
$$
In particular, $\ell \equiv 1 \pmod{m}$ is unramified whenever $\ell \nmid p_1p_2f(p_1p_2)$.
Furthermore, as $P^{r+1} < Y \leq m$, and since $|f(p_i)| \leq P^{r+1}$ for $i = 1,2$, the trivial bound $[L_m:\mb{Q}] \leq m^4(m-1)$ leads to
\begin{equation}\label{eq:keytoCheb}
\log D_{L_m} \leq [L_m:\mb{Q}] \log(P^{2(r+2)} m^5) \ll m^{5} \log m,
\end{equation}
a bound that we will invoke shortly. \\
\underline{Step 3: Estimation of $\mc{H}_m(X)$ using Theorem \ref{thm:LMO}}  \\
Having defined $L_m$, we now split
\begin{align*}
\mc{H}_m(X) = \sum_{\ell \in \mf{P}_m(X)} \frac{1}{\ell} \leq \sum_{\ss{\ell \in \mf{P}_m(X) \\ \ell \nmid p_1p_2f(p_1p_2) \\ f(p_i)^D \equiv p_i^{g_\ell} \pmod{\ell} \\ i = 1,2}} \frac{1}{\ell} + \sum_{\ss{\ell | p_1p_2f(p_1p_2) \\ \ell \equiv 1 \pmod{m}}} \frac{1}{\ell} + \sum_{\ss{\ell \in \mf{P}_m(X) \\ p_1 \text{ or } p_2 \in \mc{T}_{\ell}}} \frac{1}{\ell}.
\end{align*}
Since $p_1,p_2 \leq P$ and $p_1,p_2 \notin \mc{S}_{\delta}$, we have 
$$
\max\{p_1,p_2,|f(p_1)|,|f(p_2)|\} \leq P^{r+1} < Y \leq m.
$$
Thus, $\ell \nmid p_1p_2f(p_1p_2)$ for any prime $\ell \equiv 1 \pmod{m}$, as each of these must obviously exceed $m$. In particular,
$$
\sum_{\ss{\ell| p_1p_2f(p_1p_2) \\ \ell \equiv 1 \pmod{m}}} \frac{1}{\ell} = 0.
$$
Furthermore, by the union bound, the definition of $\mc{T}(p_i)$ and property (iii) above, we have
$$
\sum_{\ss{\ell \in \mf{P}_m(X) \\ p_1 \text{ or } p_2 \in \mc{T}_{\ell}}} \frac{1}{\ell} \leq \sum_{\ell \in \mc{T}(p_1) \cap \mf{P}_m(X)} \frac{1}{\ell} + \sum_{\ell \in \mc{T}(p_2) \cap \mf{P}_m(X)} \frac{1}{\ell} \leq \frac{1}{50} \sum_{\ell \in \mf{P}_m(X)} \frac{1}{\ell} = \frac{1}{50} \mc{H}_m(X).
$$
Thus, upon rearranging we deduce that
\begin{equation}\label{eq:HmUppBd}
\mc{H}_m(X) \ll \sum_{\ss{\ell \in \mf{P}_m(X) \\ \ell \nmid p_1p_2f(p_1p_2) \\ f(p_i)^D \equiv p_i^{d_\ell} \pmod{\ell} \\ i = 1,2}} \frac{1}{\ell}.
\end{equation}
Note that the primes in the sum on the right-hand side of \eqref{eq:HmUppBd} are all unramified in $L_m$, and by \cite[Cor. 5.11]{Jones} there is a conjugacy class $\mc{C}_4$ of prime ideals of $L_m$ such that
$$
\text{if } \ell \in \mf{P}_m(X) \text{ and } \ell \nmid p_1p_2f(p_1p_2) \text{ then } \text{Frob}_\ell \in \mc{C}_4.
$$
Let now $2 \leq \tilde{X} \leq X$ be defined implicitly by 
$$
\log \tilde{X} = C Z^5 (\log Z)^3,
$$ 
where $C > 0$ is a large absolute constant to be specified shortly. Splitting the sum in \eqref{eq:HmUppBd} at $X$ and applying partial summation, we deduce that
\begin{equation}\label{eq:HmBdPS}
\mc{H}_m(X) \ll \sum_{\ss{\ell \leq X \\ \ell \equiv 1 \pmod{m} \\ \text{Frob}_{\ell} \in \mc{C}_4}} \frac{1}{\ell} \ll \sum_{\ss{m < \ell \leq \tilde{X} \\ \ell \equiv 1 \pmod{m}}} \frac{1}{\ell} + \int_{\tilde{X}}^{X} \frac{\pi_{\mc{C}_4}(v,L_m/\mb{Q})}{v^2} dv + \frac{1}{m\log \tilde{X}}.
\end{equation}
By Lemma \ref{lem:BT} and partial summation, 
\begin{align*}
\sum_{\ss{m < \ell \leq \tilde{X} \\ \ell \equiv 1 \pmod{m}}} \frac{1}{\ell} &\ll \sum_{\ss{m < n \leq m^{10} \\ n \equiv 1 \pmod{m}}} \frac{1}{n} + \int_{m^{10}}^{\tilde{X}} \frac{\pi(v;m,1)}{v^2} dv + \frac{1}{m \log m} \\
&\ll \frac{\log m}{m} + \frac{1}{m}\log\log \tilde{X} \ll \frac{1}{m} \log Z.
\end{align*}
Next, as $p_1 \neq p_2$ we have $[\mb{Q}(\zeta_m,p_1^{1/m},p_2^{1/m}):\mb{Q}] = m^2(m-1)$, so we may combine \eqref{eq:degLm} with \cite[Lem. 5.12]{Jones} to obtain
$$
\frac{|\mc{C}_4|}{\text{Gal}(L_m/\mb{Q})} \leq \frac{2}{m(m-1)}.
$$
Moreover, for each $Y \leq m \leq Z$ and each $\tilde{X} \leq v \leq X$ (with $C$ chosen sufficiently large in the definition of $\tilde{X}$) Theorem \ref{thm:LMO} gives
$$
\pi_{\mc{C}_4}(v,L_m/\mb{Q}) \ll \frac{|\mc{C}_4|}{|\text{Gal}(L_m/\mb{Q})|} \pi(v) \ll \frac{v}{m^2 \log v}.
$$
Inserting this bound into \eqref{eq:HmBdPS} for $v \geq \tilde{X}$, and bounding trivially elsewhere, we obtain
$$
\mc{H}_m(X) \ll \frac{1}{m^2} \int_{\tilde{X}}^X \frac{dv}{v\log v} + \frac{1}{m} (\log Z + 1) \ll \frac{1}{m^2} \log\log X + \frac{1}{m}\log Z,
$$
uniformly over primes $m \in [Y,Z]$. \\
\underline{Concluding the proof:} \\
Summing over the primes $Y \leq m \leq Z$ and inserting the resulting bound together with \eqref{eq:EBound} into \eqref{eq:initDecomp}, we obtain 
\begin{equation}\label{eq:finalBdforCeb}
\sum_{\ss{\ell \leq X \\ \ell \in S_f}} \frac{1}{\ell} \ll \left(\frac{\log Y}{\log Z} + \frac{1}{Y \log Y}\right)\log\log X + (\log Z)\log\left(\frac{\log Z}{\log Y}\right) + \log\log Y. 
\end{equation}
We now select
$$
Y = (\log\log X)^{1/2}, \quad Z = \exp((\log\log X)^{1/2}),
$$
so that $Y \leq Z \leq (\log X)^{1/6}$, as needed. Then
$$
\sum_{\ss{\ell \leq X \\ \ell \in S_f}} \frac{1}{\ell} \ll (\log\log X)^{1/2} \log\log \log X,
$$
which is $o(\log\log X)$ as $X \ra \infty$, as claimed.
\end{proof}
\begin{rem} \label{rem:GRHImprovement}
The constraint $\log y \gg m^5(\log m)^3$ required to apply Theorem \ref{thm:LMO} limited the choices of $Y$ and $Z$ in the above proof, and we may do much better by assuming the Extended Riemann Hypothesis.
As stated in Theorem \ref{thm:LMO}, for any conjugacy class $\mc{C} \subseteq \text{Gal}(L/K)$ one has under ERH that
$$
\pi_{\mc{C}}(y,L/\mb{Q}) \ll \frac{|\mc{C}|}{|G|} \frac{y}{\log y} \text{ for } y \gg (\log D_L)^2(\log\log D_L)^4.
$$
In particular, when $L = L_m$ with $m \in [Y,Z]$ as in the proof above, we may instead take $\tilde{X} = Z^{10}(\log Z)^6$. \\
Using $\tilde{X} > Y$ and keeping track only of the primes $\tilde{X} < \ell \leq X$ (which has the effect that the first term in each of the bounds \eqref{eq:EBound} and \eqref{eq:HmBdPS} can be removed), the above proof gives
\begin{equation} \label{eq:relBoundProp63}
\sum_{\ss{\tilde{X} < \ell \leq X \\ \ell \in \mc{S}_f \cap \mc{G}_{\delta,k;2b}}} \frac{1}{\ell} \ll \left(\frac{\log Y}{\log Z} + \frac{1}{Y\log Y}\right) \log\log X + \frac{\log\log Z}{\log \tilde{X}},
\end{equation}
which is valid as long as $X$ is sufficiently large subject to the upper bound $X \leq x_k$. \\
Set now $J := \lfloor \log\log X\rfloor$. For each $j_0 \leq j \leq J$, where $j_0 \gg_{\delta} 1$, take 
$$
Y_j = j^{12}, \, Z_j = \exp(j^{20}), \, \tilde{X}_j = Z_j^{10}(\log Z_j)^6 \text{ and } X_j = e^{e^j}.
$$ 
Applying \eqref{eq:relBoundProp63} with $X,\tilde{X},Y$ and $Z$ replaced by $X_j, \tilde{X}_j, Y_j$ and $Z_j$, respectively, we find
$$
\sum_{\ss{X_{j-1} < \ell \leq X_j \\ \ell \in \mc{S}_f \cap \mc{G}_{\delta,k;2b}}} \frac{1}{\ell} \leq \sum_{\ss{\exp(20 j^{20}) < \ell \leq X_j \\ \ell \in \mc{S}_f \cap \mc{G}_{\delta,k;2b}}} \frac{1}{\ell} \leq \sum_{\ss{\tilde{X}_j < \ell \leq X_j \\ \ell \in \mc{S}_f \cap \mc{G}_{\delta,k;2b}}} \frac{1}{\ell}  \ll \frac{1}{j^{10}}
$$
(the constructions of the extensions $L_m$ may be relativised to $m \in [Y_j,Z_j]$, as long as $Y_j,Z_j$ are sufficiently large with respect to $\delta$). Since $X_J \leq X < X_{J+1}$, summing over $j_0 \leq j \leq J$ yields
$$
\sum_{\ss{\ell \leq X \\ \ell \in \mc{S}_f \cap \mc{G}_{\delta,k;2b}}} \frac{1}{\ell} \leq \sum_{j_0+1 \leq j \leq J} \sum_{\ss{X_{j-1} < \ell \leq X_j \\ \ell \in S_f \cap \mc{G}_{\delta,k;2b}}} \frac{1}{\ell} + O_{\delta}(1) \ll \sum_{j > j_0} \frac{1}{j^{10}} + O_{\delta}(1) \ll_{\delta} 1.
$$
Combined with Lemma \ref{lem:Gdeltakm}, we thus obtain
$$
\sum_{\ss{\ell \leq X \\ \ell \in S_f}} \frac{1}{\ell} \leq \sum_{\ss{\ell \leq X \\ \ell \in S_f \cap \mc{G}_{\delta,k;2b}}} \frac{1}{\ell} + \sum_{\ss{\ell \leq x_k \\ \ell \notin \mc{G}_{\delta,k;2b}}} \frac{1}{\ell} \ll_{\delta} 1,
$$
uniformly over $k$, and thus over $X$. We hence deduce  assuming ERH that $S_f$ must be $O_{\delta}(1)$-sparse, which contradicts the hypotheses of Proposition \ref{prop:Assumlfl}. 
\end{rem}
\begin{rem}
One may ask whether the full strength of the ERH is actually needed in the conditional proof. In fact, the success of our method relies only on the convergence as $J \ra \infty$ of the series
$$
\sum_{j_0 < j \leq J} \sum_{\ss{X_{j-1} < \ell \leq X_j \\ \ell \in S_f \cap \mc{G}_{\delta,k;2b}}} \frac{1}{\ell}.
$$
For this to work, it is enough to be able to take $Z_j$ so that $\log Z_j \geq (\log\log X_j)^{1+c}$ for some $c > 0$. While this range of parameters is not accessible using unconditional estimates in Theorem \ref{thm:LMO} (the range available is rather of the form $\log Z_j \geq (\log\log X_j)^c$, with $c \in (0,1/5)$), it does hold if one assumes the existence of fairly narrow zero-free regions for all of the Hecke $L$-functions of suitable cyclic subfields of each of the extensions $L_m$, $m \in [Y,Z]$. For instance, applying \cite[Thm. 7.1]{LO}, we see that if $\mb{Q} \subseteq E_m \subseteq L_m$ has (cyclic) Galois subgroup $H_m \leq \text{Gal}(L_m/\mb{Q})$ then we have
$$
\psi_{\mc{C}}(x,L_m/\mb{Q}) := \sum_{\ss{p^k \leq x \\ \text{Frob}_p \in \mc{C}}} \log p  \ll \frac{|\mc{C}|}{|\text{Gal}(L_m/\mb{Q})|} x
$$
as long as $m \leq Z \leq x^{1/6}$ (so that $n_{L_m} \leq \log D_{L_m} \ll m^5 \log m = o(x)$) and each of the Hecke $L$-functions $L(s,\chi, L_m/E_m)$, $\chi \in H_m$, has no zeroes in a region
$$
\text{Re}(s) > 1-\delta_m, \quad |\text{Im}(s)| \leq T_m, \quad \delta_m := \frac{C(\log m + \log\log x)}{\log x}, \, T_m := m^5 (\log x)^2.
$$
To access the range $\log Z_j \geq (\log\log X_j)^{1+c}$ it would suffice therefore to have access to a zero-free region in which $\text{Re}(s) > 1-\frac{(\log\log X_j)^2}{\log X_j}$ and $|\text{Im}(s)| \ll \exp((\log\log X_j)^{2+\e})$, for example. Any such zero-free region would be implied by one for the corresponding Dedekind zeta function $\zeta_{L_m}$, and thus something far weaker than ERH would suffice. \\
The potential for such improvements under the assumption of weak zero-free regions suggests a possible unconditional improvement to Proposition \ref{prop:Assumlfl} by invoking zero-density estimates. This seems reasonable in light of the fact that in constructing the extensions $L_m$ we had a significant amount of flexibility in the choice of primes satisfying conditions (i)-(iii) of Step 2, and thus there are rather substantial collections of extensions that could be produced in this way. However, currently existing zero-density estimates for Dedekind zeta functions, such as those of \cite{PTBW} and \cite{TZ} are effective for collections of fields in which the Galois group is either of fixed or very slowly growing size (see for instance \cite[Thm. 1.1]{TZ}, where the quality of the bound depends exponentially on the size $|G|$ of the Galois group of the number fields in the collection). In our case, however, $\text{Gal}(L_m/\mb{Q})$ grows polynomially with $m \in [Y,Z]$, which unfortunately makes these latter results unsuitable. 
\end{rem}
\subsection{Proof of Proposition \ref{prop:veryRigid}}
The proof of Proposition \ref{prop:veryRigid} relies on a few elementary sieve lemmas.
\begin{lem} \label{lem:sieveSparse}
Let $C > 0$ and let $S$ be a set of prime powers that is $C$-sparse.
For each $n \in \mb{N}$ write 
$$
n_S := \prod_{\ss{p^\nu || n \\ p^{\nu} \in S}} p^{\nu}
$$ 
to denote the product of all prime factors of $n$ that belong to $S$. 
Then for any $\e > 0$ there is a $Z_0 = Z_0(\e,C)$ such that if $Z \geq Z_0$ then for any $X$ sufficiently large,
$$
N_S(X;Z) := |\{n \leq X : n_S > Z\} \ll \e X.
$$
\end{lem}
\begin{proof}
Since $n_S|n$ for each $n\in \mb{N}$ the claim is immediate if $Z > X$. We assume henceforth that $Z \leq X$. As in Section \ref{sec:arch}, define 
$$
\omega_S(n) := |\{p^k || n : p^k \in S\}|, \quad E_S(Y) := \sum_{\ss{p^k \leq Y \\ p^k \in S}} \frac{1}{p}.
$$ 
By assumption, $E_S(Y) \leq C$ for all  $Y\geq 2$. \\
Let $Y_1 = \sqrt{\log Z}$ and $Y_2 = e^{\tfrac{1}{2}\sqrt{\log Z}}$. For $Z_0$ large enough and $Z \geq Z_0$ we have $Y_1 \geq 2C \geq 2E_S(X)$, so that by Lemma \ref{lem:TK},
$$
|\{n \leq X : \omega_S(n) > Y_1\}| \leq |\{n \leq X : |\omega_S(n) - E_S(X)| > Y_1/2\}| \ll \frac{1}{Y_1^{2}} \sum_{n \leq X} |\omega_S(n)-E_S(X)|^2 \ll \frac{X}{Y_1}.
$$
Furthermore, we also note that
$$
|\{n \leq X : \exists \ p^\nu||n, \ p^\nu \in S \text{ and } p^\nu > Y_2\}| \leq \sum_{\ss{Y_2 < p^\nu \leq X \\ p^\nu \in S}} \sum_{\ss{n \leq X \\ p^\nu ||n}} 1 \ll X\sum_{\ss{p^\nu > Y_2 \\ p^\nu \in S}} \frac{1}{p^\nu}. 
$$
Now for any $n \leq X$ satisfying $\omega_S(n) \leq Y_1$ and $\max\{p^\nu \in S : p^\nu || n\}\leq Y_2$, 
$$
n_S = \prod_{\ss{p^\nu || n \\ p^\nu \in S}} p^\nu \leq Y_2^{\omega_S(n)} \leq Y_2^{Y_1} < Z.
$$
Consequently, by the union bound
\begin{align*}
N_S(X;Z) &\leq |\{n \leq X : \omega_S(n) > Y_1\}| + |\{n \leq X : \exists \, p^\nu || n, \, p^\nu \in S \text{ and } p^\nu > Y_2\}| \\
&\ll X\left(Y_1^{-1} + \sum_{\ss{p^\nu > Y_2 \\ p^\nu \in S}} \frac{1}{p^\nu}\right) \ll \e X,
\end{align*}
as long as $Z \geq Z_0$ and $Z_0$ is chosen large enough, as claimed.
\end{proof}
\begin{lem}\label{lem:PollApp}
Given $\eta \in (0,1)$, define $\mc{R}_{\eta}$ to be the set of primes $\ell$ such that
\begin{equation*}
|\{a \in (\mb{Z}/\ell \mb{Z})^\times : \, a \text{ is a primitive root}\}| < \eta (\ell-1).
\end{equation*}
Then as $X \ra \infty$ we have
$$
\sum_{\ss{\ell \leq X \\ \ell \in \mc{R}_{\eta}}} \frac{1}{\ell} \ll \eta^{100} \log\log X.
$$
\end{lem}
\begin{proof}
By partial summation, it suffices to show that
$$
|\mc{R}_{\eta} \cap [2,y]| \ll \eta^{100}\frac{y}{\log y}, \quad y \geq y_0.
$$
It is a classical fact that there are $\phi(\ell-1)$ distinct primitive roots modulo a prime $\ell$. Thus, defining the multiplicative function $g(n) := \left(n/\phi(n)\right)^{100}$, we have
$$
\ell \in \mc{R}_{\eta} \text{ if and only if } g(\ell-1) > \eta^{-100}.
$$ 
Note that $g(p^k) = 1 + O(1/p)$ for all prime powers $p^k$, so that $g(n) \leq d(n)$ for all sufficiently large $n \geq 1$. Markov's inequality and Lemma \ref{lem:Pol}
thus combine to yield
\begin{align*}
|\mc{R}_{\eta} \cap [2,y]| \leq \eta^{100} \sum_{\ell \leq y} g(\ell-1) \leq \eta^{100}\sum_{\ss{n \leq y-1 \\ p|(n+1) \Rightarrow p > y^{1/10}}} g(n) &\ll \eta^{100} \frac{y}{\log y} \exp\left(\sum_{p\leq y} \frac{g(p)-1_{p \leq y^{1/10}}}{p}\right) \\
&\ll \eta^{100}\frac{y}{\log y},
\end{align*}
as claimed.
\end{proof}
\begin{lem} \label{lem:LangZacApp}
Let $x > y \geq 10$ and let $1 \leq L \leq \exp(\sqrt{\log y})$. Then, outside of a single possible exception, for any prime $\ell \leq L$ we have
$$
\sum_{\ss{y < p \leq x \\ \text{ord}_{\ell}(p) = \ell-1}} \frac{1}{p} = \frac{\phi(\ell-1)}{\ell-1} \log\left(\frac{\log x}{\log y}\right) + O(L^{-10}).
$$
\end{lem}
\begin{proof}
Let $a \in (\mb{Z}/\ell \mb{Z})^{\times}$, and write 
$$
Q(X) := \exp((\log X)^{3/5}(\log\log X)^{-1/5}), \quad X \geq 3.
$$
By \cite[Thm. 2]{LangZac}, outside of possibly one $\ell \leq Q(X)$ (in case Siegel zeroes exist) there is a positive real number $C(\ell,a)$ such that
$$
\sum_{\ss{p \leq X \\ p \equiv a \pmod{\ell}}} \frac{1}{p} = \frac{1}{\ell-1} \log\log X - \log C(\ell,a) + O\left(Q(X)^{-100}\right).
$$
Applying this with $X \in \{y,x\}$ and subtracting the two, we obtain
$$
\sum_{\ss{y < p \leq x \\ p \equiv a \pmod{\ell}}} \frac{1}{p} = \frac{1}{\ell-1} \log\left(\log x/\log y\right) + O\left(Q(y)^{-100}\right).
$$
Since $\ell \leq L \ll Q(y)$, upon summing over all $a \pmod{\ell}$ with $\text{ord}_\ell(a) = \ell-1$ we obtain
$$
\sum_{\ss{y< p \leq x \\ \text{ord}_{\ell}(p) = \ell-1}} \frac{1}{p} = \frac{\phi(\ell-1)}{\ell-1} \log(\log x/\log y) + O(LQ(y)^{-100}) = \frac{\phi(\ell-1)}{\ell-1} \log(\log x/\log y) + O(L^{-10}),
$$
as required.
\end{proof}
\begin{proof}[Proof of Proposition \ref{prop:veryRigid}]
Let $\mc{G} := \{\ell \in \mc{P} : \ell'|f(\ell) \Rightarrow \ell' = \ell\} = \mc{P} \bk \mc{L}_f$, in the notation of the proof of Proposition \ref{prop:Assumlfl}.
Thus for each $\ell \in \mc{G}$ there is $\alpha_{\ell} \geq 1$ such that $f(\ell) = \ell^{\alpha_{\ell}}$, and also
\begin{equation}\label{eq:mcGHyp}
\sum_{\ell \notin \mc{G}} \frac{1}{\ell} < \infty.
\end{equation}
Moreover, if $\ell \notin \mc{S}_{\delta}$ (defined in Proposition \ref{prop:ArchInhom}) and is large enough in terms of $\delta$ then $1 < |f(\ell)| < 2\ell^r$, so that 
$$
1 \leq \alpha_{\ell} \leq r + 1 \text{ for all } \ell \in \mc{G} \bk \mc{S}_{\delta}, \, \ell \gg_{\delta} 1.
$$ 
In particular, $\alpha_{\ell} = O_{\delta}(1)$ uniformly on $\mc{G} \bk \mc{S}_{\delta}$. \\
The remainder of the proof splits into the following steps:
\begin{enumerate}[(a)]
\item finding $\alpha_0 \in \mb{N}$ such that $\{\ell \in \mc{P}: \, |f(\ell)| \neq \ell^{\alpha_0}\}$ is $O_{\delta}(1)$-sparse,
\item showing that $\alpha_0 = 1$, and
\item in the case that $f$ is completely multiplicative, showing that there is a divisor $d_0|a$ such that $f(d_0)|b$, and that $a/d_0$ divides $b/f(d_0)$.
\end{enumerate}
\, \\
\underline{Step (a):} In order to show (a), we will prove that for each $P \geq 2$ there is an integer $1 \leq \alpha(P) \ll_{\delta} 1$ such that
\begin{equation}\label{eq:goalScale}
\sum_{\ss{p \leq P \\ |f(p)| \neq p^{\alpha(P)}}} \frac{1}{p} = O_{\delta}(1).
\end{equation}
Since $\alpha(P)$ is an integer of bounded size, we can choose $\alpha_0 \in \mb{N}$ by the pigeonhole principle such that $\alpha(P_j) = \alpha_0$ on an infinite sequence $(P_j)_j$. In this case, we have
$$
\sum_{\ss{p \leq P_j \\ |f(p) |\neq p^{\alpha_0}}} \frac{1}{p} = O_{\delta}(1) \text{ for all } j \geq 1,
$$
and since the latter sum is monotone in $j$, we deduce on taking $j \ra \infty$ that
$$
\sum_{p : |f(p)| \neq p^{\alpha_0}} \frac{1}{p} < \infty,
$$
the size of the series depending only on $\delta$. Our goal is therefore to prove \eqref{eq:goalScale}. \\
Let $P \geq 2$. Since the claim is trivial otherwise, we may assume that $P$ is larger than any fixed constant depending only on $\delta$. Let $L_1 \geq e^P$, choose $k$ large enough so that $x_k \geq e^{(\log L_1)^{5}}$ and set $L_2 := \exp((\log x_k)^{2/5})$. 
We begin by showing that there is a large set of primes $\ell \in [L_1,L_2]$ for which the exponents $g_{\ell}$ in \eqref{eq:ellDivfell} are all equal and of size $O_{\delta}(1)$. \\
Define $\mc{R}_{\delta}$ as in the statement of Lemma \ref{lem:PollApp}, and consider the set 
$$
\mc{G}' := (\mc{G} \cap \mc{G}_{\delta,k;2b}) \bk \mc{R}_{\delta} \subseteq \mc{G}\bk \mc{S}_{\delta}.
$$ 
Combining \eqref{eq:mcGHyp} with Lemmas \ref{lem:Gdeltakm} and \ref{lem:PollApp}, we have
\begin{align*}
\sum_{\ss{L_1 \leq \ell \leq L_2 \\ \ell \in \mc{G}'}} \frac{1}{\ell} &\geq \sum_{\ss{\ell \leq x_k \\ \ell \in \mc{G}_{\delta,k;2b}}} \frac{1}{\ell} - \sum_{\ss{\ell \leq x_k \\ \ell \notin \mc{G}}} \frac{1}{\ell} - \sum_{\ss{\ell \leq x_k \\ \ell \in \mc{R}_{\delta}}} \frac{1}{\ell} - \sum_{\ell \in [2,L_1] \cup [L_2,x_k]} \frac{1}{\ell} \\
&\geq (1/5-\delta^{50}) \log\log x_k.
\end{align*}
Now for each $\ell \in \mc{G}' \cap [L_1,L_2]$ we have 
$$
\log \ell \leq \log L_2 = (\log x_k)^{2/5} \text{ and } \phi(\ell-1)/(\ell-1) \geq \delta.
$$ 
Set $z_k := \exp((\log x_k)^{4/5})$, so that $L_2 = e^{\sqrt{\log z_k}}$. By Lemma \ref{lem:LangZacApp},
outside of possibly one prime $\ell_0$, for $\ell \leq L_2$ we have
$$
\sum_{\ss{p \leq x_k \\ \text{ord}_\ell(p) = \ell-1}} \frac{1}{p} \geq \sum_{\ss{z_k < p \leq x_k \\ \text{ord}_{\ell}(p) = \ell-1}} \frac{1}{p} = \frac{\phi(\ell-1)}{\ell-1} \log(\log x_k/\log z_k) + O(1) \geq \frac{\delta}{5} \log\log x_k + O(1), \quad \ell \in \mc{G}' \cap [L_1,L_2].
$$
Taking account of the primes not in $\mc{G}' \bk \mc{T}_{\ell}$ using Lemma \ref{lem:Gdeltakm} and Proposition \ref{prop:almostHom}, we find that
$$
\sum_{\ss{p \leq x_k \\ p \in \mc{G}' \bk \mc{T}_{\ell} \\ \text{ord}_\ell(p) = \ell-1}} \frac{1}{p} \geq \sum_{\ss{p \leq x_k \\ \text{ord}_\ell(p) = \ell-1}} \frac{1}{p} - \sum_{\ss{p \leq x_k \\ p \notin \mc{G}}} \frac{1}{p} - \sum_{\ss{p \leq x_k \\ p \in \mc{T}_\ell}} \frac{1}{p} 
\geq \frac{\delta}{5}\log\log x_k + O_{\delta}(1).
$$
In particular, for each of the $\ell \in \mc{G}' \cap [L_1,L_2]$, $\ell \neq \ell_0$, we may select a prime $p_{\ell} \notin \mc{T}_{\ell}$ that is a primitive root modulo $\ell$, such that $f(p_{\ell}) = p_{\ell}^{\alpha_{p_{\ell}}}$ for some $\alpha_{p_{\ell}} = O_{\delta}(1)$.
Since $(\alpha_{p_{\ell}})_{\ell \in \mc{G}' \cap [L_1,L_2]}$ belongs to a finite set of integers depending only on $\delta$, by the pigeonhole principle there is an $\alpha = \alpha(P)$ with $1 \leq \alpha \ll_{\delta}(1)$ (which might depend on $L_1,L_2$ and $x_k$, thus on $P$) such that
$$
\mc{U}_{\alpha} := \{\ell \in \mc{G}' \cap [L_1,L_2] : \, \alpha_{p_{\ell}} = \alpha\}
$$
satisfies
$$
\sum_{\ss{
\ell \in \mc{U}_{\alpha}}} \frac{1}{\ell} \gg_{\delta} \log\log x_k.
$$
Fix $\ell \in \mc{U}_{\alpha}$
for the moment. As $\ell \in \mc{G}$, $f(\ell) = \ell^{\alpha_\ell} = q_{\ell}$ in the notation of 
Proposition \ref{prop:almostHom}. That result implies that for prime $p \leq x_k$,
$$
\text{if } p \in [2,x_k] \bk \mc{T}_{\ell} \text{ then } f(p)^D \equiv 
p^{g_\ell} \pmod{\ell},
$$
for some $0 \leq g_{\ell} < \ell-1$. 
By choice, the prime $p_{\ell}$ satisfies
$$
p_{\ell}^{D\alpha} = f(p_{\ell})^D \equiv p_\ell^{g_{\ell}} \pmod{\ell}.
$$
Rearranging, we deduce using the fact that $p_\ell$ is a primitive root modulo $\ell$ that
$$
p_\ell^{g_{\ell}-D\alpha} \equiv 1 \pmod{\ell}, \text{ or equivalently, } (\ell-1)|(g_{\ell} - D\alpha).
$$
Since $1 \leq D\alpha = O_{\delta}(1)$ uniformly in $\ell \in \mc{U}_{\alpha}$, and $g_{\ell} < \ell-1$ by construction, we deduce that 
$$
g_\ell = D\alpha \text{ for all } \ell \in \mc{U}_{\alpha}.
$$
We are now in a position to prove \eqref{eq:goalScale}. 
Since $P \leq \log L_1$, 
applying Lemma \ref{lem:controlExcep} we get
\begin{equation}\label{eq:TpShare}
\sum_{\ell \in \mc{T}(p) \cap \mc{U}_\alpha} \frac{1}{\ell} < \frac{1}{2} \sum_{\ell \in \mc{U}_\alpha} \frac{1}{\ell}
\end{equation}
outside of a set of $p \leq P$ such that
$$
\sum_{\ss{p \leq P \\ \eqref{eq:TpShare} \text{ fails}}} \frac{1}{p} = O_{\delta}(1).
$$
Now suppose $p \leq P$ satisfies \eqref{eq:TpShare}, and $p \notin \mc{S}_{\delta}$. Then there is at least one prime $\ell \in \mc{U}_{\alpha}$
such that $p \notin \mc{T}_{\ell}$, and
$$
f(p)^D \equiv p^{D\alpha} \pmod{\ell}.
$$
If $P$ is large enough then $\ell \geq L_1 > P^{D(r+\alpha+1)}$, and thus 
$$
\ell > (2p^r)^D + p^{D\alpha} \geq |f(p)^D-p^{D\alpha}|.
$$ 
We therefore deduce that $f(p)^D = p^{D\alpha}$, and in particular $|f(p)| = p^{\alpha}$, whenever $p \leq P$ satisfies \eqref{eq:TpShare}. \\
It follows that if $p \leq P$ satisfies $|f(p)| \neq p^{\alpha}$ then either \eqref{eq:TpShare} fails or $p \in \mc{S}_{\delta}$, and so we deduce from \eqref{eq:TpShare} that
$$
\sum_{\ss{p \leq P \\ |f(p)| \neq p^{\alpha}}} \frac{1}{p} = O_{\delta}(1),
$$
where $\alpha = \alpha(P)$. This proves \eqref{eq:goalScale}, as required.
\\ \\
\underline{Step (b):} Next, we show that $\alpha_0 = 1$. Set 
$$
\tilde{\mc{G}} := \{p \in \mc{G} \bk \mc{S}_{\delta} : |f(p)| = p^{\alpha_0}\},
$$
the complement of which, according to the preceding arguments, is $O_{\delta}(1)$-sparse. For each $n \in \mb{N}$ we write
$$
A(n) := \prod_{\ss{p^k||n \\ p \notin \tilde{\mc{G}}}} p^k
\cdot \prod_{\ss{p^k || n \\ p \in \tilde{\mc{G}} \\ k \geq 2}} p^k, \quad s(n) := \prod_{\ss{p||n \\ p \in \tilde{\mc{G}}}} p.
$$
Define also the $\pm 1$-valued multiplicative function $h(n)$ on prime powers by
$$
h(p^k) = \begin{cases} f(p)/p^{\alpha_0} &\text{ if } p \in \tilde{\mc{G}}, \, k = 1 \\ 
1 &\text{ otherwise.}
\end{cases}
$$
Since every integer can be written as
$$
n = A(n)s(n)
$$
with $(A(n),s(n)) = 1$, we always have the identity 
$$
f(n) = f(A(n)) f(s(n)) = f(A(n)) h(n)s(n)^{\alpha_0} =  \frac{f(A(n))}{A(n)^{\alpha_0}} h(n) n^{\alpha_0} =: G(n) n^{\alpha_0}.
$$
Thus, for each $n \in \mc{N}_{f,a,b}$,
\begin{equation}\label{eq:withAn}
(n+a)^{\alpha_0} G(n+a) = n^{\alpha_0} G(n) + b.
\end{equation}
Now, applying Lemma \ref{lem:sieveSparse} with $\e = \delta^2$ and the $O_{\delta}(1)$-sparse set
$$
S = \{p^\nu : p \notin \tilde{\mc{G}}, \, \nu \geq 1\} \cup \{p^\nu : p \in \tilde{\mc{G}}, \, \nu \geq 2\},
$$
there is a choice $N_0 = O_{\delta}(1)$ such that the set
$$
\mc{N}_{f,a,b}' := \{n \in \mc{N}_{f,a,b} : \, A(n), A(n+a) \leq N_0\}
$$
satisfies $\overline{\delta}(\mc{N}_{f,a,b}') \geq \delta/2$, when $\delta$ is small enough. Since also $h(n),h(n+a) \in \{-1,+1\}$, by the pigeonhole principle, there are integers $1 \leq A,B \leq N_0$ and elements $\eta_1,\eta_2 \in \{-1,+1\}$ for which 
$$
\mc{N}_{f,a,b}'' := \{n \in \mc{N}_{f,a,b}' : \, A(n+a) = A, \, A(n) = B, \, (h(n),h(n+a)) = (\eta_1,\eta_2)\}
$$
has $\overline{\delta}(\mc{N}_{f,a,b}'') \gg_\delta 1$. In particular, for each $n \in \mc{N}_{f,a,b}''$,
$$
G(n+a) = \eta_2 f(A)/A^{\alpha_0}, \quad G(n) = \eta_1 f(B)/B^{\alpha_0}.
$$
Writing 
$$
\tilde{A} := \eta_2 f(A)B^{\alpha_0}, \quad \tilde{B} :=  \eta_1 f(B)A^{\alpha_0}, \quad \tilde{C} := b(AB)^{\alpha_0}
$$
and rearranging \eqref{eq:withAn}, we obtain
$$
\tilde{A}(n+a)^{\alpha_0} = \tilde{B} n^{\alpha_0} + \tilde{C} \text{ for all } n \in \mc{N}_{f,a,b}''.
$$ 
Since $\overline{\delta}(\mc{N}_{f,a,b}'') > 0$, dividing both sides by $n^{\alpha_0}$ and taking $n \ra \infty$ along an infinite increasing subsequence of $\mc{N}_{f,a,b}''$ yields $\tilde{A} = \tilde{B}$, from which we obtain
$$
\tilde{A}((n+a)^{\alpha_0} - n^{\alpha_0}) = \tilde{C}.
$$
Note that $\tilde{C} \neq 0$ by definition, so $\tilde{A} \neq 0$ as well. Since $a \geq 1$ we have $(n+a)^{\alpha_0} - n^{\alpha_0} \geq \alpha_0 a n^{\alpha_0-1}$. As this holds for infinitely many $n$ we conclude that $\alpha_0 = 1$, as required. \\ \\
\underline{Step (c):} 
When $f$ is completely multiplicative we can make the following reduction, with the effect that every element of $\mc{N}_{f,a,b}$ may be assumed to be coprime to $a$. \\
For each $d|a$ let us define
$$
\mc{N}_{f,a,b}(d) := \{n \in \mc{N}_{f,a,b} : (n,a) = d\} = \{n \in \mb{N} : f(n+a) = f(n) + b, \, (n,a) = d\}.
$$
We thus obtain the partition
$$
\mc{N}_{f,a,b} = \bigcup_{d|a} \mc{N}_{f,a,b}(d),
$$
and by the pigeonhole principle, there is a divisor $d_0|a$ for which $\bar{\delta}(\mc{N}_{f,a,b}(d_0)) > \delta/d(a)$, where $d(m)$ denotes the divisor function. \\
Now, for each $n \in \mc{N}_{f,a,b}(d_0)$, writing $n = md_0$ and $a' = a/d_0$ we get by complete multiplicativity that
$$
f(d_0)f(m + a') = f(n+a) = f(md_0) + b = f(d_0)f(m)+b.
$$
We thus deduce that $f(d_0)|b$. We now set $b' := b/f(d_0)$, $\eta := \delta/(d_0d(a)) > 0$ and 
$$
\tilde{\mc{N}}_{f,a',b'} := \{m \in \mb{N} : md_0 \in \mc{N}_{f,a,b}(d_0)\} = \mc{N}_{f,a',b'}(1).
$$ 
Then $\bar{\delta}(\tilde{\mc{N}}_{f,a',b'}) > \eta$, and $(n,a') = 1$ for each $n \in \tilde{\mc{N}}_{f,a',b'}$. \\
By analogy with the definitions in Step (b), we define
$$
\tilde{\mc{N}}_{f,a',b'}' := \{n \in \tilde{\mc{N}}_{f,a',b'} : A(n),A(n+a') \leq N_0\}.
$$
Choosing $N_0$ larger in terms of $\delta$ if needed, we may assume that $\overline{\delta}(\tilde{\mc{N}}_{f,a',b'}') \geq \eta/2$.
Rearranging \eqref{eq:withAn} (with $a',b'$ in place of $a,b$) we get 
$$
n \left(G(n+a') - G(n)\right) = b'-a' G(n), \quad n \in \tilde{\mc{N}}_{f,a',b'}.
$$
Note that the bracketed expression, if non-zero, is of size $\asymp_{\delta} 1$ for all $n \in \tilde{\mc{N}}_{f,a',b'}'$, since each of $G(n)$ and $G(n+a')$ belongs to a set bounded at most in terms of $\delta$. Thus, if $n \in \tilde{\mc{N}}_{f,a',b'}'$ is sufficiently large then
$$
h(n+a')\frac{f(A(n+a'))}{A(n+a')} = h(n)\frac{f(A(n))}{A(n)} = \frac{b'}{a'}.
$$
Clearing denominators, we see from here that $a'|A(n)b'$. But since $A(n)|n$ and $(n,a') = 1$ for all $n \in \tilde{\mc{N}}_{f,a',b'}$, we must conclude that $a'|b'$, as claimed.
\end{proof}

\section{Proof of Theorem \ref{thm:Inhom}}
We now complete the proof of Theorem \ref{thm:Inhom}. Thus, let $f: \mb{N} \ra \mb{Z}$ be a completely multiplicative function for which
$$
\sum_{p : |f(p)| \neq 1} \frac{1}{p} = \infty,
$$
and suppose that the set
$$
\mc{N}_{f,a,b} = \{n \in \mb{N} : f(n+a) = f(n)+b\}
$$
satisfies $\delta(\mc{N}_{f,a,b}) \neq 0$ for some non-zero $a,b \in \mb{Z}$. Thus, there is a $\delta \in (0,1)$ such that $\bar{\delta}(\mc{N}_{f,a,b}) > \delta$. Assume also that the set
$$
S_f := \{ \ell \text{ prime}: \, \ell | f(\ell)\}
$$
has positive Dirichlet density,
$$
\lim_{X \ra \infty} \frac{1}{\log\log X} \sum_{\ss{p \leq X \\ p \in S_f}} \frac{1}{p} > 0.
$$
We begin with some reductions. \\
First of all, suppose $a < 0$ and let $n \in \mc{N}_{f,a,b}$ with $n > |a|$. Making the change of variables $m = n+a$ we see that
$$
\text{if } n \in \mc{N}_{f,a,b}, \, n > |a|  \text{ then } m \in \mc{N}_{f,|a|,-b}.
$$
Thus, as $a$ is fixed, upon replacing $\mc{N}_{f,a,b}$ by $\mc{N}_{f,|a|,-b}$ if necessary we may assume in the sequel that $a \geq 1$. \\
Furthermore, by Remark \ref{rem:sparseVan}, 
$$
\text{if } \sum_{p : f(p) = 0} \frac{1}{p} = \infty \text{ then } \delta(\mc{N}_{f,a,b}) = 0,
$$ 
contradicting our hypotheses. Thus, in the sequel we may assume that $f$ satisfies \emph{both} conditions
$$
\sum_{p:|f(p)| \neq 1} \frac{1}{p} = \infty, \quad \sum_{p: f(p) = 0} \frac{1}{p} < \infty.
$$
Applying Proposition \ref{prop:ArchInhom}, there exists $0 < r \ll_{\delta} 1$ such that the set of primes
$$
\mc{S}_{\delta} := \{p \in \mc{P} : |f(p)|/p^r \notin (1/2,2)\}
$$
is $O_{\delta}(1)$-sparse, i.e.,
$$
\sum_{\ss{p \leq X \\ p \in \mc{S}_{\delta}}} \frac{1}{p} \ll_{\delta} 1, \quad X \ra \infty.
$$
By Proposition \ref{prop:Assumlfl},we see that
$$
\{\ell \in \mc{P} : \ell' \nmid  f(\ell) \text{ for all } \ell' \neq \ell \}
$$
is a sparse set. Consequently, Proposition \ref{prop:veryRigid} allows us to conclude that $|f(\ell)| = \ell$ except on an $O_{\delta}(1)$-sparse set of primes $\ell$, and that when $f$ is completely multiplicative there is a divisor $d|a$ such that $f(d)|b$, and $a/d$ divides $b/f(d)$, as claimed.
\begin{rem}\label{rem:SfSizeSpec}
It is worthwhile speculating whether a condition on the size of $S_f$ is also necessary in Conjecture \ref{conj:Inhom}. Certainly, our method requires that we at least assume that
\begin{equation}\label{eq:SfnotSparse}
\sum_{p \in S_f} \frac{1}{p} = \infty,
\end{equation}
since we needed to discard sparse subsets of $S_f$ at various points in the argument. \\
Conversely, one can show that if $f$ satisfies $\delta(\mc{N}_{f,a,b}) \neq 0$ then one can produce many infinite sets $S$ of primes that satisfy $\sum_{p \in S} 1/p < \infty$, and associated multiplicative functions $g = g_{S}$ with $S_g \supseteq \{p : f(p) \neq g(p)\} = S$ such that $\delta(\mc{N}_{g,a,b}) \neq 0$. For example, if 
$
\bar{\delta}(\mc{N}_{f,a,b}) > \delta
$
and $S$ is a set of primes with\footnote{In fact, one can argue more carefully and obtain such $S$ with $\sum_{p \in S} p^{-1} \in [C,C+1)$ for any fixed $C$, by employing Lemma \ref{lem:multSet} to better sieve out those $n \in \mc{N}_{f,a,b}$ with $d|n(n+a)$ and $d$ composed of primes $p \in S$; we leave this argument to the interested reader. However, our corresponding lower bound for $\mc{N}_{g,a,b}$ also worsens in the process, at least heuristically, i.e. 
$$
\bar{\delta}(\{n \in \mc{N}_{f,a,b} : (n(n+a),S) = 1\}) \approx \bar{\delta}(\mc{N}_{f,a,b}) \prod_{p \in S}(1-2/p) \gg e^{-2C} \delta.
$$}
$\sum_{p \in S} p^{-1} < \delta/3$, then we may define $g$ to be the multiplicative function defined at prime powers by
$$
g(p^k) = \begin{cases} f(p^k) &\text{ if } p \notin S, \, k \geq 1, \\ p^k &\text{ if } p \in S, \, k \geq 1.
\end{cases}
$$
Clearly, $
g(n+a) = g(n) + b$ whenever $n \in \mc{N}_{f,a,b}$ with  $p|n(n+a) \Rightarrow p \notin S$.
Moreover, along a suitable subsequence we find
\begin{align*}
\frac{1}{\log X}\sum_{\ss{n \leq X \\ n \in \mc{N}_{f,a,b} \\ p|n(n+a) \Rightarrow p \notin S}} \frac{1}{n} \geq \frac{1}{\log X}\sum_{\ss{n \leq X \\ n \in \mc{N}_{f,a,b}}} \frac{1}{n} - \sum_{p \in S} \frac{1}{\log X} \sum_{\ss{n \leq X \\ p|n(n+a)}} \frac{1}{n}
&\geq \delta - 2\sum_{p \in S} \frac{1}{p} - o(1) \\
&> \delta/3 - o(1). 
\end{align*}
Therefore, if $f$ does not satisfy \eqref{eq:SfnotSparse} then there is not much more that $S_f$ can tell us about $f$.
\end{rem}

\section{Gaps and the proof of Corollary \ref{cor:growingPos}} \label{sec:GapCors}
\noindent Conjecture \ref{conj:Inhom}, Theorem \ref{thm:Inhom} and the main theorem of \cite{ManConsec} imply the following result 
from which Corollary \ref{cor:growingPos} follows.
\begin{cor}\label{cor:ofConj}
Suppose $f: \mb{N} \ra \mb{Z}$ is a multiplicative function for which
\begin{equation}\label{eq:gapCorHyp}
\sum_{p : |f(p)| \neq p} \frac{1}{p} = \infty, \quad \sum_{p : |f(p)| \neq 1} \frac{1}{p} = \infty \text{ and } \sum_{p : f(p) = 0} \frac{1}{p} < \infty.
\end{equation}
Assume one of the following:
\begin{enumerate}[(a)]
\item Conjecture \ref{conj:Inhom} holds, 
\item ERH holds and $\sum_{p|f(p)} p^{-1} = \infty$, or
\item $d_{\text{Dir}}(\{p \in \mc{P} : p | f(p)\}) > 0$.
\end{enumerate}
Then for all $C > 0$ we have
$$
\delta(\{n \in \mb{N} : f(n) \neq 0, \, |f(n+1)-f(n)| \leq C\}) = 0.
$$
\end{cor}
\begin{proof}
Assume for the sake of contradiction that
$$
\limsup_{X \ra \infty} \frac{1}{\log X}\sum_{\ss{n \leq X \\ f(n) \neq 0 \\ |f(n+1)-f(n)| \leq C}} \frac{1}{n} = \delta > 0.
$$
By the pigeonhole principle, there is an integer $b \in [-C,C]$ such that 
$$
\limsup_{X \ra \infty} \frac{1}{\log X}\sum_{\ss{n \leq X \\ f(n) \neq 0 \\ f(n+1)-f(n) = b}} \frac{1}{n} \geq \frac{\delta}{2\lfloor C\rfloor+1},
$$
and therefore
\begin{equation}\label{eq:posUppDens}
\delta(\{n \in \mb{N} : f(n) \neq 0, \, f(n+1) = f(n) + b\}) \neq 0.
\end{equation}
Suppose first that $b \neq 0$. Assuming either (a),(b) or (c) in conjunction with Theorem \ref{thm:Inhom}, we find that
$$
\sum_{p : |f(p)| \neq p} \frac{1}{p} < \infty,
$$
contradicting the first of the conditions in \eqref{eq:gapCorHyp}. \\
On the other hand, if $b = 0$ then by \cite[Thm. 1.1]{ManConsec} there are no multiplicative functions satisfying \eqref{eq:gapCorHyp} for which \eqref{eq:posUppDens} holds (here the condition $f(n) \neq 0$ is needed). Thus, in either case we obtain a contradiction, and the claim follows. 
\end{proof}
\noindent Next, we record how the assumption $f(n) \ra \infty$ is related to the frequency of prime power values $f(p^\nu) = 1$.
\begin{lem} \label{lem:fnGrows}
Let $f: \mb{N} \ra \mb{N}$ be a multiplicative function. Then
\begin{equation}\label{eq:eventBdd}
\limsup_{C \ra \infty} \, \overline{d}\{n \in \mb{N} : 1 \leq f(n) \leq C\} > 0
\end{equation}
if and only if 
\begin{equation}\label{eq:fpBig}
\sum_{p^\nu : f(p^\nu) \neq 1} \frac{1}{p^\nu} < \infty.
\end{equation}
\end{lem}
\begin{proof}
If \eqref{eq:fpBig} fails then Lemma \ref{lem:bddCase} (more precisely Remark \ref{rem:bddCase}) implies that 
$$
d \{n \in \mb{N} : 1 \leq f(n) \leq C\} = 0
$$
for all $C > 0$, as desired. \\
%
Conversely, suppose that \eqref{eq:fpBig} holds. For each $y \geq 1$ and $C \geq 1$ we of course have
$$
\frac{1}{y}|\{n \leq y : 1 \leq f(n) \leq C\}| \geq \frac{1}{y}|\{n \leq y : f(n) = 1\}| \geq \frac{1}{y}|\{n \leq y : \mu^2(n) = 1, \, p || n \Rightarrow f(p) = 1\}|.
$$
By a zero-dimensional sieve, it is easily verified that as $y \ra \infty$ this last expression is
$$
\frac{6}{\pi^2} \prod_{\ss{p \in \mc{P} \\ f(p) \neq 1}} \left(1-\frac{1}{p}\right) + o(1) \geq \frac{6}{\pi^2}\exp\left(-\sum_{\ss{p \in \mc{P} \\ f(p) \neq 1}} \frac{1}{p} \right) + o(1) > 0.
$$
Hence, \eqref{eq:eventBdd} holds as required.
\end{proof}

\begin{proof}[Proof of Corollary \ref{cor:growingPos}]
Let $f: \mb{N} \ra \mb{N}$ is multiplicative, such that
$$
\sum_{\ss{p \in \mc{P} \\  f(p) \neq p}} \frac{1}{p} = \infty.
$$
Further, suppose that for each $C > 0$ we have
$$
d\{n \in \mb{N} : f(n) \leq C\} = 0.
$$
By Lemma \ref{lem:fnGrows}, it follows that
$$
\sum_{p : f(p) \neq 1} \frac{1}{p} = \infty.
$$
Since $f(p) \neq 0$ for all $p$, Corollary \ref{cor:ofConj} implies that
$$
\delta\{n \in \mb{N} : |f(n+1)-f(n)| \leq C\} = 0 \text{ for all } C > 0
$$
under the assumption of ERH, assuming furthermore that $\sum_{p|f(p)} 1/p = \infty$, as claimed.
\end{proof}

\section{Proof of Proposition \ref{prop:converse}} \label{sec:converse}
To prove Proposition \ref{prop:converse} we will construct the prime sets $\mc{T}$ via the probabilistic method. \\
Let us introduce the following notation. Given a set of primes $\mc{T}$ we define the Liouville-type function $\lambda_{\mc{T}}$ 
to be the completely multiplicative function such that for each prime $p$,
$$
\lambda_{\mc{T}}(p) = \begin{cases} -1 \text{ if } p \in \mc{T} \\ +1 \text{ if } p \notin \mc{T}.\end{cases}
$$
Fix the sparse set $\mc{S}$. Let $(X_p)_{p \notin \mc{S}}$ be a sequence of i.i.d. random variables on an auxiliary probability space $(\Omega, \mc{B},\mb{P})$ with
$$
\mb{P}(X_p = 0) = \mb{P}(X_p = 1) = 1/2.
$$
We now define a random set of primes 
\begin{equation}\label{eq:TrandSet}
\mathbf{T} := \{p \notin \mc{S} : X_p = 1\}.
\end{equation}
In the next several lemmas we collect some facts about $\mathbf{T}$. \\
(In the remainder of this section we follow the convention that the trivial character is primitive.)
\begin{lem} \label{lem:probProps}
For $k \geq 1$ define the sequence $x_k := e^{e^{k^{3}}}$. Then with the above scheme, the following statements hold almost surely (i.e. on a set $A \subseteq \Omega$ with $\mb{P}(A) = 1$):
\begin{enumerate}[(a)]
\item for each $k \geq k_0(\omega)$, 
$$
\sum_{\ss{p \leq x_k \\ p \in \mathbf{T}}} \frac{1}{p} \geq \frac{1}{4}\log\log x_k,
$$
\item for every real primitive Dirichlet character $\chi \pmod{q}$ with $1 < q \leq k$,
$$
\left|\sum_{\ss{p \leq x_k \\ p \in \mathbf{T}}} \frac{\chi(p)}{p}\right| \leq \frac{1}{10} \log\log x_k, \quad k \geq k_0(\omega).
$$
\end{enumerate}
\end{lem}
\begin{proof}
By linearity of expectation we have
$$
\mb{E}\left(\sum_{\ss{p \leq x_k \\ p \in \mathbf{T}}} \frac{1}{p}\right) = \mb{E}\left(\sum_{\ss{p \leq x_k \\ p \notin \mc{S}}} \frac{X_p}{p}\right) = \sum_{p \leq x_k} \frac{\mb{E}(X_p)}{p} + O\left(\sum_{\ss{p \leq x_k \\ p \in \mc{S}}} \frac{1}{p}\right) = \frac{1}{2}\log\log x_k + O(1).
$$
Therefore, by independence the variance satisfies
\begin{align*}
\mb{E}\left(\sum_{\ss{p \leq x_k \\ p \in \mathbf{T}}} \frac{1}{p} - \frac{1}{2} \log\log x_k\right)^2 &= \sum_{\ss{p,q \leq x_k \\ p,q \notin \mc{S}}} \frac{\mb{E}(X_pX_q)}{pq} - \frac{1}{4} (\log\log x_k)^2 + O(\log\log x_k) \\
&= \left(\sum_{\ss{p \leq x_k \\ p \notin \mc{S}}} \frac{\mb{E}(X_p)}{p}\right)^2 + \frac{1}{2} \sum_{p \leq x_k} \frac{1}{p^2} - \frac{1}{4} \log\log x_k^2 + O(\log\log x_k) \\
&= O(\log\log x_k).
\end{align*}
By Chebyshev's inequality, we deduce that
$$
\mb{P}\left(\sum_{\ss{p \leq x_k \\ p \in \mathbf{T}}} \frac{1}{p} \leq 
\frac{1}{4}\log\log x_k\right) 
\leq 16\frac{\mb{E}\left(\sum_{\ss{p \leq x_k \\ p \in \mathbf{T}}} \frac{1}{p} - \frac{1}{2} \log\log x_k\right)^2}{(\log\log x_k)^2} \ll \frac{1}{\log\log x_k} \ll \frac{1}{k^3}.
$$
Next, 
suppose $1 < q \leq k \leq \log\log x_k$ and let $\chi$ be a real primitive character modulo $q$. Then
\begin{align}
\mb{E}\left(\left|\sum_{\ss{p \leq x_k \\ p \in \mathbf{T}}} \frac{\chi(p)}{p}\right|^2\right) &= \frac{1}{4} \sum_{\ss{p,q \leq x_k \\ p,q \notin \mc{S} }} \frac{\chi(pq)}{pq} + O\left(\sum_{p \leq x_k} \frac{1}{p^2}\right) = \left(\frac{1}{2}\sum_{p \leq x_k} \frac{\chi(p)}{p} + O(1) \right)^2 + O(1). \nonumber
\end{align}
By the Siegel-Walfisz theorem and partial summation (see e.g. \cite[Cor. 5.29]{IK}), 
\begin{equation}
\sum_{p \leq x_k} \frac{\chi(p)}{p} \ll \sum_{p \leq e^k} \frac{1}{p} + \int_{e^k}^{x_k} \left|\sum_{p \leq v} \chi(p)\right|\frac{dv}{v^2} \ll \log k + k^{1/2}k^{-10} \ll \log\log\log x_k. \label{eq:SWApp}
\end{equation}
Combining these two estimates and using Chebyshev's inequality, we find
$$
\mb{P}\left(\left|\sum_{\ss{p \leq x_k \\ p \in \mathbf{T}}} \frac{\chi(p)}{p}\right| \geq \frac{1}{10}\log\log x_k\right) \ll \frac{(\log\log\log x_k)^2}{(\log\log x_k)^2} \ll \frac{(\log k)^2}{k^4} \ll \frac{1}{k^3}.
$$
Since there are $\ll k$ real primitive characters with conductor $\leq k$, by the union bound we deduce that the probability that either (a) fails or that (b) fails for some $\chi \pmod{q}$, $1 < q \leq k$ is
$$
\ll \frac{1}{k^3} + k \cdot \frac{1}{k^3} \ll \frac{1}{k^2}.
$$
By the Borel-Cantelli lemma (see \cite[VIII.3]{Fel}), we deduce that, on a set of $\omega \in \Omega$ with probability $1$, both
$$
\sum_{\ss{p \leq x_k \\ p\in \mathbf{T}(\omega)}} \frac{1}{p} \geq \frac{1}{4}\log\log x_k + O(1) 
$$
and also
$$
\max_{1 < q \leq k} \left|\sum_{\ss{p \leq x_k \\ p \in \mathbf{T}(\omega)}} \frac{\chi(p)}{p}\right| \leq \frac{1}{10}\log\log x_k
$$
hold for all $k \geq k_0(\omega)$, as claimed.
\end{proof}
\begin{lem} \label{lem:nonPretrandom}
For every real primitive character $\chi$,
$$
\lim_{x \ra \infty} \mb{D}(\lambda_{\mathbf{T}},\chi; x) = \infty
$$
almost surely.
\end{lem}
\begin{proof}
Let $\chi$ be a real primitive character with modulus $q \geq 1$. Given $k \geq q$ we have
\begin{equation}\label{eq:distLambdaT}
\mb{D}(\lambda_{\mathbf{T}}, \chi;x_k)^2 = \sum_{\ss{p \leq x_k \\ p \in \mathbf{T}}} \frac{1+\chi(p)}{p} + \sum_{\ss{p \leq x_k \\ p \notin \mathbf{T}}} \frac{1-\chi(p)}{p} = \sum_{p \leq x_k} \frac{1-\chi(p)}{p} + 2\sum_{\ss{p \leq x_k \\ p \in \mathbf{T}}} \frac{\chi(p)}{p}.
\end{equation}
If $q = 1$ then by Lemma \ref{lem:probProps} the right-hand side of \eqref{eq:distLambdaT} is simply
$$
2\sum_{\ss{p\leq x_k \\ p \in \mathbf{T}}} \frac{1}{p} \geq \frac{1}{2}\log\log x_k
$$
for all $k \geq k_0(\omega)$ almost surely.\\
If $1 < q \leq k$ then combining Lemma \ref{lem:probProps} with the Siegel-Walfisz theorem (as in \eqref{eq:SWApp}), the right-hand side of \eqref{eq:distLambdaT} is bounded from below as
\begin{align*}
\geq \log\log x_k- \left|\sum_{p \leq x_k} \frac{\chi(p)}{p}\right| - 2\left|\sum_{\ss{p \leq x_k \\ p \in \mathbf{T}}} \frac{\chi(p)}{p}\right| + O(1) &\geq \log\log x_k + O(\log\log \log x_k) - \frac{1}{5}\log\log x_k \\ &\geq \frac{1}{2}\log\log x_k
\end{align*}
for $k \geq k_0(\omega)$ almost surely. Thus, we obtain
$$
\inf_{\ss{\chi \text{ real primitive} \\ q \leq k}} \mb{D}(\lambda_{\mathbf{T}},\chi;x_k)^2 \geq \frac{1}{2}\log\log x_k, \quad k \geq k_0(\omega)
$$
almost surely.\\
Therefore, for any real primitive character $\chi$ modulo $q_0$, any $k \geq \max\{k_0(\omega),q_0\}$ and any $x \geq x_k$,
$$
\mb{D}(\lambda_{\mathbf{T}},\chi; x)^2 \geq \inf_{q \leq k} \mb{D}(\lambda_{\mathbf{T}},\chi_q;x_k)^2 \geq \frac{1}{2}\log\log x_k + O(1)
$$
almost surely. Taking $k \ra \infty$ forces $x \ra \infty$, and the claim follows.
\end{proof}
\begin{lem} \label{lem:unctble}
If $\mc{C}$ is any countable collection of subsets of primes,
$$
\mb{P}(\mathbf{T} \in \mc{C}) = 0.
$$
In particular, for any set $A$ with $\mb{P}(A) > 0$, the collection
$
\{\mathbf{T}(\omega) : \omega \in A\}
$
is uncountable.
\end{lem}
\begin{proof}
For any fixed set of primes $\mc{T}$ and any $N \geq 1$ the probability
\begin{equation}\label{eq:cylinderbound}
\mb{P}(\mathbf{T} = \mc{T}) \leq \mb{P}(\bigcap_{\ss{p \leq N \\ p \notin \mc{S}}} \{X_p = 1_{p \in T}\}) \leq 2^{-\pi_{\mc{S}^c}(N)},
\end{equation}
where $\pi_{\mc{S}^c}(N) := \pi(N) - |\mc{S} \cap [2,N]|$. Since $\mc{S}$ is sparse we have $\pi_{\mc{S}^c}(N_k) \geq \tfrac{2}{3}\pi(N_k)$ on an infinite increasing sequence $(N_k)_k$. Taking $N \ra \infty$ along $(N_k)_k$, we deduce that $\mb{P}(\mathbf{T} = \mc{T}) = 0$. Thus, by countable additivity of the measure $\mb{P}$ we find
$$
\mb{P}(\mathbf{T} \in \mc{C}) = \sum_{\mc{T} \in \mc{C}} \mb{P}(\mathbf{T} = \mc{T}) = 0.
$$
This establishes the first claim.\\
For the second, observe that if $\{\mathbf{T}(\omega) : \omega \in A\}$ were countable then there would exist a collection $\mc{C} = (\mc{T}_j)_j$ such that 
$$
\omega \in A \Rightarrow \mathbf{T}(\omega) \in \mc{C}.
$$
By the previous claim,
$$
0 < \mb{P}(A) \leq \mb{P}(\mathbf{T} \in \mc{C}) = 0,
$$
which is a contradiction.
\end{proof}
\begin{proof}[Proof of Proposition \ref{prop:converse}]
Set $\mc{S}_d := \mc{S} \bk \{p : p|d\}$. Let $\mc{T}$ be a realisation of the random set $\mathbf{T}$ defined in \eqref{eq:TrandSet}. By Lemma \ref{lem:nonPretrandom}, with probability $1$ we have
\begin{equation}\label{eq:nonpretCheck}
\lim_{X \ra \infty} \mb{D}(\lambda_{\mc{T}}, \chi;X) = \infty
\end{equation}
for each fixed real primitive character $\chi$, and Lemma \ref{lem:unctble} shows that the collection of $\mc{T}$ for which this holds is uncountable. For the remainder of the proof we fix some realisation $\mc{T}$ satisfying \eqref{eq:nonpretCheck}. \\
Define now the set 
$$
\mc{N} = \mc{N}_{\mc{T}} := \{dpm : \, p \in S_d, \, p|m \Rightarrow p \notin \mc{S} \text{ and } \lambda_{\mc{T}}(m) = +1\}.
$$
Let $\tilde{d} = \prod_{\ss{p || d}} p$. Thus, if $d = 1$ or $d$ is squarefull then $\tilde{d} = 1$; otherwise, $\tilde{d} > 1$. Define also $k := db/a \in \mb{Z}$, $k \neq 0$.
We define the completely multiplicative function $f = f_{\mc{T}}$ at primes via
$$
f(p) = \begin{cases} \lambda_{\mc{T}}(p) p &\text{ if $p\notin \mc{S}$}  \\ P^+(|k|) &\text{ if $p = P^+(\tilde{d})$} \\ p k/f(P^+(\tilde{d})) &\text{ if $p \in \mc{S}_d$} \\ 1  &\text{ if $p|d$, $p \neq P^+(\tilde{d})$.}
\end{cases}
$$
Since $k \neq 1$ by assumption, $f$ clearly satisfies items (i) and (ii). Note furthermore that $f(d) = f(P^+(\tilde{d}))$ divides $k$, and is thus a divisor of $b$ as in item (iii). It thus remains to show that $f$ so-defined satisfies item (iv). \\
Consider the set
$$
\mc{N}' := \{n \in \mb{N} : n,n+a \in \mc{N}\}.
$$
If $n \in \mc{N}'$ then we can find primes $p_1,p_2 \in \mc{S}_d$ and $m_1,m_2$ not divisible by any prime $p \in \mc{S}$, such that
$$
n = m_1p_1 d, \quad n+a = m_2p_2d = d(m_1p_1 + a/d),
$$
and $\lambda_{\mc{T}}(m_1) = \lambda_{\mc{T}}(m_2) = +1$.
By construction we have $f(d) = f(P^+(\tilde{d}))$. Thus,
\begin{align*}
f(n+a) &= f(d)f(m_2)f(p_2) = f(P^+(\tilde{d})) \lambda_{\mc{T}}(m_2) m_2 (k/f(P^+(\tilde{d}))) p_2 = k m_2p_2, \\
f(n)+b &= f(d)f(m_1)f(p_1) + k_1k_2a/d = k(\lambda_{\mc{T}}(m_1)m_1p_1 + a/d) = k(m_1p_1 + a/d).
\end{align*}
Since $m_2p_2 = m_1p_1 + a/d$, we deduce that $f(n+a) = f(n) + b$, and so
$$
\mc{N}' \subseteq \mc{N}_{f,a,b}.
$$
To complete the proof we will show that $\bar{\delta}(\mc{N}') > 0$. \\
Since $|\mc{S}_d| \geq 2$ and the sets
$$
\mc{N}_q := \{dqm : p | m \Rightarrow p \notin \mc{S} \text{ and } \lambda_{\mc{T}}(m) = +1\}, \quad q \in \mc{S}_d
$$
are mutually disjoint, it suffices to show that there are $p_1,p_2 \in \mc{S}_d$ with $p_1 \neq p_2$ and $(p_1p_2,a) = 1$ such that
$$
\{n \in \mc{N}_{p_1} : n+a \in \mc{N}_{p_2}\} \subseteq \mc{N}'
$$
has positive upper logarithmic density. To this end, we will estimate
$$
\frac{1}{\log x}\sum_{\ss{n \leq x \\ n \in \mc{N}_{p_1}, \, n+a \in \mc{N}_{p_2}}} \frac{1}{n} = \frac{1}{d \log x} \sum_{\ss{n' \leq x/d \\ n' = m_1p_1, \, n'+a/d = m_2p_2 \\ (m_1m_2,\mc{S}) = 1 \\ \lambda_{\mc{T}}(m_1) = \lambda_{\mc{T}}(m_2) = +1}} \frac{1}{n'}. 
$$
By the Chinese remainder theorem, we can find $0 \leq u < p_2$ so that
$$
n' \equiv 0 \pmod{p_1}, \, n'+a/d \equiv 0 \pmod{p_2} \text{ if and only if } n' \equiv up_1 \pmod{p_1p_2}.
$$
Define the multiplicative function $g(n) := 1_{(n,\mc{S}) = 1}$.
Using the changes of variables 
$$
n' = mp_1p_2 + up_1, \quad v := (up_1+a/d)/p_2 \in \mb{Z},
$$ 
and the identity $1_{\lambda_{\mc{T}}(m) = +1} = (1+\lambda_{\mc{T}}(m))/2$, we can recast the above expression as
\begin{align*}
&\frac{1}{4d\log x} \sum_{m \leq x/dp_1p_2} \frac{g(mp_2 + u)g(mp_1 + v)}{m}(1 + \lambda_{\mc{T}}(mp_2+u))(1 + \lambda_{\mc{T}}(mp_1+v)) \\
&= \sum_{\eta_1,\eta_2 \in \{0,1\}} \frac{1}{4d \log x} \sum_{m \leq x/dp_1p_2} \frac{g\lambda_{\mc{T}}^{\eta_1}(mp_2 + u)g\lambda_{\mc{T}}^{\eta_2}(mp_1 + v)}{m}.
\end{align*}
Since $g(p) = 1$ for all $p$ outside the sparse set $\mc{S}$, the pretentious triangle inequality \eqref{eq:usualTri} implies that for any fixed, real primitive character $\chi$,
$$
\mb{D}(g \lambda_{\mc{T}},\chi;x) \geq \mb{D}(\lambda_{\mc{T}},\chi;x) - \mb{D}(g,1;x) = \mb{D}(\lambda_{\mc{T}},\chi;x) - \left(\sum_{\ss{p \leq x \\ p \in \mc{S}}} \frac{1}{p}\right)^{1/2} = \mb{D}(\lambda_{\mc{T}},\chi;x) - O(1).
$$
Combining \eqref{eq:nonpretCheck} with Lemma \ref{lem:realNonPret}, we deduce therefore that
$$
\lim_{x \ra \infty}
\inf_{|t| \leq x} \mb{D}(g\lambda_{\mc{T}}, \chi;x) = \infty 
$$
for every fixed Dirichlet character $\chi$ (real or complex). As $d,p_1,p_2,u$ and $v$ are all fixed, Theorem \ref{thm:Tao} yields
$$
\frac{1}{4d \log x} \sum_{m \leq x/dp_1p_2} \frac{g\lambda_{\mc{T}}^{\eta_1}(mp_2 + u)g\lambda_{\mc{T}}^{\eta_2}(mp_1 + v)}{m} = o(1)
$$
for each $(\eta_1,\eta_2) \neq (0,0)$. Thus,
\begin{equation}\label{eq:afterTao}
\frac{1}{\log x}\sum_{\ss{n \leq x \\ n \in \mc{N}_{p_1}, \, n+a \in \mc{N}_{p_2}}} \frac{1}{n} = \frac{1}{4d \log x} \sum_{\ss{m \leq x \\ (mp_1+u,\mc{S}) = (mp_2+v,\mc{S}) = 1}} \frac{1}{m} + o(1).
\end{equation}
By partial summation, it now suffices to show that the set
$$
\{m \in \mb{N} : (mp_2+u,\mc{S}) = (mp_1+v,\mc{S}) = 1\}
$$
has positive density. 
By a zero-dimensional sieve this density is
$$
\frac{1}{p_1p_2}\left(1-\frac{1}{p_1}\right)\left(1-\frac{1}{p_2}\right) \prod_{\ss{p \in \mc{S} \\ p \nmid p_1p_2}} \left(1-\frac{2}{p}\right),
$$
which is $>0$ as $2 \notin \mc{S}$. \\
When $\mc{T} = \emptyset$ the condition $\lambda_{\mc{T}}(m) = +1$ in the definition of $\mc{N}$ is trivial, and \eqref{eq:afterTao} holds with the prefactor $1/4$ replaced by $1$. The same conclusion follows.
\end{proof}
\begin{rem}
The above proof may easily be modified to allow $2 \in \mc{S}$, though the analysis differs according to whether $2$ divides $d$ and/or $a$. We leave this to the interested reader. 
\end{rem}

\bibliographystyle{plain}
\bibliography{InhomBib.bib}

\end{document}